\documentclass[a4paper,11pt]{article}
\usepackage[utf8]{inputenc}
\usepackage{amssymb,mathrsfs,amsmath,amsthm,bbm}
\usepackage{verbatim}
\usepackage{indentfirst}
\usepackage{geometry}
\usepackage{tikz-cd, tikz}
\usepackage{xcolor}
\usepackage[titletoc]{appendix}

\usepackage{hyperref} 
\hypersetup{colorlinks=true,allcolors=blue}
\usepackage{hypcap}

\usepackage{enumerate}

\usepackage{graphicx}
\graphicspath{ {./images/} }

\newtheorem{thm}{Theorem}[section]
\newtheorem{cor}[thm]{Corollary}
\newtheorem{prop}[thm]{Proposition}
\newtheorem{lem}[thm]{Lemma}

\theoremstyle{definition}
\newtheorem{defn}[thm]{Definition}

\newtheorem{innercustomgeneric}{\customgenericname}
\providecommand{\customgenericname}{}
\newcommand{\newcustomtheorem}[2]{%
  \newenvironment{#1}[1]
  {%
   \renewcommand\customgenericname{#2}%
   \renewcommand\theinnercustomgeneric{##1}%
   \innercustomgeneric
  }
  {\endinnercustomgeneric}
}

\newcustomtheorem{customthm}{Theorem}
\newcustomtheorem{customprop}{Proposition}

\theoremstyle{remark}
\newtheorem{rem}[thm]{Remark}

\newtheorem*{rem*}{Remark}

\newcommand{\bms}{m^{\operatorname{BMS}}}
\newcommand{\Z}{\mathbb Z}
\newcommand{\R}{\mathbb R}
\newcommand{\C}{\mathbb C}
\newcommand{\D}{\mathrm D}
\newcommand{\N}{\mathbb N}

\renewcommand{\S}{\mathbb S}
\renewcommand{\H}{\mathbb H}

\newcommand{\calH}{\mathcal H}
\newcommand{\calP}{\mathcal P}
\newcommand{\calG}{\mathcal G}

\newcommand{\calB}{\mathcal B}
\newcommand{\calC}{\mathcal C}

\newcommand{\LL}{\tilde{L}}
\newcommand{\h}{h}
\newcommand{\T}{\operatorname{T}}
\newcommand{\SO}{\operatorname{SO}(d+1,1)}
\newcommand{\SOm}{\operatorname{SO}(m,1)}

\newcommand{\uLip}{\lVert u\rVert_{\text{Lip}}}
\newcommand{\ulip}{\lvert u\rvert_{\text{Lip}}}
\newcommand{\ub}{\lVert u\rVert_b}
\newcommand{\height}{h}

\newcommand{\dd}{\mathrm{d}}
\renewcommand{\r}{r}

\newcommand{\eps}{\epsilon}
\newcommand{\three}{3}
\newcommand{\four}{4}
\newcommand{\five}{5}
\newcommand{\six}{6}
\newcommand{\seven}{7}
\newcommand{\eight}{8}
\newcommand{\nine}{9}
\newcommand{\ten}{10}
\newcommand{\ele}{11}
\newcommand{\twl}{12}
\newcommand{\thi}{13}
\newcommand{\fou}{14}
\newcommand{\fif}{15}
\newcommand{\sixt}{16}
\newcommand{\sev}{17}
\newcommand{\eig}{18}
\newcommand{\nin}{19}
\newcommand{\twi}{20}

\newcommand{\RN}[1]{%
  \textup{\uppercase\expandafter{\romannumeral#1}}%
}

\makeatletter
\let\c@equation\c@thm
\makeatother
\numberwithin{equation}{section}

\begin{document}
\title{Exponential mixing of geodesic flows for geometrically finite hyperbolic manifolds with cusps}
\author{Jialun Li and Wenyu Pan}

\date{}
\maketitle
\begin{abstract}
Let $\Gamma$ be a geometrically finite discrete subgroup in $\operatorname{SO}(d+1,1)^{\circ}$ with parabolic elements. We establish exponential mixing of the geodesic flow on the unit tangent bundle $\T^1(\Gamma\backslash \mathbb{H}^{d+1})$ with respect to the Bowen-Margulis-Sullivan measure, which is the unique probability measure on $\T	^1(\Gamma\backslash \mathbb{H}^{d+1})$ with maximal entropy. As an application, we obtain a resonance-free region for the resolvent of the Laplacian on $\Gamma\backslash \mathbb{H}^{d+1}$. Our approach is to construct a coding for the geodesic flow and then prove a Dolgopyat-type spectral estimate for the corresponding transfer operator. 
\end{abstract}

\section{Introduction}
\subsection{Exponential mixing of the geodesic flow}
Let $\mathbb{H}^{d+1}$ be the hyperbolic $(d+1)$-space. Let $G=\operatorname{SO}(d+1,1)^{\circ}$, which is the group of orientation preserving isometries of $\mathbb{H}^{d+1}$. Let 
$\Gamma<G$ be a non-elementary, torsion-free, geometrically finite discrete subgroup with parabolic elements. Denote by $\delta$ the critical exponent of $\Gamma$, which is defined as the abscissa of convergence of the Poincar\'e series $\sum_{\gamma\in \Gamma}e^{-sd(o,\gamma o)}$. Set $M=\Gamma\backslash \H^{d+1}$, so $M$ contains cusps. We consider the geodesic flow $(\mathcal{G}_t)_{t\in \mathbb{R}}$ acting on the unit tangent bundle $\T^1(M)$ over $M$. The invariant measure for the flow we will work with is the Bowen-Margulis-Sullivan measure $m^{\operatorname{BMS}}$, which is supported on the non-wandering set of the geodesic flow and is known to be the unique probability measure with maximal entropy $\delta$ \cite{OtPe}. 
%

Our main result is establishing exponential mixing of the geodesic flow.
\begin{thm}
\label{main thm}
The geodesic flow is exponentially mixing with respect to $\bms$: there exists $\eta>0$ such that for any functions $\phi, \psi\in C^1(\T^1(M))$ and any $t>0$, we have
\begin{equation*}
\int_{\T^1(M)} \phi\cdot\psi\circ\calG_t\ \dd\bms =\bms (\phi) \bms (\psi)+O(\lVert \phi \rVert_{C^1} \lVert \psi\rVert_{C^1}e^{-\eta t}),
\end{equation*}
 where $\|\cdot\|_{C^1}$ is the $C^1$-norm with respect to the Riemannian metric on $\T^1(M)$.
\end{thm}

	For a geometrically finite discrete subgroup $\Gamma$, Sullivan~\cite{Sul} proved the ergodicity of the geodesic flow with respect to $\bms$ and Rudolph~\cite{Rud} proved that the geodesic flow is mixing with respect to $\bms$.
	When $\delta>d/2$, Theorem \ref{main thm} was proved by Mohammadi-Oh~\cite{MoOh} and Edwards-Oh \cite{EO} using the representation theory of $L^2(M)$ and the spectral gap of Laplace operator~\cite{LP}.
	When $\Gamma$ is convex cocompact, i.e., geometrically finite without parabolic elements, Theorem~\ref{main thm} and its corollaries were proved by Naud \cite{Nau}, Stoyanov \cite{Sto} and Sarkar-Winter \cite{SaWi} building on the work of Dolgopyat \cite{Dol}. Therefore, the main contribution of our work lies in the groups with small critical exponent and with parabolic elements, completing the story of exponential mixing of the geodesic flow on a geometrically finite hyperbolic manifold.
	
	Using Roblin's transverse intersection argument \cite{Rob, OS, OhWi}, we obtain the decay of matrix coefficients (Theorem \ref{thm:matrix}) from Theorem \ref{main thm}. Theorem \ref{main thm} and \ref{thm:matrix} are known to have many immediate applications in number theory and geometry. To name a few, see \cite{MMO} for counting closed geodesics, \cite{KeOh} for shrinking target problems and \cite{BPP} for some general counting results.
	



\subsection{Resonance-free region}

Recall $M=\Gamma\backslash\H^{d+1}$. Consider the Laplace operator $\Delta_M$ on $M$. Lax and Phillips completely described its spectrum on $L^2(M)$ when $M$ has infinite volume \cite{LP}. The half line $[d^2/4,\infty)$ is the continuous spectrum and it contains no embedded eigenvalues. The rest of the spectrum (point spectrum) is finite and starting at $\delta (d/2-\delta)$ if $\delta>d/2$ and is empty if $\delta\leq d/2$. Let $S$ be the set of eigenvalues of $\Delta_M$. The resolvent of the Laplacian $$R_M(s)=(\Delta_M-s(d-s))^{-1}:L^2(M)\to L^2(M)$$ is well-defined and analytic on$\{\Re s>d/2,\ s(d-s)\notin S\}$. 
Guillarmou and Mazzeo showed that $R_M(s)$ has a meromorphic continuation to the whole complex plane as an operator from $C^\infty_c(M)$ to $C^\infty(M)$ with poles of finite rank \cite{GM}. These poles are called resonances. Patterson showed that on the line $\Re s= \delta$, the point $s=\delta$ is the unique pole of $\Gamma(s-\frac{d}{2}+1)R_M(s)$ and it is a simple pole \cite{Pat1}. We use Theorem \ref{main thm} to further obtain a resonance-free region.
\begin{thm}\label{cor:resonance}
	There exists $\eta>0$ such that on the half-plane $ \Re s>\delta-\eta$, $s=\delta$ is the only resonance for the resolvent $R_M(s)$ if $\delta\notin d/2-\N_{\geq 1}$; otherwise, $R_M(s)$ is analytic on $\Re s>\delta-\eta$.
\end{thm}
In the convex cocompact case, a resonance-free region of the resolvent is closely related to a zero free region of the Selberg zeta function. But in the geometrically finite case, such relation is not well understood except for the surface case.

\subsection{On the proof of the main theorem}
The proof of Theorem \ref{main thm} can be reduced to the case when $\Gamma$ is Zariski dense and then the proof falls into two parts:
 we code the geodesic flow and prove a Dolgopyat-type spectral estimate for the corresponding transfer operator. Ultimately, the obstructions to applying Dolgopyat’s original argument in our context are purely technical, but to overcome these obstructions in any context is the heart of the matter.

To prove exponential mixing using the symbolic-dynamic approach of Dolgopyat, one approach is to construct a section to the flow. In sum, one seeks a $2d$ submanifold $S$ in $\T^1(M)$ transversal to the geodesic flow which is a Poincar\'{e} section, on which the return map can be tightly organized. The challenge lies in that it is required to find a return map $R$ defined on a full measure subset $S_0$ of $S$, such that the map $F(v):=\calG_{R(v)}(v)$, $v\in S_0$, on $S_0$ is hyperbolic and can be modelled on a full shift of countable many symbols. 

We overcome this difficulty by connecting the return map on $S_0$ to an expanding map on the boundary $\partial\H^{d+1}$.
The precise description of the expanding map on the boundary is as follows. We consider the upper-half space model for $\mathbb{H}^{d+1}$ and without loss of generality, we may assume that $\infty$ is a parabolic fixed point of $\Gamma$. Let $\operatorname{Stab}_{\infty}(\Gamma)$ be the group of stabilizers of $\infty$ in $\Gamma$ and $\Gamma_{\infty}$ be a maximal normal abelian subgroup in $\operatorname{Stab}_{\infty}(\Gamma)$. Set $\Delta_{0}:=\Delta_{\infty}$ to be a fundamental domain of $\Gamma_{\infty}$ in $\partial \mathbb{H}^{d+1}\backslash \{\infty\}$ (see Section \ref{sec:cusps} for details). Denote by $\Lambda_{\Gamma}$ the limit set of $\Gamma$ and $\mu$ the Patterson-Sullivan measure, which is a finite measure supported on $\Lambda_{\Gamma}$.

\begin{customprop}{\ref{prop:coding}}
There are constants $C_1>0$, $\lambda,\,\epsilon_0\in (0,1)$, a countable collection of disjoint, open subsets $\{\Delta_j\}_{j\in\N}$ in $\Delta_0$ and an expanding map $T:\sqcup_j\Delta_j\to \Delta_0$ such that: 
		\begin{enumerate}
		\item $\sum_{j}\mu(\Delta_j)=\mu(\Delta_0)$.
		\item For each $j$, there is an element $\gamma_j\in \Gamma$ such that $\Delta_j=\gamma_j \Delta_0$ and $T|_{\Delta_j}=\gamma_j^{-1}$. 
		\item Each $\gamma_j$ is a uniform contraction: $|\gamma_j'(x)|\leq \lambda$ for all $x\in \Delta_0$.
		\item For each $\gamma_j$, $|\D(\log |\gamma_j'|)(x)|<C_1$ for all $x\in \Delta_0$, where $\D(\log |\gamma_j'|)(x)$ is the differential of the map $z\mapsto \log |\gamma_j'(z)|$ at $x$.
		\item Let $R$ be the function on $\sqcup_j \Delta_j$ given by $R(x)=\log|\D T(x)|$. Then
		$\int e^{\epsilon_o R}\dd\mu<\infty.$
	\end{enumerate}
\end{customprop}
The last property is known as the exponential tail property. Moreover, we show that the coding satisfies the \textbf{uniform nonintegrable condition (UNI)} (Lemma~\ref{lem:uni}). We use Proposition \ref{lem:loc} as a bridge to connect the geodesic flow $(\mathcal{G}_t)_{t\in \mathbb{R}}$ on $\T^1(M)$ and the expanding map $T$ on $\Delta_0$. We show that the geodesic flow is a factor of a hyperbolic skew product flow constructed using $T$. 

The construction of the coding starts with the following observation: locally, in a neighborhood of a parabolic fixed point, we can use the structure of the parabolic fixed points to find a ``flower" centered at this parabolic fixed point, and the ``flower" can be partition into a countable union of open sets of the form $\gamma(\Delta_0)$ for some $\gamma\in \Gamma$. Once we have the algorithm to find the partition for local regions, we still face the question of how to patch these flowers together. We introduce an inductive algorithm to find pairwise disjoint flowers.


But the bulk of the work lies in proving the exponential tail property. We show that this follows from Proposition \ref{keylemma} which says that the measure of the set that has not been partitioned at time $n$ decays exponentially. At the time $n$, the remaining part is a sheet with many holes, consisting of ``flowers"; while the Patterson-Sullivan measure is a measure supported on the fractal limit set, and we have limited knowledge of the regularity of this measure. It is interesting to figure out how to use minimal tools to get the required estimate.

When the non-wandering set of the geodesic flow is compact, the coding is well-studied and we have, for example, the Bowen-Series' coding \cite{BKS}, Bowen's coding \cite{Bowen} and Ratner's coding \cite{Rat}. When manifolds contain cusps, only some partial knowledge is available. Dal'bo-Peign\'{e} \cite{DaPe, DaPe1} and Babillot-Peign\'{e} \cite{BaPe} provided the coding for generalized Schottky groups. Stadlbauer \cite{Sta} and Ledrappier-Sarig \cite{LeSa} provided the coding for non-uniform lattices in $\operatorname{SO}(2,1)^{\circ}$. They made use of the fact that such a discrete subgroup is a free group and has a nice fundamental domain in $\mathbb{H}^2$. Our coding works for general geometrically finite discrete subgroups with parabolic elements and is partly inspired by the works of Lai-Sang Young \cite{You} and Burns-Masur-Matheus-Wilkinson \cite{BMMW}.


In a forthcoming joint work with Sarkar \cite{LPS}, we establish the exponential mixing for frame flows on geometrically finite hyperbolic manifolds with cusps. 
We prove this by using the coding of the geodesic flow constructed in this paper and then performing a frame flow version of Dolgopyat's method. 
The crucial cancellations of the summands of the transfer operators twisted by holonomy are obtained from the local non-integrability condition and the non-concentration property of Sarkar-Winter \cite{SaWi}. But the challenge in the presence of cusps is that the latter holds only on a certain \emph{good} subset. This is resolved by a large deviation property for symbolic recurrence to the good subset, which is inspired by the work of Tsujii-Zhang \cite{TsZh}. It is proved by studying the combinatorics of cusp excursions and showing an effective renewal theorem, as in the work of Li \cite{Li}, which uses the spectral gap of the transfer operator for the geodesic flow in Proposition \ref{L2contracting}.

In \cite{GLZ}, using the coding of Schottky groups, Guillop\'e-Lin-Zworski were able to study the Selberg zeta function through a dynamical zeta function. They gave a simple proof of the analytic continuation and a growth estimate of the Selberg zeta function. Hopefully, the coding constructed in our work will be helpful in the study of the Selberg zeta function for higher dimensional geometrically finite manifold.

Other applications include 
the Fourier decay of the Patterson-Sullivan measure. In \cite{BD}, Bourgain-Dyatlov proved Fourier decay of Patterson-Sullivan measures for convex cocompact Fuchsian groups. The first step of their proof is to use the coding of the limit set to construct an appropriate transfer operator. With our coding available, it is very likely to generalize the Fourier decay to geometrically finite discrete subgroups with parabolic elements.

Our proof of obtaining a Dolgopyat-type spectral estimate is influenced by the one in~\cite{ArMe, AGY, BaVa, Dol, Nau, Sto}. 
The key of Dolgopyat's approach is to estimate the decay of certain oscillatory integrals against the \textit{fractal} Patterson-Sullivan measure:
for function $f$ of the form $\sum_{j\in J}\exp(ib\tau_j(x))$, where $b$ is a real number, $\tau_j\in C^2(\Delta_{0})$ and $J$ is some index set, we have
$ |\int f\,d\mu|$
is bounded by some negative power of $|b|$.
We successfully attain this estimate by combining dynamics and the regularity properties of the Patterson-Sullivan measure, which we think are the essential ingredients to gain the decay.

Another possible argument is to analyze each $\int \exp(ib\tau_j)\ d\mu$ and show the decay. Such a result is known as the Fourier decay of the Patterson-Sullivan measure. This is especially challenging when the critical exponent $\delta$ is small. In \cite{JS}, Jordan-Sahlsten proved the Fourier decay of some fractal measures. Their idea is to approximate the fractal measure by the Lebesgue measure and use the Fourier decay of the Lebesgue measure, which is well-studied. But this approximation is sensitive to the Hausdorff dimension of the fractal sets. A similar idea also appears in a preprint by Kahlil \cite{Kha}. It is unclear whether their approach provides an alternative way to establish the Fourier decay of fractal measures without dimension restriction. 

 
There are works trying to use anisotropic Banach spaces to prove exponential mixing. The key is to show there exists $\eps>0$ such that the strip $\{-\eps<\Re s<0 \}$ is free of \textit{Pollicott-Ruelle resonances}. For a geometrically finite discrete subgroup with the critical exponent $\delta<d/2$, it might happen that there are resonances with large imaginary parts and real parts close to zero. Recently, there is a work in progress of Gou\"ezel-Tapie-Schapira on the Pollicott-Ruelle resonances for SPR manifolds, which include geometrically finite manifolds. They show the resonances are discrete, but it is not clear whether one can use this property to attain the required resonance-free region.

\textbf{Organization of the paper.}
\begin{itemize}
\item In Section \ref{sec:pre}, we gather the basic facts and preliminaries about hyperbolic spaces, geometrically finite discrete subgroups, the structure of cusps, Patterson-Sullivan measure, and Bowen-Margulis-Sullivan measure.

\item In Section \ref{sec:reduction zariski}, we prove that Theorem \ref{main thm} can be reduced to Zariski dense case. 

\item In Section \ref{sec:geo}, we state the results of the coding (Proposition \ref{prop:coding}, Lemma \ref{lem:uni}, \ref{lem:l1}). We construct a hyperbolic skew product flow and state the result that it is exponential mixing (Theorem \ref{thm:skew}). We show that the geodesic flow on $\T^1(M)$ is a factor of this hyperbolic skew product flow (Proposition \ref{lem:loc}) and deduce the exponential mixing of the geodesic flow from Theorem \ref{thm:skew}.

\item In Section \ref{sec:parmea}, we provide an explicit description of the action of an element $\gamma\in \Gamma$ on $\partial \mathbb{H}^{d+1}$ and the estimate on the norm of the derivative of $\gamma$ (Section \ref{sec:explicit}). We list the basics for the multi-cusp case (Section \ref{sec:multi}). The doubling property and the friendliness of Patterson-Sullivan measure are proved in Section \ref{sec:double} and \ref{sec:friendliness}.

\item In Section \ref{sec:code}, we start with the construction of the coding for one cusp case, which is also the first step for multi-cusp case. The main result is exponential decay of the remaining set (Proposition \ref{keylemma}). Section \ref{sec:sep}-\ref{sec:energy} are devoted to the proof the Proposition \ref{keylemma}. The coding for the multi-cusp case will be provided in Section \ref{sec:exptail} and \ref{sec:codmulti}. The results of the coding (Proposition \ref{prop:coding}, Lemma \ref{lem:uni}, \ref{lem:l1}) will be proved in Section \ref{sec:exptail}-\ref{sec:UNI}.

\item In Section \ref{sec:spegap}, we prove a Dolgopyat-type spectral estimate for the corresponding transfer operator and the main result is an $L^2$-contraction proposition (Proposition \ref{L2contracting}).

\item In Section \ref{sec:expmix}, we finish the proof of Theorem \ref{thm:skew}. 

\item In Section \ref{sec:res}, we prove the application of obtaining a resonance-free region for the resolvent $R_{M}(s)$ (Theorem \ref{cor:resonance}).
\end{itemize}

\textbf{Notation.}
 In the paper, given two real functions $f$ and $g$, we write $f\ll g$ if there exists a constant $C>0$ only depending on $\Gamma$ such that $f\leq Cg$. We write $f\approx g$ if $f\ll g$ and $g\ll f$.

\subsection*{Acknowledgement}
We would like to thank S\'{e}bastien Gou\"{e}zel and Carlos Matheus for helpful discussion and thank Amie Wilkinson for suggesting the paper by Lai-Sang Young. The second author would like to express her gratitude to Hee Oh for introducing her to this circle of areas.

Part of this work was done while two authors were in the Bernoulli center for the workshop: Dynamics, Geometry and Combinatorics, we would like to thank the organizers and the hospitality of the center.

\section{Preliminary of hyperbolic spaces and PS measure}\label{sec:pre}

\subsection{Hyperbolic spaces}
We will use the upper-half space model for $\H^{d+1}$: $$\H^{d+1}=\{x=(x_1,\ldots,x_{d+1})\in\R^{d+1}:\,x_{d+1}>0 \}.$$ Let $o=(0,\cdots,0,1)\in\H^{d+1}$. For $x\in \H^{d+1}$, write $h(x)$ for the height of the point $x$, which is the last coordinate of $x$. The Riemannian metric on $\H^{d+1}$ is given by $$\dd s^2=\frac{\dd x_1^2+\cdots+\dd x_{d+1}^2}{x_{d+1}^2}.$$ 
Let $\partial \H^{d+1}$ be the visual boundary. On $\partial\H^{d+1}=\R^d\cup\{\infty \}$, we have the spherical metric, denoted by $d_{\mathbb{S}^d}(\cdot,\cdot)$. 
We also have the Euclidean metric, denoted by 
$d_{E}(x,x')$ or $|x-x'|$ for any $x,x'\in \partial \H^{d+1}$. This metric will be used most frequently; we will simply write $d(\cdot, \cdot)$ when there is no confusion. 

For $g\in G$, it acts on $\partial\H^{d+1}$ conformally. For $x\in\partial\H^{d+1}$, let $|g'(x)|$ be the linear distortion of the conformal action of $g$ at $x$ with respect to the Euclidean metric. It is also the norm of the derivative seen as a linear map on tangent spaces. Let $|g'(x)|_{\S^d}$ be the norm with respect to the spherical metric.
We have the relation
\begin{equation}\label{equ:change}
|g'(x)|_{\S^d}=\frac{1+|x|^2}{1+|gx|^2}|g'(x)|.
\end{equation}
Another formula for $|g'(x)|_{\mathbb{S}^d}$ is 
\[ |g'(x)|_{\S^d}=e^{-\beta_x(g^{-1}o,o)},\]
where $\beta_x(\cdot,\cdot)$ is the Busemann function given by $\beta_x(z,z')=\lim_{t\to +\infty}d(z,x_t)-d(z',x_t)$ with $x_t$ an arbitrary geodesic ray tending to $x$.





We denote $\mathbb{H}^{d+1}\cup \partial \mathbb{H}^{d+1}$ by $\overline{\mathbb{H}^{d+1}}$. 

\subsection{Geometrically finite discrete subgroups}

Let $\Gamma$ be a torsion-free, non-elementary discrete subgroup in $G$. We list some basics of geometrically finite discrete subgroups.

The {\textbf{limit set}} of $\Gamma$ is the set $\Lambda_{\Gamma}$ of all the accumulation points of an orbit $\Gamma x$ for some $x\in \mathbb{H}^{d+1}$. As we assume $\Gamma$ is torsion-free, $\Lambda_{\Gamma}$ is contained in $\partial \mathbb{H}^{d+1}$. The {{convex hull}}, $\text{hull}(\Lambda_{\Gamma})$, of $\Lambda_{\Gamma}$ is the smallest convex subset in $\mathbb{H}^{d+1}$ which contains all the geodesics connecting any two distinct points of $\Lambda_{\Gamma}$. The {{convex core}} of $M$ is $C(M)=\Gamma\backslash \text{hull}(\Lambda_{\Gamma})\subset M$. 

A limit point $x\in\Lambda_{\Gamma}$ is called {\textbf{conical}} if there exists a geodesic ray tending to $x$ and a sequence of elements $\gamma_n\in\Gamma$ such that $\gamma_no$ converges to $x$, and the distance between $\gamma_no$ and the geodesic ray is bounded. A subgroup $\Gamma'$ of $\Gamma$ is called {{parabolic}} if $\Gamma'$ fixes only one point in $\partial\H^{d+1}$. A point $x\in\Lambda_\Gamma$ is called a {\textbf{parabolic fixed point}} if its stabilizer in $\Gamma$, ${\rm{Stab}}_{\Gamma}(x)$, is parabolic. A parabolic fixed point is called {\textbf{bounded parabolic}} if the quotient $\mathrm{Stab}_{\Gamma}(x)\backslash(\Lambda_{\Gamma}-\{x\})$ is compact.

A horoball based at $x\in \partial \mathbb{H}^{d+1}$ is the set $\{y\in\H^{d+1}:\,\ \beta_x(y,o)<t \}$ for some $t\in\R$. The boundary of a horoball is called a horosphere. We call a horoball $H$ based at a parabolic fixed point $x\in\Lambda_{\Gamma}$ a horocusp region, if we have $\gamma H\cap H=\emptyset$ for any $\gamma\in\Gamma-\mathrm{Stab}_{\Gamma}(x)$. Then the image of $H$ in $M$ under the quotient map, $\Gamma\backslash \Gamma H$, is isometric to $\mathrm{Stab}_{\Gamma}(x)\backslash H$ and is called a \textbf{proper horocusp} of $M$.

\begin{defn}[Geometrically finite discrete subgroup~\cite{Bow},~\cite{Ratc}]\label{def:geofinite} A non-elementary discrete subgroup $\Gamma<\operatorname{SO}(d+1,1)^{\circ}$ is called {\textbf{geometrically finite}} if it satisfies one of the following equivalent conditions:
\begin{enumerate}[(i)]
\item There is a (possibly empty) finite union $V$ of proper horocusps of $M$, with disjoint closures, such that $C(M)-V$ is compact.

\item Every limit point of $\Gamma$ is either conical or bounded parabolic.
\end{enumerate}
\end{defn}

\subsection{Structure of cusps}\label{sec:cusps}
Assume that $\Gamma$ is a geometrically finite discrete subgroup with parabolic elements and $\infty$ is a parabolic fixed point of $\Gamma$. Let $\Gamma_\infty^{'}=\mathrm{Stab}_\Gamma(\infty)$ be the parabolic subgroup of $\Gamma$ fixing $\infty$. Then $\Gamma_\infty^{'}$ acts on $\R^d$, part of $\partial \mathbb{H}^{d+1}$, isometrically with respect to the Euclidean metric. 
The following is a result of Bieberbach (see~\cite[Page 5]{GM} or~\cite[Section 2.2]{Bow}). 
\begin{lem}[Bieberbach]
\label{lem:biberbach}
	Consider the action of $\Gamma_{\infty}'$ on $\mathbb{R}^d$. Then there exist a maximal normal abelian subgroup $\Gamma_\infty \subset \Gamma_\infty^{'}$ of finite index and an affine subspace $Z\subset \mathbb{R}^d$ of dimension $k$, invariant under $\Gamma_{\infty}'$, such that $\Gamma_{\infty}$ acts as a group of translations of rank $k$ on $Z$. If $\mathbb{R}^d=Y\times Z$ is an orthogonal decomposition, with $Y\simeq \mathbb{R}^{d-k}$ and associated coordinates $(y,z)$, then we can write each element $\gamma\in \Gamma_\infty^{'}$ in the form
	\begin{equation*}
	\gamma(y,z)=(A_\gamma y, R_\gamma z+b_\gamma), \,\,\, b_{\gamma}\in \R^{k},\,\,\,A_\gamma\in O(n-k),\ R_\gamma\in O(k) 
	\end{equation*}
	where for each $\gamma$, $R_{\gamma}^m=\operatorname{Id}$ for some $m\in \mathbb{N}$, with $m=1$
	if $\gamma\in\Gamma_\infty$.
\end{lem}
The dimension $k$ is called the \textbf{rank} of the parabolic fixed point $\infty$.

Fix an orthogonal decomposition $\mathbb{R}^d=Y\times Z\simeq \mathbb{R}^{d-k}\times \mathbb{R}^k$. As $\Gamma_{\infty}$ acts on $\mathbb{R}^k$ as a group of translations, it admits a fundamental region $\Delta_{\infty}'$ which is an open $k$-dimensional parallelotope in $\R^k$. Since $\Gamma$ is geometrically finite, $\infty$ is a bounded parabolic fixed point. By definition, the quotient $\Gamma_{\infty}'\backslash (\Lambda_{\Gamma}-\{\infty\} )$ is compact; the quotient $\Gamma_{\infty}\backslash (\Lambda_{\Gamma}-\{\infty\} )$ is also compact as $\Gamma_{\infty}$ is a finite index subgroup of $\Gamma_{\infty}'$. Therefore, there exists a constant $C>0$ such that the set $B_Y(C)=\{y\in \R^{d-k}:\, |y|< C \}$ in $\R^{d-k}$ has the property that 
$$\Lambda_{\Gamma}\subset \{\infty\}\cup \left(\cup_{\gamma\in\Gamma_\infty}\gamma \left(\overline{B_Y(C/2)\times \Delta'_{\infty}}\right)\right),$$
\begin{defn}
	We call the open set $\Delta_{\infty}:=B_Y(C)\times \Delta_\infty'$ a \textbf{fundamental region} for the parabolic fixed point $\infty$.
\end{defn}

\begin{figure}
	\begin{center}
	\def\svgwidth{10cm}
\begingroup%
  \makeatletter%
  \providecommand\color[2][]{%
    \errmessage{(Inkscape) Color is used for the text in Inkscape, but the package 'color.sty' is not loaded}%
    \renewcommand\color[2][]{}%
  }%
  \providecommand\transparent[1]{%
    \errmessage{(Inkscape) Transparency is used (non-zero) for the text in Inkscape, but the package 'transparent.sty' is not loaded}%
    \renewcommand\transparent[1]{}%
  }%
  \providecommand\rotatebox[2]{#2}%
  \newcommand*\fsize{\dimexpr\f@size pt\relax}%
  \newcommand*\lineheight[1]{\fontsize{\fsize}{#1\fsize}\selectfont}%
  \ifx\svgwidth\undefined%
    \setlength{\unitlength}{420.43622457bp}%
    \ifx\svgscale\undefined%
      \relax%
    \else%
      \setlength{\unitlength}{\unitlength * \real{\svgscale}}%
    \fi%
  \else%
    \setlength{\unitlength}{\svgwidth}%
  \fi%
  \global\let\svgwidth\undefined%
  \global\let\svgscale\undefined%
  \makeatother%
  \begin{picture}(1,0.65078187)%
    \lineheight{1}%
    \setlength\tabcolsep{0pt}%
    \put(0,0){\includegraphics[width=\unitlength,page=1]{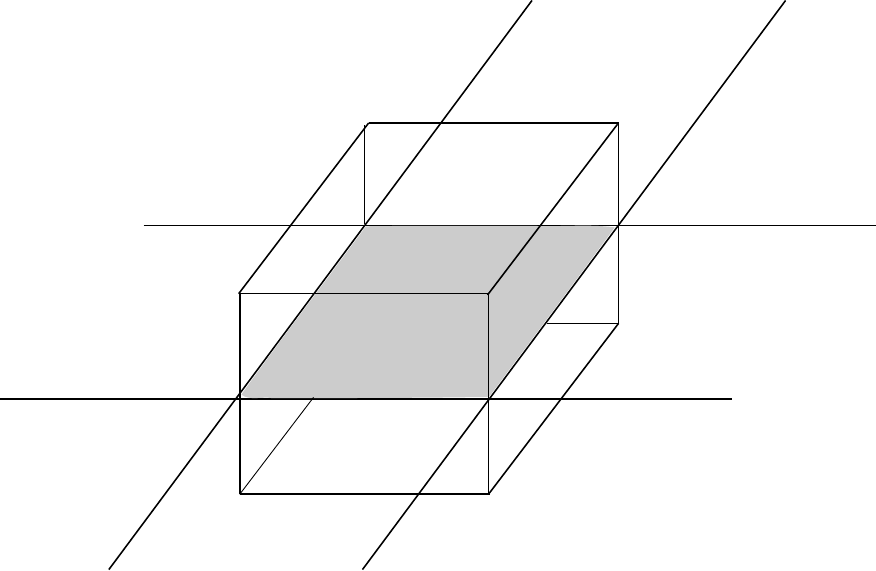}}%
    \put(0.42109719,0.34865532){\makebox(0,0)[lt]{\lineheight{1.25}\smash{\begin{tabular}[t]{l}$\Delta_\infty'$\end{tabular}}}}%
    \put(0.75081546,0.40652781){\makebox(0,0)[lt]{\lineheight{1.25}\smash{\begin{tabular}[t]{l}$\mathbb{R}^2$\end{tabular}}}}%
  \end{picture}%
\endgroup%

		\end{center}
	\caption{Here $\infty$ is a parabolic fixed point of rank 2 in $\partial\H^4$. The intersection $\Lambda_{\Gamma}\cap \R^3$ has bounded distance to $\R^2$. }
\end{figure}

\subsection{PS measure and BMS measure}\label{sec:PS}
\textbf{Patterson-Sullivan measure.} Recall $\delta$ is the critical exponent of $\Gamma$. Patterson~\cite{Pat2} and Sullivan~\cite{Sul1} constructed a $\Gamma$-invariant conformal density $\{\mu_y \}_{y\in\H^{d+1}}$ of dimension $\delta$ on $\Lambda_{\Gamma}$, which is a set of finite Borel measures such that for any $y,z\in \H^{d+1}$, $x\in \partial \H^{d+1}$ and $\gamma \in \Gamma$,
\begin{equation}
\label{ps quasi}
\frac{\dd\mu_y}{\dd\mu_z}(x)=e^{-\delta\beta_{x}(y,z)}\,\,\,\text{and}\,\,\, (\gamma)_{*}\mu_y=\mu_{\gamma y},
\end{equation}
where $\gamma_{*}\mu_{y}(E)=\mu_{y}(\gamma^{-1}E)$ for any Borel subset $E$ of $\partial \H^{d+1}$. This family of measures is unique up to homothety, and the action of $\Gamma$ on $\partial \H^{d+1}$ is ergodic relative to the measure class defined by these measures. 

As $\mu_y$'s are absolutely continuous with respect to each other, for most of the paper, we will consider $\mu_{o}$ and denote it by $\mu$ for short. We call it the Patterson-Sullivan measure (or PS measure). The following {\textbf{quasi-invariance}} property of the PS measure will be frequently used: for any Borel subset $E$ of $\partial \H^{d+1}$ and any $\gamma\in \Gamma$,
\begin{equation}
 \mu(\gamma E)=\int_E|\gamma'(x)|_{\S^d}^\delta\dd\mu(x).
 \end{equation}
 
\textbf{Bowen-Margulis-Sullivan measure.} Let $\partial^2 (\H^{d+1})=\partial\H^{d+1}\times\partial\H^{d+1}-\text{Diagonal}$.
The Hopf parametrization of $\T^1(\mathbb{H}^{d+1})$ as $\partial^2 (\H^{d+1})\times \mathbb{R}$ is given by
\begin{equation*}
v\mapsto (x,x_-,s=\beta_{x}(o,v_*)),
\end{equation*}
where $x$ (resp. $x_-$) is the forward endpoint (resp. backward) endpoint of $v$ under the geodesic flow, and $v_*\in \mathbb{H}^{d+1}$ is the based point of $v$. The geodesic flow on $\T^1(\H^{d+1})$ is represented by the translation on $\mathbb{R}$-coordinate. 

The Bowen-Margulis-Sullivan measure (or BMS measure) on $\operatorname{T}^1(\H^{d+1})$ is defined by 
\begin{equation*}
\dd\tilde{m}^{\operatorname{BMS}}(x,x_-,s)=e^{\delta \beta_{x}(o,x_*)} e^{\delta \beta_{x_-}(o,x_*)}\dd \mu(x) \dd\mu(x_-)\dd s,
\end{equation*}
where $x_*$ is the based point of the unit tangent vector given by $(x,x_-,s)$.
It is invariant under the geodesic flow $\calG_t$ from the definition.
The group $\Gamma$ acts on $\partial^2 (\H^{d+1})\times \mathbb{R}$ by
\begin{equation*}
\gamma (x,x_-,s)=(\gamma x, \gamma x_-,s-\beta_{x}(o,\gamma^{-1}o)). 
\end{equation*}
This formula, together with \eqref{ps quasi}, implies that $\tilde{m}^{\operatorname{BMS}}$ is left $\Gamma$-invariant; hence $\tilde{m}^{\operatorname{BMS}}$ induces a measure $m^{\operatorname{BMS}}$ on $\operatorname{T}^1(M)$, which is the Bowen-Margulis-Sullivan measure on $\operatorname{T}^1(M)$. For geometrically finite discrete subgroups, Sullivan showed that $m^{\operatorname{BMS}}$ is finite and ergodic with respect to the action of the geodesic flow~\cite{Sul}. Otal and Peign\'{e} showed that $m^{\operatorname{BMS}}$ is the unique measure supported on the non-wandering set of the geodesic flow with maximal entropy $\delta$~\cite{OtPe}. 
After normalization, we suppose that $m^{\operatorname{BMS}}$ is a probability measure.

\section{Reduction to Zariski dense case}
\label{sec:reduction zariski}
The group $\SO$ is Zariski closed and connected and the subgroup $\SO^\circ$ is its analytic connected component containing identity. For a subgroup $\Gamma$ of $\SO^\circ$, it is said to be Zariski dense in $\SO^\circ$ if it is Zariski dense in $\SO$. The proof of Theorem \ref{main thm} can be reduced to Zariski dense case. 
\begin{thm}
\label{zariski}
Assume that $\Gamma<\operatorname{SO}(d+1,1)^{\circ}$ is a Zariski dense, torsion-free, geometrically finite subgroup with parabolic elements. The geodesic flow $(\mathcal{G}_t)_{t\in \mathbb{R}}$ on $\operatorname{T}^1(M)$ is exponentially mixing with respect to $\bms$: there exists $\eta>0$ such that for any functions $\phi, \psi\in C^1(\T^1(M))$ and any $t>0$, we have
\begin{equation*}
\int_{\T^1(M)} \phi\cdot\psi\circ\calG_t\ \dd\bms =\bms (\phi) \bms (\psi)+O(\lVert \phi \rVert_{C^1} \lVert \psi\rVert_{C^1}e^{-\eta t}).
\end{equation*}
 \end{thm}

\begin{proof}[\textbf{From Theorem \ref{zariski} to Theorem \ref{main thm}}]
Suppose $\Gamma$ is not Zariski dense. Let $H$ be the Zariski closure of $\Gamma$ in $\SO$ and let $H_1$ be the Zariski connected component of $H$ containing the identity. Let $\Gamma_1=\Gamma\cap H_1$. Then $\Gamma_1$ is a finite index subgroup of $\Gamma$ and the Zariski closure of $\Gamma_1$ is $H_1$. We will only consider $\Gamma_1$ because the exponential mixing of $\Gamma$ follows from the same statement for $\Gamma_1$ by taking covering space.
	
	Let $H_o$ be the analytic connected component of $H_1$ containing identity. Since $\Gamma$ is non-elementary, the group $H_o$ doesn't fix any point on the boundary. By a classic result (see~\cite{BZ} for example), up to conjugacy, $H_o$ preserves a hyperbolic subspace $\H^m$ with $m\leq d$ and the restriction of $H_o$ to $\H^m$ contains $\SOm^{\circ}$ with compact kernel. Preserving subspace is a Zariski closed condition, we know that $H_1$ also preserves $\H^m$ and the restriction of $H_1$ to $\H^m$ satisfies the same properties as $H_o$. Since $\Gamma_1$ is a torsion free discrete subgroup, the restriction map $\Gamma_1\rightarrow \Gamma_1|_{\H^m}$ is injective. Then the Zariski closure of $\Gamma_1|_{\H^m}$ also contains $\SOm^{\circ}$. 
	At most passing to an index 4 subgroup, we can suppose that $\Gamma_1|_{\H^m}$ is a subgroup of $\SOm^\circ$. Hence $\Gamma_1|_{\H^m}$ is Zariski dense in $\SOm^{\circ}$ and geometrically finite. (Definition \ref{def:geofinite} (2) implies that $\Gamma_1|_{\H^m}$ is still geometrically finite.) The BMS measure $\bms$ of $\Gamma_1$ on the unit tangent bundle $\Gamma_1\backslash \T^1\H^{d+1}$ is actually supported on $\Gamma_1\backslash \T^1\H^m$, which is the Zariski dense case.
	\end{proof}

\section{The geodesic flow and the boundary map}
\label{sec:geo}
For the rest of the paper, our standing assumption is
\begin{align*}
	&\Gamma<G\,\,\,\text{Zariski dense, torsion-free, geometrically finite with parabolic elements}\\
	&\text{and}\,\,\,\infty\,\,\,\text{is a parabolic fixed point of}\,\,\, \Gamma.
\end{align*}
Let $\Delta_0:=\Delta_\infty$ be a fundamental region for the parabolic fixed point $\infty$ described in Section~\ref{sec:cusps}. In Section~\ref{sec:code}, we will construct a coding of the limit set satisfying the following properties.
\begin{prop}
\label{prop:coding}
	There are constants $C_1>0$, $\lambda,\,\epsilon_0\in (0,1)$, a countable collection of disjoint, open subsets $\{\Delta_j\}_{j\in\N}$ in $\Delta_0$ and an expanding map $T:\sqcup_j\Delta_j\to \Delta_0$ such that: 
		\begin{enumerate}
		\item $\sum_{j}\mu(\Delta_j)=\mu(\Delta_0)$.
		\item For each $j$, there is an element $\gamma_j\in \Gamma$ such that $\Delta_j=\gamma_j \Delta_0$ and $T|_{\Delta_j}=\gamma_j^{-1}$. 
		\item Each $\gamma_j$ is a uniform contraction: $|\gamma_j'(x)|\leq \lambda$ for all $x\in \Delta_0$.
		\item For each $\gamma_j$, $|\D(\log |\gamma_j'|)(x)|<C_1$ for all $x\in \Delta_0$, where $\D(\log |\gamma_j'|)(x)$ is the differential of the map $z\mapsto \log |\gamma_j'(z)|$ at $x$.
		\item Let $R$ be the function on $\sqcup_j \Delta_j$ given by $R(x)=\log|\D T(x)|$. Then
		$\int e^{\epsilon_o R}\dd\mu<\infty.$
	\end{enumerate}
\end{prop}

Denote by $\mathcal{H}=\{\gamma_j\}_{j\in\N}$ the set of inverse branches of $T$. The last property is known as the exponential tail property and 
we will prove another form instead:
\begin{equation}
\label{sum}
\sum_{\gamma\in\calH }|\gamma'|_\infty^{\delta-\epsilon_o}<\infty, 
\end{equation}
where $|\gamma'|_\infty=\sup_{x\in\Delta_0}|\gamma'(x)|$. Proposition~\ref{prop:coding} (5) can be deduced from \eqref{sum} by separating the integral to the sum of integrals over $\Delta_j$ and using quasi-invariance of PS measure.

Using Proposition~\ref{prop:coding}, it can be shown that there exists a $T$-invariant ergodic probability measure $\nu$ on $\Delta_0$ which is absolutely continuous with respect to PS measure and the density function $\bar{f}_0$ is a positive Lipschitz function bounded away from $0$ and $\infty$ on $\Delta_0\cap\Lambda_\Gamma$ (see for example~\cite[Lemma 2]{You}).

The coding satisfies \textbf{uniform nonintegrable condition (UNI)}. 
Let 
\begin{align*}
&R_n(x):=\sum_{0\leq k\leq n-1}R(T^k(x))\,\,\,\text{for}\,\,\,x\,\,\,\text{with}\,\,\,T^k(x)\in \sqcup_j \Delta_j\,\,\,\text{for all}\,\,\,0\leq k\leq n-1,\\
&\mathcal{H}^n=\{\gamma_{j_1} \cdots \gamma_{j_n}:\,\gamma_{j_k}\in \mathcal{H}\,\,\,\text{for}\,\,\,1\leq k\leq n\}.
\end{align*}
For $\gamma\in\calH^n$, we have
$R_n(\gamma x)=-\log|\gamma'(x)|$. Set 
\begin{equation}
\label{constant c2}
C_2=C_1/(1-\lambda).
\end{equation}
 Then by Proposition~\ref{prop:coding} (3) and (4), we obtain for any $\gamma\in \mathcal{H}_n$,
\begin{equation}
\label{uniform contraction}
\sup_{x\in \Delta_0}\left| \D(\log |\gamma'|)(x)\right|\leq C_2.
\end{equation}

\begin{lem}[UNI]\label{lem:uni}
	There exist $\r>0$ and $\epsilon_0>0$ such that for any $C>1$ the following holds for any large $n_0$. There exist $j_0\in\N$ and $\{\gamma_{mj}:1\leq m\leq 2, 1\leq j\leq j_0\}$ in $\mathcal H_{n_0}$ such that for any $x\in \Lambda_\Gamma\cap\overline\Delta_0$ and any unit vector $e\in\R^d$ there exists $j\leq j_0$ such that for all $y\in B(x,\r)$
	\begin{equation}\label{equ:uni}
	|\partial_e(\tau_{1j}-\tau_{2j})(y)|\geq \epsilon_0,
	\end{equation}
	where $\tau_{mj}(x)=R_{n_0}(\gamma_{mj}x)$. 
	Moreover, for all $m,j$,
	\begin{equation}\label{equ:hm}
	|\D\tau_{mj}|_\infty\leq C_2, \,\,\,
	|\gamma_{mj}'|_\infty\leq \epsilon_0/C.
	\end{equation}
\end{lem}

The expanding map in the coding gives a contracting action in a neighborhood of $\infty$.
\begin{lem}\label{lem:l1}
	There exist $0<\lambda<1$ and a neighbourhood $\Lambda_-$ of $\infty$ in $\Lambda_{\Gamma}$ such that $\Lambda_{-}$ is disjoint from $\overline{\Delta}_0$ and for any $\gamma\in\calH$ and any $y,y'\in \Lambda_-$, 
	\begin{equation}\label{equ:lam1}
	\gamma^{-1}(\Lambda_{-})\subset\Lambda_{-},\ \ 
	d_{\S^d}(\gamma^{-1}y,\gamma^{-1}y')\leq \lambda d_{\S^d}(y,y').
	\end{equation}
\end{lem}
The proofs of these results will be postponed to Section~\ref{sec:code}. Proposition~\ref{prop:coding} and Lemma~\ref{lem:l1} will be proved at the end of Section~\ref{sec:codmulti} and Lemma~\ref{lem:uni} will be proved in Section~\ref{sec:UNI}.



\subsection{A semiflow over hyperbolic skew product}



\textbf{Hyperbolic skew product.} We construct a hyperbolic skew product using Lemma~\ref{lem:l1}.
Let $\Lambda_{+}=\Lambda_{\Gamma}\cap\left(\sqcup_j\Delta_j\right)$ and $\Lambda_{-}$ be given as Lemma~\ref{lem:l1}. 
Define the map $\hat{T}$ on $\Lambda_+\times \Lambda_-$ by 
\begin{equation}
\label{equ:expanding}
 \hat{T}(x,x_-)=(\gamma_j^{-1}x,\gamma_j^{-1}x_-)\,\,\,\text{for}\,\,\, (x,x_-)\in \Lambda_+\times \Lambda_- \,\,\,\text{with}\,\,\, x\in\Delta_j,
\end{equation}
where $\gamma_j$ is given as in Proposition \ref{prop:coding} (2). Lemma~\ref{lem:l1} implies $\gamma^{-1}\Lambda_-\subset \Lambda_-$ for any $\gamma \in \mathcal{H}$. So $\hat{T}$ is well-defined. 

 Let $p:\Lambda_+\times \Lambda_-\to \Lambda_+$ be the projection to the first coordinate. 
This gives rise to a semiconjugacy between $\hat{T}$ and $T$. We equip $\Lambda_+\times\Lambda_-$ with the metric 
$$d((x,x_-),(x',x_-'))=d_E(x,x')+d_{\S^d}(x_-,x_-').$$ 
\eqref{equ:lam1} implies that the action of $\hat{T}$ on the fibre $\{ x\}\times\Lambda_-$ is contracting. Using this observation,we obtain
\begin{prop}\label{prop:dis}
\begin{enumerate}
\item
	 There exists a unique $\hat{T}$-invariant, ergodic probability measure $\hat\nu$ on $\Lambda_+\times \Lambda_-$ whose projection to $\Lambda_+$ is $\nu$. 
\item	
	 We have a disintegration of $\hat{\nu}$ over $\nu$: for any continuous function $w$ on $\Lambda_+\times\Lambda_-$,
	\begin{equation*}
	 \int_{\Lambda_+\times\Lambda_-} w\dd\hat{\nu}=\int_{\Lambda_+}\int_{\Lambda_-}w\dd\nu_x(x_-)\dd\nu(x). 
	\end{equation*}
	Moreover, there exists $C>0$ such that for any Lipschitz function $w$ on $\Lambda_+\times \Lambda_-$, defining $\bar{w}(x)=\int w\dd\nu_x$, we have
	\begin{equation*}
	 \|\bar{w}\|_{\rm Lip}\leq C\|w\|_{\rm Lip}. 
	\end{equation*}
\end{enumerate}	
\end{prop}
\begin{proof}
	For the first statement, see~\cite[Theorem A]{Kloeckner} or~\cite[Proposition 1]{BM}. For the second statement, see~\cite[Proposition 3, Proposition 6]{BM}, where they consider Riemannian manifold case but the same proofs also work in our fractal case.
\end{proof}
\begin{rem*}
	The measure $\hat\nu$ is actually independent of the choice of the stable direction $\Lambda_-$: any $\Lambda_-$ satisfying Lemma~\ref{lem:l1} will lead to the same measure $\hat{\nu}$.
\end{rem*}

\textbf{Hyperbolic skew product flow.}
Let $R:\Lambda_+\to \mathbb{R}_+$ be the function given in Proposition~\ref{prop:coding}. By abusing notation, define $R:\Lambda_+\times\Lambda_-\to \R_+$ by setting $R(x,x_-)=R(x)$. Define the space
\begin{equation*}
\Lambda^{R}=\{(x,x_-,s)\in \Lambda_+\times\Lambda_-\times \mathbb{R}:\,0\leq s< R(x,x_-)\}.
\end{equation*}
Let $R_n=\sum_{j=0}^{n-1}R\circ \hat{T}^j$. The hyperbolic skew product flow $\{\hat{T}_t\}_{t\geq 0}$ over $\Lambda^R$ is defined by $\hat{T}_t(x,x_-,s)=(\hat{T}^n(x,x_-),s+t-R_n(x,x_-))$ for $\hat{\nu}$-almost every $x$, where $n$ is the nonnegative integer such that $0\leq s+t-R_n(x,x_-)<R(\hat{T}^n(x,x_-))$. We equip $\Lambda^R$ with the measure $\dd\hat{\nu}^R:=\dd\hat{\nu}\times \dd t/\bar{R}$, where $dt$ is Lebesgue measure on $\mathbb{R}_+$ and $\bar{R}=\int_{\Lambda_+\times\Lambda_-}R\dd\hat{\nu}$. This is a $\hat{T}_t$-invariant ergodic measure. 

\begin{rem*}
We don't use the commonly used ``suspension space" construction to construct $\Lambda^R$. 
The reason is that we will use a cutoff function in the proof of Theorem~\ref{main thm} and such cutoff functions are ill-defined in the suspension space, which is a quotient space of $\Lambda_+\times \Lambda_-\times \mathbb{R}$.
\end{rem*}

For any $L^{\infty}$ function $w:\Lambda^{R}\to \mathbb{R}$, the Lipschitz norm of $w$ is defined by
\begin{equation}
\label{equ:function norm}
\|w\|_{\operatorname{Lip}}=|w|_\infty+\sup_{ (y,a)\neq(y',a')\in \Lambda^R}\frac{|w(y,a)-w(y',a')|}{d(y,y')+|a-a'|}. 
\end{equation}
In Section~\ref{sec:expmix}, we will prove that $\hat{T}_t$ is exponential mixing with respect to $\hat{\nu}^R$.
\begin{thm}\label{thm:skew}
	There exist $\epsilon_1>0$ and $C>1$ such that for any Lipschitz functions $u,w$ on $\Lambda^R$ and any $t>0$, we have
	\begin{equation*}
	\left|\int u\ w\circ\hat{T}_t\dd \hat \nu^R-\int u \dd \hat \nu^R\int w\dd \hat \nu^R\right|\leq Ce^{-\epsilon_1 t}\|u\|_{\operatorname{Lip}}\|w\|_{\operatorname{Lip}}. 
	\end{equation*}
\end{thm}


\subsection{Exponential mixing of geodesic flow}
\textbf{The map from $\Lambda^R$ to $\T^1(M)$.} We construct a map from $\Lambda^R$ to $\T^1(M)$ which allows us to deduce the exponential mixing of the geodesic flow from that of $\hat{T}_t$.

Recall the Hopf parametrization in Section~\ref{sec:PS}. We introduce the following time change map to have the function $R$ given by derivative (see Proposition~\ref{prop:coding}):
\begin{equation*}
\tilde{\Phi}:\Lambda^R\to \partial^2(\H^{d+1})\times \mathbb{R},\,\,\,
(x,x_-,s)\mapsto (x,x_-,s-\log(1+|x|^2)).
\end{equation*}
The map $\tilde{\Phi}$ induces a map $\Phi:\Lambda^R\to \T^1(M)$, where we use the Hopf parametrization to identify $\T^1(M)$ with $\Gamma\backslash \partial^2(\H^{d+1})\times \mathbb{R}$. Note that $\Lambda_+\times \{\infty\}\times \{0\}$ is mapped to the unstable horosphere based at $\infty$ and passing $o$. The map $\Phi$ defines a semiconjugacy between two flows:
\begin{equation}
\label{equ:flow conj}
\Phi\circ \hat{T}_t=\mathcal{G}_t\circ \Phi,\,\,\,\text{for}\,\,\,t\geq 0.
\end{equation}
To see this, note that for any $(x,x_-,s)\in \Lambda^R$, we have the expresssion
\begin{align*}
&\hat{T}_t(x,x_-,s)=(\hat{T}^n(x,x_-),s+t-R_n(x,x_-)),\ \ \hat{T}^n(x,x_-)=\gamma^{-1}(x,x_-)\,\,\,\text{for some}\,\,\,\gamma\in \calH^n.
\end{align*}
By straightforward computation, we obtain
\begin{equation*}
\tilde{\Phi}\circ \hat{T}_t(x,x_-,s)=\mathcal{G}_t\circ \gamma^{-1}\tilde{\Phi}(x,x_-,s),
\end{equation*}
which leads to (\ref{equ:flow conj}) by passing to the quotient space.
\vspace{2mm}

\textbf{Relating $\hat{\nu}^{R}$ with $m^{\operatorname{BMS}}$.} 
 The map $\Phi$ is not injective in general. Nevertheless, we are able to use $(\Lambda^R,\hat{T}_t,\hat{\nu}^R)$ to study $(\T^1(M), \mathcal{G}_t,m^{\operatorname{BMS}})$. The main result is the following proposition.
\begin{prop}
\label{lem:loc}
The map $\Phi:(\Lambda^R,\hat{T}_t,\hat{\nu}^R)\to (\T^1(M),\mathcal{G}_t,m^{\operatorname{BMS}})$ is a factor map, i.e., 
\begin{equation*}
\Phi_*\hat{\nu}^R=m^{\operatorname{BMS}}\,\,\, \text{and}\,\,\,\Phi\circ \hat{T}_t=\mathcal{G}_t\circ \Phi\,\,\,\text{for all}\,\,\,t\geq 0.
\end{equation*}
\end{prop}

We need two lemmas to prove this proposition.
\begin{lem}\label{lem:V}
There exists a measurable subset $U$ in $\Lambda^R$ such that by setting $V=\Phi(U)$ in $\T^1(M)$, the restriction map of $\Phi$ on $U$ gives a bijection between $U$ and $V$. Moreover, the set $V$ is of positive $BMS$ measure.
\end{lem}
\begin{proof}
We make use of the following commutative diagram
\begin{equation*}
\begin{tikzcd}
\Lambda^R \arrow[d,"\Phi"] \arrow[r, "\tilde{\Phi}"] & \partial^2(\H^{d+1})\times \mathbb{R}\arrow[ld,"\pi"] \\
 \T^1(M) &
\end{tikzcd}
\end{equation*}
where $\pi$ is the covering map. Let $\epsilon>0$ be a number such that $\epsilon<\inf_{(x,x_-)\in \Lambda_+\times \Lambda_-}R(x,x_-)$. Set $S=\Lambda_+\times \Lambda_-\times [0,\epsilon)$. The restriction map $\tilde{\Phi}|_S$ gives a bijection between $S$ and its image. Pick any $x\in S$. As $\pi$ is a covering map, there exists an open set $W\subset \partial^2(\H^{d+1})\times \mathbb{R}$ containing $\tilde{\Phi}(x)$ such that the restriction map $\pi|_W$ is a bijection. The sets $U=\tilde{\Phi}^{-1}(W\cap \tilde{\Phi}(S))$ in $\Lambda^R$ and $V=\pi(W\cap \tilde{\Phi}(S))$ satisfy the proposition.
\end{proof}

\begin{lem}\label{lem:stable}
Let $\cal Q'$ be any subset in $\Lambda^R$ with full $\hat{\nu}^R$ measure and $\cal Q$ be any subset in $\T^1 (M)$ with full $m^{\operatorname{BMS}}$ measure. Then there exist $x\in \cal Q'$ and $y\in \cal Q$ such that $\Phi(x)$ and $y$ are in the same stable leaf.
\end{lem}
\begin{proof}
The idea of the proof is straightforward: we make use of the local product description of $\hat{\nu}^R$ and $m^{\operatorname{BMS}}$. 

Let $\Phi_U$ be the restriction of $\Phi$ on $U$. In view of Lemma~\ref{lem:V}, we can consider the measure $\Phi_U^*(m^{\operatorname{BMS}}|_V)$ on $U$, the pull back of $m^{\operatorname{BMS}}|_V$, and denote it by $m$ for simplicity. We can choose $U$ and $V$ sufficiently small so that $m$ is given by 
\begin{equation*}
\dd m(x,x_-,t)=cD(x,x_-)^{-2\delta}\dd\mu(x)\dd\mu(x_-)dt,
\end{equation*} 
where $c$ is a positive constant and $D(x,x_-)=e^{\beta_{x}(o,x_*)/2} e^{\beta_{x_-}(o,x_*)/2}$ known as the visual distance.
Let $p:\Lambda^R\to \Lambda_+\times \mathbb{R}$ be the projection map, forgetting the $\Lambda_-$-coordinate. Then the pushforward measure $p_*m$ is given by
\begin{equation*}
\dd p_*m(x,t)=c\dd\mu(x)\dd t\int_{\{x\}\times \Lambda_-\times \{t\}\cap U} D(x,x_-)^{-2\delta}\dd\mu(x_-).
\end{equation*}
So it is absolutely continuous with respect to the measure $\dd\nu\otimes \dd t$.

We can find a set of the form $B=B_+\times \Lambda_-\times (t_1,t_2)$ such that $\nu(B_+)>0$ and $m(B\cap U)>0$. The pushforward measure $p_*(\hat{\nu}^R|_B)$ is given by
\begin{equation}
\label{eqn:absolute continuity}
dp_*(\hat{\nu}^R|_B)=\dd\nu \otimes \dd t.
\end{equation}

On the one hand, we have that $p({\cal{Q}}'\cap B)$ is a conull set in $p(B)$ with respect to $p_*(\hat{\nu}^R|_B)$.

On the other hand, we consider $B':=\Phi_U^{-1}({\cal{Q}}\cap V)\cap B$. It is a set with positive $m$ measure and hence $p_*(B')$ is of positive $p_*(m)$ measure. The fact that $dp_*(m)$ is absolutely continuous with respect to $\dd\nu\otimes \dd t$ and (\ref{eqn:absolute continuity}) imply that $p_*(B')$ is of positive $p_*(\hat{\nu}^R|_B)$ measure. Therefore, 
$$p({\cal{Q}}'\cap B)\cap p(B')\neq \emptyset.$$
Let $(x,t)$ be a point in the intersection. Then the points $(x,x_-,t)\in \cal{Q}'$ and $\Phi(x,x_-',t)\in \cal{Q}$ satisfy the conditions of the lemma.
\end{proof}

\begin{proof}[Proof of Proposition~\ref{lem:loc}]
	Let $f$ be a $C^1$ function on $\T^1(M)$ with finite $C^1$-norm. Since $m^{\operatorname{BMS}}$ is ergodic~\cite{Sul}, by Birkhoff ergodic theorem, for $m^{\operatorname{BMS}}$-a.e. $y$ in $\T^1(M)$
	\begin{equation}\label{equ:bms}
	\lim_{T\rightarrow+\infty}\frac{1}{T}\int_{0\leq t\leq T} f(\calG_t y)\dd t=\int f\dd m^{\operatorname{BMS}}.
	\end{equation}
	Let $\cal Q$ be the set of points at which (\ref{equ:bms}) hold and it is a set of full $m^{\operatorname{BMS}}$ measure.
	
	We consider $f\circ \Phi$, which can be thought as the lifting of $f$ to $\Lambda^R$. It is $\hat{\nu}^R$-integrable. Since $\hat{T}_t$ is mixing with respect to $\hat{\nu}^R$, by Birkhoff ergodic theorem, for $\hat{\nu}^R$-a.e. $x$,
	\[ \lim_{T\rightarrow+\infty}\frac{1}{T}\int_{0\leq t\leq T} f\circ\Phi(\hat{T}_t x)\dd t=\int f\circ\Phi\dd \hat{\nu}^R. \]
	Using the semiconjugacy $\Phi\circ\hat{T}_t=\calG_t\circ\Phi$, we actually have
	\begin{equation}
	\label{equ:nu}
	\lim_{T\rightarrow+\infty}\frac{1}{T}\int_{0\leq t\leq T} f(\calG_t\Phi x)\dd t=\int f\circ\Phi\dd \hat{\nu}^R.
	\end{equation}
	Let $\cal Q'$ be the set of points at which (\ref{equ:nu}) hold and it is a set of full $\hat{\nu}^R$ measure.
	
	By Lemma~\ref{lem:stable}, there exist points $x\in \Lambda^R$ and $y\in \T^1(M)$ such that $\Phi(x)$ and $y$ are in the same stable leaf. Due to $d(\calG_t y,\calG_t\Phi x)\rightarrow 0$ as $t\rightarrow+\infty$ and the uniform continuity of $f$,
	\[ \lim_{T\rightarrow+\infty}\left(\frac{1}{T}\int_{0\leq t\leq T} f(\calG_t y)\dd t-\frac{1}{T}\int_{0\leq t\leq T} f(\calG_t\Phi x)\dd t\right)=0.\]
	Therefore, we can deduce that
	\[\int f\dd m^{\operatorname{BMS}}=\int f\circ\Phi\dd \hat{\nu}^R. \]
	The above equation holds for every $C^1$ function on $\T^1(M)$. The proof is complete.
\end{proof}

\textbf{Proof of Theorem~\ref{zariski}.} We are ready to prove Theorem~\ref{zariski}. With Theorem~\ref{thm:skew} and Proposition~\ref{lem:loc} available, the work lies in the comparing the norm of the functions on $\Lambda^{R}$ with that on $\T^1(M)$. This is not obvious. Consider two points of the form $(y,a)$ and $(y',a)$ in $\Lambda^R$. By (\ref{equ:function norm}), $d((y,a),(y',a))$ remains the same when $a$ changes. But if these two points are projected to $\T^1(M)$, changing $a$ means flowing these two points by the geodesic flow and $d(\Phi((y,a),\Phi(y',a)))$ will change. Moreover, the function $R$ used to define $\Lambda^R$ is unbounded, making the argument more complex.
\begin{proof}[Proof of Theorem~\ref{zariski}]
Let $u,v$ be any two $C^1$-functions on $\T^1(M)$ with finite $C^1$-norm. Without loss of generality, we may assume that $m^{\operatorname{BMS}}(u)=0$. Set $U=u\circ \Phi$ and $W=w\circ \Phi$. Using the semiconjugacy of $\Phi$, we obtain
\begin{equation*}
\int u\cdot w\circ \mathcal{G}_t \dd m^{\operatorname{BMS}}=\int U \cdot W\circ \hat{T}_t\dd\hat{\nu}^R.
\end{equation*}

We use a cutoff function to relate the norms of $U,W$ with those of $u,w$. Let $\epsilon>0$ be a constant less than $\epsilon_1/2$. Let $\tau_t$ be a decreasing Lipschitz function on $[0,\infty)$ such that $\tau_t=1$ on $[0,\epsilon t]$, $\tau_t =0$ on $(\epsilon t+1,\infty)$ and $|\tau_t|_{\operatorname{Lip}}<2$. Set $U_t=U\cdot \tau_t$ and $W_t=W\cdot \tau_t$. For any two points $(y,a)$ and $(y',a')$ (we may assume $a\geq a'$), we have
\begin{align*}
&|U_t(y,a)-U_t(y',a')|\\ 
\leq& |U_t(y,a)-U_t(y,a')|+|U_t(y,a')-U_t(y',a')|\\
\leq &\tau_t(a')|U(y,a)-U(y,a')|+|u|_{\infty} |\tau_t(a)-\tau_t (a')|+\tau_t(a')|U(y,a')-U(y',a')|\\
\ll & |u|_{C^1}|a-a'|+|u|_{\infty} |a-a'|+e^{\epsilon t} |u|_{C^1} d(y,y'),
\end{align*}
where to obtain the last inequality, we use the fact $d(\Phi(y,a'),\Phi(y',a'))\leq e^{a'}d(y,y')$ and $\tau_t\neq 0$ only on $[0,\epsilon t+1]$. Therefore, we have
\begin{equation}
\label{UT}
\lVert U_t\rVert_{\operatorname{Lip}}\ll e^{\epsilon t} \lVert u\rVert_{C^1}.
\end{equation}
A verbatim of the above argument also implies $\lVert W_t\rVert_{\operatorname{Lip}}\ll e^{\epsilon t} \lVert w\rVert_{C^1}$.

We also need the following $L^1$-estimate. Using the exponential tail condition (Proposition~\ref{prop:coding} (5)), we obtain
\begin{align}\label{ut}
\nonumber&\lvert U_t-U\rvert_{L^1(\hat{\nu}^R)}\leq |u|_{\infty} \int \max\{R(x)-\epsilon t,0\}\dd\nu(x)\\
\ll & |u|_{\infty} \int e^{\epsilon_o (R(x)-\epsilon t)}\dd\nu(x)\ll e^{-\epsilon_o \epsilon t}|u|_{\infty}. 
\end{align}
The similar estimate holds for $W_t-W$.
As $m^{\operatorname{BMS}}(u)=0$, we have 
\begin{equation}
\label{UL}
|\int U_t\dd\hat{\nu}^R|\ll e^{-\epsilon_o \epsilon t}|u|_{\infty}.
\end{equation}

Using Theorem~\ref{thm:skew} together with \eqref{UT}, \eqref{ut} and \eqref{UL}, we obtain
	\begin{align*}
	&|\int U \cdot W\circ \hat{T}_t\dd\hat{\nu}^R|\\
	\leq &|\int U_t\cdot W_t\circ\hat{T}_t\dd\hat{\nu}^R|+|\int (U-U_t)\cdot W_t\circ\hat{T}_t\dd\hat{\nu}^R|+|\int U\cdot (W-W_t)\circ\hat{T}_t\dd\hat{\nu}^R |\\
	\ll &|\int U_t\dd\hat{\nu}^R|\cdot|\int W_t\dd\hat{\nu}^R|+e^{-\epsilon_1 t}\|U_t\|_{\operatorname{Lip}}\|W_t\|_{\operatorname{Lip}}+|w|_\infty|U-U_t|_{L^1(\hat{\nu}^R)}+|u|_\infty|W-W_t|_{L^1(\hat{\nu}^R)}\\
	\ll&(e^{-(\epsilon_1-2\epsilon)t}+e^{-\epsilon_o\epsilon t} )|u|_{C^1}|w|_{C^1}.
	\end{align*}
	Due to $\epsilon<\epsilon_1/2$, the proof is complete.
\end{proof}
	

\section{Parabolic fixed points and Measure estimate}
\label{sec:parmea}

In this section, we provide a detailed description of the $\Gamma$-action on $\partial\H^{d+1}$ and different types of estimate for the PS measure.

\subsection{Explicit computation}
\label{sec:explicit}

Let $H_\infty$ be the horoball based at $\infty$ given by $\R^d\times \{x\in\R:\, x> 1\}$. For a horoball $H$, we define the height of the horoball by
\[ h(H):=\sup_{y\in H}h(y). \]

\begin{lem}\label{lem:gammax}
	Suppose $g\in G$ is not in $Stab_G(\infty)$. Let $p=g \infty$ and $p'=g ^{-1}\infty$. Then we have
	\begin{itemize}
	\item \begin{equation}\label{equ:xpp}
		h(gH_\infty)=h(g^{-1}H_\infty).
	\end{equation}
	\item For any $x\in \mathbb{H}^{d+1}\cup \partial \mathbb{H}^{d+1}$, we have
	\begin{equation*}
	g x=h(gH_\infty) \frac{x-(p',0)}{|x-(p',0)|^2}\begin{pmatrix} A & 0\\ 0 & 1\end{pmatrix}+(p,0)
	\end{equation*}
	\begin{equation}
	\label{equ:gamma-1}
	g ^{-1}x=h(gH_\infty)\frac{x-(p,0)}{|x-(p,0)|^2}\begin{pmatrix}A^{-1} & 0\\ 0 & 1\end{pmatrix}+(p',0),
	\end{equation}
	where $A\in \operatorname{SO}(d)$, and we view $x, (p,0)$ and $(p',0)$ as row vectors in $\mathbb{R}^{d+1}$.
	\end{itemize}
\end{lem}
\begin{proof}
	By Proposition A.3.9 (2) in~\cite{BP}, the action of $g $ on the upper half space is given by 
	\begin{equation*}
	g x=\lambda \iota(x) \begin{pmatrix}
	A & 0\\ 0 &1
	\end{pmatrix}+(b,0), 
	\end{equation*}
	where $A$ is in $\operatorname{SO}(d)$, $\lambda\in\R^+$, $b\in \R^d$ and $\iota(x)$ either equals $x$ or is given by an inversion with respect to a unit sphere centered at $\R^d\times\{0\}$. In fact, this is the Bruhat decomposition of $G$. Since $g $ does not fix $\infty$, $\iota(x)$ is an inversion. We have for any $x\in \mathbb{H}^{d+1}$
	\begin{equation*}
	g x=\lambda \frac{x-(x',0)}{|x-(x',0)|^2} \begin{pmatrix}
	A & 0\\ 0 &1
	\end{pmatrix}+(b,0)	
	\end{equation*}
	with $x'\in \mathbb{R}^d$. Hence $b=g \infty=p$ and $x'=g ^{-1}\infty=p'$.
	
	Note that
	\[ \height(g x)=\lambda \height(x)/|x-(p',0)|^2. \]
	Since $g $ maps the original horoball $H_\infty=\R^d\times \{ x> 1\}$ to the horoball $gH_\infty$ based at $p$, it follows from the above formula that
	\[h(gH_\infty)=\sup_{x\in \R^d\times \{1\} }\height (g x)=\lambda. \]
	
	To obtain the formula for $g ^{-1}x$, note that 
	\begin{equation*}
	x=g (g ^{-1}x)=h(gH_\infty) \frac{g ^{-1} x-(p',0)}{|g ^{-1} x-(p',0)|^2}\begin{pmatrix} A & 0\\ 0 & 1\end{pmatrix}+(p,0).
	\end{equation*}
	This yields 
	\begin{equation*}
	\frac{1}{h(gH_\infty)}\left(x-(p,0)\right) \begin{pmatrix} A^{-1} & 0\\ 0 & 1\end{pmatrix}+(p',0)=\frac{g ^{-1} x-(p',0)}{|g ^{-1} x-(p',0)|^2}+(p',0).
	\end{equation*}
	Applying both sides the inversion with respect to the unit sphere centered at $(p',0)$ and using the fact that $A\in \operatorname{SO}(d)$, we obtain the formula for $g ^{-1}x$. Meanwhile, we apply to $g ^{-1} x$ the argument used to get the formula for $g x$. Comparing the formulas given by these two methods, we have that $h(gH_\infty)=h(g^{-1}H_\infty)$.
	\qedhere
\end{proof}

\begin{lem}\label{lem:rg1}
	For a horoball $H$ based at $p\neq \infty$ and $g\in G$ not in $Stab_G(\infty)$, we have 
	\begin{align}
	\label{upper bound height}
	&h(g H)\geq \frac{h(g ^{-1}H_\infty) h(H)}{d(g ^{-1}\infty,p)^2+h(H)^2},\\
	\label{lower bound height}
	&h(g H)\leq \frac{h(g ^{-1}H_\infty) h(H)}{(d(g ^{-1}\infty,p)-h(H)/2)^2}.
	\end{align}
\end{lem}
\begin{proof}
	Using \eqref{equ:gamma-1}, we obtain
	\[h(g H)\geq \height(g (p,h(H)))=\frac{h(g ^{-1}H_\infty)h(H)}{d(p,g ^{-1}\infty)^2+h(H)^2}. \]
	
	For \eqref{lower bound height}, we have
	\begin{equation*}
	h(g H)=\sup_{y\in \partial H}h(g y)=\sup_{y\in \partial H} \frac{h(g ^{-1}H_\infty) h(y)}{|y-(g^{-1}\infty,0)|^2}.
	\end{equation*}
	Note that for every $y\in \partial H$, we have $|y-(g^{-1}\infty,0)|^2\geq d_{E}(y',g^{-1}\infty)^2\geq (d(g^{-1}\infty,p)-h(H)/2)^2$, where $y'$ is the projection of $y$ to $\partial \mathbb{H}^{d+1}$ and this yields \eqref{lower bound height}.
\end{proof}

Let $\calP$ be the set of parabolic fixed points in $\partial\H^{d+1}$. Two parabolic fixed points are called equivalent if they are in the same $\Gamma$-orbit. Let $\mathbf P$ be a complete set of inequivalent parabolic fixed points. As $\Gamma$ is geometrically finite, the set $\mathbf P$ is finite. Suppose that $\mathbf{P}=\{p_1,\cdots,p_j \}$ for some $j\in\N$, a complete set of inequivalent parabolic fixed points and set $p_1=\infty$.

Fix a collection of pairwise disjoint horoballs based at parabolic fixed points as follows. 
Without loss of generality, we may assume that $H_{\infty}$ is a horocusp region for $\infty$. For the parabolic fixed point $\gamma \infty$, we attach the horoball $H_{\gamma\infty}:=\gamma H_{\infty}$ to it. For other parabolic fixed point $p_i$ in $\mathbf{P}$, we fix a horoball $H_{p_i}$ based at $p_i$ which is a horocusp region for $p_i$. For other parabolic fixed points $\gamma p_i$ in the $\Gamma$ orbit of $p_i$, we attach the horoball $H_{\gamma p_i}:=\gamma H_{p_i}$ to it. By Definition \ref{def:geofinite}, we can choose the horoballs in such a way that they are pairwise disjoint.
For $p\in \mathcal{P}$, we define the \textbf{height} function $h(p)$ of $p$ as the height of $H_p$ based at $p$, that is
\[h(p):=h(H_p). \]

For $x\in\partial\H^{d+1}$ and $r>0$, set $B(x,r)$ to be the ball centered at $x$ of radius $r$ in $\partial\H^{d+1}$ with respect to the Euclidean metric.

\begin{lem}[Explicit computation]\label{lem:explicit} 
	Suppose $\gamma$ is not in $Stab_\infty(\Gamma)$. Then 
	for any $r>0$ and $x\in\partial\H^{d+1}$, 
	\begin{itemize}
		\item $\gamma^{-1}B(p,r)=B(p',h(p)/r)^c$,
		\item $|(\gamma^{-1})'(x)|=h(p)/d(x,p)^2,\,\,\, |\gamma'(x)|=h(p)/d(x,p')^2$,
	\end{itemize}
	where $p=\gamma\infty$ and $p'=\gamma^{-1}\infty$.
\end{lem}

\begin{proof}
	The first equation follows from \eqref{equ:gamma-1} easily.
	In view of Lemma~\ref{lem:gammax}, the computation of the derivative of inversion maps gives the expression of $|\gamma'(x)|$ and $|(\gamma^{-1})'(x)|$.
\end{proof}	

\subsection{Multi-cusps}\label{sec:multi}
Recall that $\mathbf{P}=\{p_1,\cdots,p_j \}$. For each $p_i$, we consider a coordinate change transformation: 
let $g_i$ be an element in $G$ such that $g_ip_i=\infty$. This $g_i$ is not unique and we can choose a $g_i$ such that $g_iH_{p_i}=H_\infty=\R^d\times\{x> 1 \}$. We will frequently make use of the following commutative diagram:
\begin{equation}
\label{commutative diagram}
 \begin{tikzcd}
\H^{d+1} \arrow{r}{g_i} \arrow[swap]{d}{\Gamma} & \H^{d+1} \arrow{d}{g_i \Gamma g_{i}^{-1}} \\%
\H^{d+1} \arrow{r}{g_i}& \H^{d+1}.
\end{tikzcd}
\end{equation}
On the right hand side of the diagram, the acting group is $g_i\Gamma g_i^{-1}$ and $\infty$ is a parabolic fixed point of the group.

Once $g_i$'s are fixed, we consider the action of $g_i\Gamma g_i^{-1}$ and its parabolic fixed points $g_i\calP$. We think they are in $i$-th hyperbolic space.
Set the horoball $H_{g_ip}(i)=g_iH_p$ for $p\in\calP$
and define the height $$h(g_i p,i):=h(H_{g_ip}(i)).$$
If there is no confusion, we always abbreviate $H_{g_ip}(i)$ and $h(g_i p,i)$ to $H_{g_ip}$ and $h(g_i p)$.

The results in Section~\ref{sec:cusps} hold for each $p_i$. We have the group $(g_i\Gamma g_i^{-1})_\infty$, which is a maximal normal abelian subgroup in $\operatorname{Stab}_{g_i\Gamma g_i^{-1}}(\infty)$.
Write
\begin{equation}
\label{pistabilizer}
 \Gamma_{p_i}=g_i^{-1}(g_i\Gamma g_i^{-1})_\infty g_i.
 \end{equation}
Let $\Delta_{p_i}'$ be a fundamental region for the parabolic fixed point $\infty$ under the $g_i\Gamma g_i^{-1}$-action. We can choose $\Delta_{p_i}'$ in such a way that for $\Delta_{p_i}=g_i^{-1}\Delta_{p_i}'$
\begin{equation}\label{glpi}
\{p_i,\ 1\leq i\leq j\}\cap (\cup_{1\leq k\leq j}\overline{\Delta}_{p_k})=\emptyset.
\end{equation}
This choice is possible because, for each $p_i$, we can find a $\Delta_{p_i}'$ such that $\Delta_{p_i}$ sufficiently close to $p_i$ in $\partial \mathbb{H}^{d+1}$ under the spherical metric. 
Set $$\Delta=\cup_{1\leq k\leq j}\Delta_{p_k}.$$
By \eqref{glpi}, we have $\overline{\Delta}\cap\{\infty \}=\emptyset$. So the set $\overline{\Delta}$ is compact.

Consider any parabolic fixed point $p=\gamma p_i$ with $\gamma\in\Gamma$. We know that such $\gamma$ is not unique (any element in $\gamma \operatorname{Stab}_{\Gamma}(p_i)$ also works) and we fix a choice of $\gamma$ such that $\gamma^{-1}p_i\in\overline{\Delta}_{p_i}$. We call $\gamma$ \textbf{the representation} of $p$. Set $$x_p:=\gamma^{-1}p_i.$$

\begin{lem}\label{lem:height}
	There exists $C>1$ such that
	for $1< i\leq j$ and for any parabolic fixed point $p$ in $\Delta$, we have
	\begin{equation*}
	1/C\leq h(g_i p)/h(p)\leq C. 
	\end{equation*}
\end{lem}
\begin{proof}
	Consider the action of $g_i$ on $\partial \mathbb{H}^{d+1}$. Notice that $p_i=g_i^{-1}\infty$ and $h(p_i)=h(H_{p_i})=h(g_i^{-1}H_\infty)$. 
	Applying \eqref{upper bound height} to the horoball $H_p$ based at $p$ and the element $g_i$, we obtain
	\begin{equation*} 
	h(g_i p)=h(g_iH_{p})\geq \frac{h(p_i) h(p)}{d(p_i,p)^2+h(p)^2}. 
	\end{equation*}
	It follows from $\overline{\Delta}$ compact that $d(p_i,p)$ is bounded for $p\in\Delta$. Then $h(g_i p)\geq h(p)/C$.
	
	For the other inequality, notice that $g_ip_1=g_i\infty$ and $h(g_i p_1)=h(g_i H_\infty)$.
	Applying \eqref{upper bound height} with with the horoball $H_{g_ip}$ based at $g_ip$ and $g_i^{-1}$ we have
	\[h(p)=h(g_i^{-1}H_{g_ip})\geq \frac{h(g_i p_1)h(g_i p) }{d(g_ip_1,g_ip)^2+h(g_i p)^2 }. \]
	It follows from \eqref{glpi} that $g_i\overline{\Delta}\cap\{\infty \}=g_i\overline{\Delta}\cap \{g_ip_i \}=\emptyset$. So $g_i\overline{\Delta}$ is compact and
	 $d(g_ip_1,g_ip)$ is bounded for $p\in\Delta$. Then $h(p)\geq h(g_i p)/C$.
\end{proof}
\begin{lem}\label{lem:bilip}
	For $1< i\leq j$, the map $g_i:\Delta\to g_i\Delta$ is bi-Lipschitz.
\end{lem}
\begin{proof}
	By \eqref{equ:gamma-1}, we have
	$$d(g_ix,g_iy)=h(p_i)d\left(\frac{x-p_i}{|x-p_i|^2},\frac{y-p_i}{|y-p_i|^2}\right)\leq C|x-y|=Cd(x,y),$$
	where the inequality due to \eqref{glpi}.
	
	For the other direction, we use \eqref{equ:gamma-1} to obtain
	\[ d(g_i^{-1}x,g_i^{-1}y)=h(p_i)d\left(\frac{x-g_ip_1}{|x-g_ip_1|^2},\frac{y-g_ip_1}{|y-g_ip_1|^2}\right)\leq C|x-y|=Cd(x,y), \]
	where the inequality is due to that $d(x,g_ip_1)$ is bounded below for $x\in g_i\Delta$ by \eqref{glpi}.
\end{proof}

\textbf{Patterson-Sullivan measure under conjugation.} In the presence of multi-cusps, 
 we need to consider the Patterson-Sullivan measure for the conjugation of $\Gamma$. Recall that $\{\mu_{y}\}_{y\in \H^{d+1}}$ is the $\Gamma$-invariant conformal density of dimension $\delta$ and we denoted $\mu_o$ by $\mu$ for short. For each $g_i$ with $1<i\leq j$, set $\Gamma_i=g_i\Gamma g_i^{-1}$. The limit set $\Lambda_{\Gamma_i}$ is $g_i\Lambda_{\Gamma}$ and the critical exponent of $\Gamma_i$ equals $\delta$. For every $y\in \H^{d+1}$, define the following measure
\begin{equation*}
\tilde{\mu}_{y}:=(g_i)_{*}\mu_{g_i^{-1}y},
\end{equation*}
where $(g_i)_{*}\mu_{g_i^{-1}y}(E)=\mu_{g_i^{-1}y}(g_i^{-1}E)$ for any Borel subset $E$ of $\partial \H^{d+1}$. It is easy to check that $\tilde{\mu}_{y}$ is supported on $\Lambda_{\Gamma_i}$ and for any $y,z\in \H^{d+1}$, $x\in \partial \H^{d+1}$ and $\gamma \in \Gamma_i$,
\begin{equation*}
\frac{\dd\tilde{\mu}_{y}}{\dd\tilde{\mu}_{z}}(x)=e^{-\delta\beta_{x}(y,z)}\,\,\, \text{and}\,\,\,(\gamma)_{*}\tilde{\mu}_{y}=\tilde{\mu}_{\gamma y}.
\end{equation*}
It follows from the uniqueness of $\Gamma_{i}$-invariant conformal density that this construction gives exactly the $\Gamma_i$-invariant conformal density on $\Lambda_{\Gamma_i}$ of dimension $\delta$. In later sections, we will denote $\tilde{\mu}_o$ by $\mu_{\Gamma_i}$ and the above analysis yields that for any Borel subset $E$ of $\partial\H^{d+1}$

\begin{equation}\label{conjugation}
	e^{-\delta d(o,g_io)} \mu(g_i^{-1}E)\leq \mu_{\Gamma_i}(E)\leq e^{\delta d(o,g_io)} \mu(g_i^{-1}E).
\end{equation}


\subsection{Doubling property of PS measure}
\label{sec:double}
We start with two results: Proposition~\ref{double} and Lemma~\ref{lem:bpr}, deduced from ~\cite[Theorem 2]{StrVel}. They used spherical metric, but locally it is equivalent to euclidean metric. 
\begin{prop}\label{double}
\begin{itemize}
	\item (Doubling property) For every $C>1$, there exists $\epsilon<1$ such that for every $x\in\Lambda_\Gamma\cap\Delta$ and $1/C\geq r>0$,
	\[\mu(B(x,r))>\epsilon \mu(B(x,Cr)). \]
	
	\item (Growth of measure) There exists $C_{\three}>1$, such that for every $x\in\Lambda_\Gamma\cap\Delta$ and $r<1/C_{\three}$, 
	\[2\mu(B(x,r))<\mu(B(x,C_{\three}r)). \]
	\end{itemize}
\end{prop} 

\begin{lem}\label{lem:bpr}
	Let $p$ be a parabolic fixed point in $\Delta$ of rank $k$. For $0<r\leq h(p)$,
	\begin{equation}\label{equ:bpr}
	\mu(B(p,r))\approx r^{2\delta-k}h(p)^{k-\delta}.
	\end{equation}
\end{lem}

\begin{lem}
	\label{lem:annulusquasi}
	For every $C>1$, there exists $C'>1$ such that for every parabolic fixed point $p=\gamma\infty\in \Delta$ with $\gamma$ the representation, for any Borel subset $E\subset B(p, C h(p))$, we have
	\begin{equation*}
	h(p)^{\delta}\mu(\gamma^{-1}E)/C'\leq \mu(E)\leq C' h(p)^{\delta}\mu(\gamma^{-1}E).
	\end{equation*}
	\end{lem}

\begin{proof}
	As the PS measure is quasi-invariant, we have
	\begin{align*}
	\mu(\gamma^{-1}E)=\int_{x\in E} |(\gamma^{-1})' x|^{\delta} \left(\frac{1+| x|^2}{1+|\gamma^{-1} x|^2}\right)^{\delta}\dd\mu(x).
	\end{align*}
	By Lemma~\ref{lem:explicit}, we have $|(\gamma^{-1})' x|=h(p)/d(x,p)^2$. We also have
		 \begin{equation*}
	 	|d(\gamma^{-1}\infty,0)-h(p)/d(x,p)| \leq |\gamma^{-1} x|\leq d(\gamma^{-1}\infty,0)+h(p)/d(x,p).
	 \end{equation*}
	Due to $\gamma^{-1}\infty\in \overline{\Delta}$, if $h(p)/d(x,p)>\max\{2d(\gamma^{-1}\infty,0),1 \}$, then $1+|\gamma^{-1}x|^2\approx (h(p)/d(x,p))^2$. Otherwise, due to $h(p)/d(x,p)\geq 1/C$, we also have $1+|\gamma^{-1}x|^2\approx (h(p)/d(x,p))^2$.
	 The lemma follows from these bounds on $|(\gamma^{-1})' x|$ and $1+|\gamma^{-1}x|^2$. The computation also makes sense even if $x=p$, because the ratio $|(\gamma^{-1})'x|/(1+|\gamma^{-1}x|^2)$ is always bounded in $B(p,Ch(p))-\{p\}$ and we can extend it continous to $p$.
\end{proof}

Recall $\Delta_{0}:=\Delta_{\infty}=\Delta_{p_1}$.

\begin{lem}\label{lem:quasi-gammainfinity}
	There exists $C>0$ such that for any Borel set $E$ of diameter less than the diameter of $\Delta_{0}$ and $\gamma_1\in\Gamma_{\infty}$, we have that for any $x\in E$
	\[\frac{\mu(\gamma_1E)}{\mu(E)}\in (1/C,C) \left(\frac{1+| x|^2}{1+|\gamma_1 x|^2}\right)^{\delta}. \]
\end{lem}

\begin{proof}
	Due to the derivative of $\gamma_1$ and the quasi-invariance of PS measure, we obtain
	\[\mu(\gamma_1 E)=\int_{E}\left(\frac{1+| x|^2}{1+|\gamma_1 x|^2}\right)^{\delta}\dd\mu(x). \]
	Now for any $x,y$ in $E$, we have
	\begin{equation}
	\label{quasi estimate 1}
	\frac{1+|x|^2}{1+|y|^2}=1+\frac{(|x|-|y|)(|x|+|y|)}{1+|y|^2}\leq C', 
	\end{equation}
	with $C'>1$ only depending on the diameter of $E$. The same argument also gives the same upper bound for $(1+|y|^2)/(1+|x|^2)$. The set $\gamma_{1}E$ is a set with the same diameter as $E$. So we also have
	\begin{equation}
	\label{quasi estimate 2}
	\frac{1+|\gamma_1x|^2}{1+|\gamma_1y|^2}\in(1/C',C') 
	\end{equation}
	for any $x,y $ in $E$. The proof is complete by applying \eqref{quasi estimate 1} and \eqref{quasi estimate 2} to the formula of $\mu(\gamma_1 E)$.
\end{proof}

\begin{lem}\label{lem:rre}
	There exist constants $c>0$ and $C>1$ such that for every parabolic fixed point $p\neq \infty$, if $r\leq h(p)/C$, then
	\[\mu(B(p,r)-B(p,r/\sqrt{e}))\geq c\mu(B(p,r)). \]
\end{lem}
\begin{proof}
	Consider $p\in \overline{\Delta}_{0}$. Indeed, since $\Gamma_\infty\overline{\Delta}_{0}$ covers the intersection $\R^d\cap \Lambda_{\Gamma}$, we can always find a $\gamma_1$ in $\Gamma_{\infty}$ such that $\gamma_1 p\in\overline{\Delta_0}$. Then applying Lemma \ref{lem:quasi-gammainfinity} to $E=B(p,r)$ and $E=B(p,r)-B(p,r/\sqrt{e})$, we have	
	\[\frac{\mu(B(p,r)-B(p,r/\sqrt{e}))}{\mu(B(p,r))}\approx \frac{\mu(B(\gamma_1p,r)-B(\gamma_1p,r/\sqrt{e}))}{\mu(B(\gamma_1p,r))}. \]
	
	We only need to give a lower bound to $\mu(B(p,r)-B(p,r/\sqrt{e}))$ and then use \eqref{equ:bpr} to obtain Lemma \ref{lem:rre}.
	
	Assume $p\in \overline{\Delta_0}$ is of rank $k$. Consider the case when $p=\gamma \infty$ with $\gamma\in \Gamma$ the representation of $p$. 
	We claim that there exists a constant $C>1$ such that for $\gamma_1\in \Gamma_{\infty}$, 
	with $\gamma \gamma_1\Delta_0\subset B(p,r)-B(p,r/\sqrt{e})$, we have
	\begin{equation}\label{equ:gammah}
	\mu(\gamma \gamma_1\Delta_0)\gg \frac{h(p)^{\delta}}{(d(\gamma_1\Delta_0,x_p)+C)^{2\delta}}.
	\end{equation}
	Proof of the claim: By Lemma~\ref{lem:annulusquasi}, we have $\mu(\gamma \gamma_1\Delta_0)\approx h(p)^\delta\mu(\gamma_1\Delta_0)$. By Lemma \ref{lem:quasi-gammainfinity}, we have $$\mu(\gamma_1\Delta_0)\approx \left(\frac{1+| x|^2}{1+|\gamma_1 x|^2}\right)^{\delta}\mu(\Delta_0), $$
	for any $x\in\Delta_{0}$.
	Now, since $|\gamma_1x|\leq d(\gamma_1\Delta_{0},x_p)+|x_p|+C'$, so
	$$\mu(\gamma_1\Delta_0)\geq \mu(\Delta_{0}) /(d(\gamma_1\Delta_0,x_p)+C)^{2\delta}$$ for some constant $C>1$.
	
	By Lemma~\ref{lem:explicit}, we have $\gamma^{-1}(B(p,r)-B(p,r/\sqrt e))=B(x_p,\sqrt eh(p)/r)-B(x_p,h(p)/r)$. Let $C'=\operatorname{diam}(\Delta_0)$. Let $\mathbb{R}^k$ be the subspace described in Lemma~\ref{lem:biberbach}.
	For a set $E$ in $\R^d$, we define $\operatorname{Vol}_{\R^k}(E)$ as $\operatorname{Vol}(E\cap\R^k)$. Since $x_p\in\Lambda_{\Gamma}\cap\R^d$ has bounded distance to $\R^k$, the number of $\gamma_1\Delta_0$'s in such region is at least 
	\begin{equation}\label{equ:volest}
	\operatorname{Vol}_{\mathbb{R}^k}\left(B\left(x_p,\sqrt{e}h(p)/r-C'\right)-B\left(x_p,h(p)/r+C'\right)\right)/\operatorname{Vol}_{\mathbb{R}^k}(\Delta_0)\gg h(p)^kr^{-k}.
	\end{equation}
	Then \eqref{equ:volest} and \eqref{equ:gammah} imply
	\begin{align*}
	\mu(B(p,r)-B(p,r/\sqrt{e}))\geq \sum_{\gamma_1\Delta_0\subset B(x_p,\sqrt eh(p)/r)-B(x_p,h(p)/r)}\mu(\gamma \gamma_1\Delta_0)\gg h(p)^{k-\delta} r^{2\delta-k}.
	\end{align*}
	
	Consider general case when $p=\gamma p_i$ with $\gamma\in \Gamma$ the representation of $p$. We estimate the measure $\mu(\gamma \gamma_1 \Delta_{p_i})$ for any $\gamma_1\in \Gamma_{p_i}$ satisfying $\gamma \gamma_1\Delta_{p_i}\subset B(p,r)-B(p,r/\sqrt{e})$. Using ~\eqref{conjugation}, we have
	\begin{equation*}
	\mu(\gamma \gamma_1\Delta_{p_i})\approx \mu_{\Gamma_i} (g_i \gamma \gamma_1 \Delta_{p_i}),
	\end{equation*} 
	where $\Gamma_i:=g_i\Gamma g_i^{-1}$. Lemma~\ref{lem:bilip} yields
	\begin{equation}
	\label{measure inclusion}
	g_i\gamma \gamma_1 \Delta_{p_i}\subset g_i(B(p,r)-B(p,r/\sqrt{e}))\subset B(g_ip, Cr)-B(g_ip,r/(C\sqrt{e})).
	\end{equation}
	So we can use the argument for the previous case to obtain
	\begin{equation}
	\label{fundamental domain measure}
	\mu_{\Gamma_i}(g_i\gamma \gamma_1\Delta_{p_i})\approx r^{2\delta} h(g_i p)^{-\delta}. 
	\end{equation}
	
	Then we count the number of $\gamma\gamma_1\Delta_{p_i}$'s in $B(p,r)-B(p,r/\sqrt{e})$. It equals the number of $g_i\gamma \gamma_1\Delta_{p_i}$'s in $g_i(B(p,r)-B(p,r/\sqrt{e}))$. The map $g_i$ maps $B(p,r)$ and $B(p,r/\sqrt{e})$ to two spheres and the distance between $g_i B(p,r)$ and $g_i B(p,r/\sqrt{e})$ is at least $(1-1/\sqrt{e})r/C$. The map $g_i\gamma^{-1}g_i^{-1}$ maps $g_i B(p,r)$ and $g_i B(p,r/\sqrt{e})$ to two spheres and let $R$ and $p'$ be the radius and the center of the outer sphere respectively. Using \eqref{measure inclusion} and Lemma~\ref{lem:explicit}, we have 
	\begin{equation}
	\label{outer radius}
	R\in (h(g_i p)/(Cr), C\sqrt{e}h(g_i p)/r).
	\end{equation}
	 For every $x\in B(g_ip,Cr)-B(g_ip, r/(C\sqrt{e}))$, we have $|(g_i\gamma^{-1}g_i^{-1})'(x)|\in (h(g_i p)/(C^2r^2), C^2eh(g_i p)/r^2)$. So the distance between $g_i\gamma^{-1}B(p,r)$ and $g_i\gamma^{-1}B(p,r/\sqrt{e})$ is at least $(1-1/\sqrt{e}) h(g_i p)/(C^3r)$. This distance estimate together with \eqref{outer radius} implies there exists some constant $c\in (0,1)$ such that 
	 \begin{equation*}
	 g_i\gamma^{-1} (B(p,r)-B(p,r/\sqrt{e}))\supset B(p', R)-B(p',cR).
	 \end{equation*}
	The number of $g_i\gamma_1\Delta_{p_i}$ in $g_i\gamma^{-1}(B(p,r)-B(p,r/\sqrt{e}))$ is at least
	\begin{equation}
	\label{number of fundamental domain}
	\operatorname{Vol}_{\mathbb{R}^k}\left(B\left(p',R-C''\right)-B\left(p',cR+C''\right)\right)/\operatorname{Vol}_{\mathbb{R}^k}(g_i\Delta_{p_i})\gg R^k\gg h(g_i p)^k r^{-k},
	\end{equation}
	where $C''=\operatorname{diam}(g_i\Delta_{p_i})$. A lower bound for $\mu(B(p,r)-B(p,r/\sqrt{e}))$ can be obtained using Lemma~\ref{lem:height}, \eqref{fundamental domain measure} and \eqref{number of fundamental domain}.
	\end{proof}


\subsection{Friendliness of PS measure}
\label{sec:friendliness}
For any $r>0$, set
\begin{equation}
\label{equ:def neighborhood1}
N_r(\Delta_0):=\{x\in \Delta_0:\,\ d(x,\partial \Delta_0)\leq r \}.
\end{equation}
\begin{lem}\label{lem:boud}
	There exist
	$0<\epsilon_0<1$ such that for all $0<\epsilon<\epsilon_0$ there exists $\lambda=\lambda(\epsilon)\in(0,1)$ for all $r<1$
	\begin{equation}\label{equ:bou}
	\mu(N_{\epsilon r}(\Delta_0))\leq \lambda\mu(N_r(\Delta_0)). 
	\end{equation}
	Moreover, the constant $\lambda(\epsilon)$ tends to zero as $\epsilon$ tends to zero.
\end{lem}
Recall from Section~\ref{sec:cusps} that $\Delta_0=B_Y(C)\times \Delta_0'$. Recall that the set $\Delta_{0}'$ is a parallelotope. Let $l'$ be a facet of $\Delta_0'$ and $l=B_Y(C)\times l'$. Let $\gamma$ be the element in $\Gamma_{\infty}$ identifying $l'$ with the opposite facet $l''$ so $\gamma$ also identifies $B_Y(C)\times l'$ with $B_Y(C)\times l''$. Set
\begin{equation*}
N_r(l):=\{x\in \Delta_0\cup \gamma^{-1}\Delta_0:\,d(x,l)\leq r\}.
\end{equation*}
Lemma \ref{lem:boud} is deduced from the following lemma.
\begin{lem}\label{lem:part}
	There exist $0<\epsilon,\lambda<1$ such that for all $r<1$
	\begin{equation*}
	\mu(N_{\epsilon r}(l))\leq \lambda \mu(N_r(l)).
	\end{equation*}
\end{lem}

\begin{proof}[\textbf{Proof of Lemma~\ref{lem:boud}}]
	Assume that $\infty$ is a rank $k$ cusp. If $\infty$ is not a cusp of maximal rank, then note that $(\partial B_Y(C))\times \Delta_0'=\{|y|=C\}\times \Delta_0'$ does not intersect $\Lambda_{\Gamma}$. A small neighborhood of this boundary has zero PS measure. Therefore, we just need to consider the neighborhood of $l$'s. Using Lemma~\ref{lem:part}, we obtain
	\begin{equation*}
	\mu(N_{\epsilon r}(\Delta_0))\leq \sum_{l}\mu(N_{\epsilon r}(l))\leq \lambda \sum_{l}\mu(N_{r}(l)).
	\end{equation*}
	Each $N_r(l)$ is covered by $\Delta_{0}$ and one of its translates $\gamma\Delta_{0}$.
	By Lemma~\ref{lem:quasi-gammainfinity},
	there exists $C'>0$ such that
	\begin{equation*}
		\lambda \sum_{l}\mu(N_{r}(l))\leq \lambda C'2k \mu(N_r(\Delta_0)).
	\end{equation*}

We can replace $\epsilon$ by $\epsilon^n$ and using Lemma \ref{lem:part} repeatedly, which will yield an arbitrary small $\lambda$ in Lemma~\ref{lem:boud}. 
\end{proof}
	
	

\begin{proof}[Proof of Lemma~\ref{lem:part}]
	
	The proof is similar to the argument of using Lemma 3.11 to deduce Lemma 3.10 in~\cite{DFSU}. Let $L$ be the hyperplane containing $l$ and $N_r(L)$ be the $r$-neighborhood of $L$. ~\cite[Lemma 3.11]{DFSU} is stated in spherical metric but locally spherical metric is equivalent to the euclidean metric. So ~\cite[Lemma 3.11]{DFSU} implies that there exists $\epsilon>0$ such that 
	for every $\xi\in E:= \Lambda_\Gamma\cap N_{\epsilon r}(l)$, there exists $0<\rho_{\xi}<1$ satisfying
	\begin{equation}\label{equ:nlr}
	\mu(B(\xi,\rho_{\xi})\cap (N_r(L)-N_{\epsilon r}(L)))\geq c\mu (B(\xi,\rho_{\xi})),
	\end{equation}
	where $0<c<1$ is a constant only depending on $\Gamma$. The family $\{B(\xi,\rho_{\xi})\}_{\xi\in E}$ forms a covering of $E$.
	It follows from Vitali covering Lemma that there exists a disjoint subcollection $\{B(\xi,\rho_{\xi}) \}_{\xi\in I}$ with $I\subset E$ countable, such that 
	\[\cup_{\xi\in I} B(\xi,5\rho_{\xi})\supset \cup_{\xi\in E} B(\xi,\rho_{\xi})\supset E. \]
	The set $B(\xi,\rho_{\xi})\cap (N_r(L)-N_{\epsilon r}(L))$ may not be contained in $N_r(l)-N_{\epsilon r}(l)$, but we can cover it by some translations of $N_r(l)-N_{\epsilon r}(l)$.
	By elementary computation, we can use no more than $k_0$ number of elements $\gamma_j$'s in $\Gamma_{\infty}$ with $k_0$ depending on $\Delta_0$ such that
	\begin{equation*}
	\cup_{j}\gamma_j(N_r(l)-N_{\epsilon r}(l))\supset B(\xi,\rho_\xi)\cap (N_r(L)-N_{\epsilon r}(L)). 
	\end{equation*}
	Using this inclusion, Lemma~\ref{lem:quasi-gammainfinity} and disjointness of $B(\xi,\rho_{\xi})$'s for $\xi\in I$, we obtain
	\begin{align*}
	&C'k_0\mu(N_r(l)-N_{\epsilon r}(l))\geq \mu(\cup_{j}\gamma_j(N_r(l)-N_{\epsilon r}(l)))
	\geq \sum_{\xi\in I}\mu(B(\xi,\rho_\xi)\cap(N_r(L)-N_{\epsilon r}(L))
	\end{align*}
	Using \eqref{equ:nlr} and doubling property in Proposition~\ref{double}, we have
	\begin{equation*}
		\sum_{\xi\in I}\mu(B(\xi,\rho_\xi)\cap(N_r(L)-N_{\epsilon r}(L))
		\geq c\sum_{\xi\in I}\mu(B(\xi,\rho_{\xi}))
		\geq c\epsilon'\sum_{\xi \in I} \mu(B(\xi,5\rho_{\xi}))\geq c\epsilon'\mu(N_{\epsilon r}(l)),
	\end{equation*}
	Combining the above two formulas, we conclude that there exists $0<\lambda<1$ such that
	\begin{equation*}
	\mu (N_{\epsilon r}(l))\leq \lambda\mu(N_r(l)). \qedhere
	\end{equation*}
\end{proof}

\section{Coding of limit set}\label{sec:code}
In this section, we construct the coding and prove Proposition~\ref{prop:coding}, Lemma~\ref{lem:uni}, and Lemma~\ref{lem:l1}. At first reading, the reader might want to concentrate on the case when there is one cusp in the manifold, i.e., $\mathbf{P}=\{p_1\}$ and the coordinate change of transformation $g_1$ is the identity. This will significantly reduce the notational burden while not sacrificing too much of the main results.

\subsection{Coding for local regions}
We introduce ``flower" $J_p$, the building block for the coding. Actually, $J_p$'s are almost the union of a countable subcollection of open subsets $\Delta_j$ in the coding. The advantage of considering $J_p$ is that $J_p$ has a clean boundary which makes it possible to estimate the measure.

We first consider the case when {\textbf{$p=\gamma\infty$ is a parabolic fixed point in $\Delta$ with $\gamma\in \Gamma$ the representation of $p$ and $x_{p}=\gamma^{-1}\infty$}}. Let $\eta\in (0,1)$. We define the set $J_{p,\eta}$ as follows. By Lemma~\ref{lem:explicit}, we have
$$\gamma^{-1}B(p,\eta h(p))=B(x_p,1/\eta)^c. $$

Suppose that $\infty$ is a parabolic fixed point of maximal rank. Then $\mathbb{R}^d\subset \partial \mathbb{H}^{d+1}$ is tessellated by the translations of $\overline\Delta_0$. Take $R_{p,\eta}$ to be the smallest parallelotope tiled by the translations of $\overline\Delta_0$ such that it contains $B(x_p,1/\eta)$. Let 
\begin{equation}
J_{p,\eta}=\gamma R_{p,\eta}^c,\\
\end{equation}
\begin{equation}
N_p=\{\gamma_1\in \Gamma_{\infty}:\, \gamma_1\Delta_0\subset R_{p,\eta}^c\}
=\{\gamma_1\in \Gamma_{\infty}:\,\gamma \gamma_1\Delta_0\subset J_{p,\eta}\}.
\end{equation}

Suppose $\infty$ is a parabolic fixed point of rank $k$ in general. Let $Z$ be the affine subspace in $\partial \mathbb{H}^{d+1}$ described in Lemma~\ref{lem:biberbach} where elements in $\Gamma_{\infty}$ act as translations, and $\Delta_{0}=B_{Y}(C)\times \Delta_0'$. So $Z$ is tessellated by the translations of $\overline{\Delta_0'}$. 
 Take $R_{p,\eta}$ in $Z$ to be the smallest parallelotope tiled by the translations of $\overline{\Delta_0'}$ such that $B_Y(2/\eta)\times R_{p,\eta}$ contains $B(x_p,1/\eta)$. Set \begin{equation}\label{flower}
J_{p,\eta}=\gamma (B_Y(2/\eta)\times R_{p,\eta})^c\subset B(p,\eta h(p)),
\end{equation}
\begin{equation}
\label{flower group}
N_p=\{\gamma_1\in\Gamma_{\infty}:\, \gamma_1\Delta_0\subset (B_Y(2/\eta)\times R_{p,\eta})^c \}=\{\gamma_1\in \Gamma_{\infty}:\,\gamma \gamma_1\Delta_0\subset J_{p,\eta}\}.
\end{equation}
 
The set $J_{p,\eta}$ enjoys the following property
\begin{equation}\label{flower1}
J_{p,\eta}\cap\Lambda_\Gamma=\gamma\left(\underset{\gamma_1\in N_p}{\cup}\gamma_1\left(\overline{\Delta_0}\cap\Lambda_{\Gamma}\right)\right),
\end{equation}
that is to say, the countable disjoint union $\underset{\gamma_1\in N_p}{\sqcup}\gamma\gamma_1\Delta_0$ is a conull set in $J_{p,\eta}$. \textbf{The open sets $\gamma \gamma_1 \Delta_0$ with $\gamma_1\in N_p$ are the ones described in Proposition \ref{prop:coding} in $J_{p,\eta}$}, and on each $\gamma \gamma_1\Delta_0$, the expanding map $T$ is given by $T|_{\gamma \gamma_1\Delta_0}=(\gamma \gamma_1)^{-1}$.

We also have the following distance relation:
\begin{equation}\label{flower3}
	d((\gamma \gamma_1)^{-1}\infty,\Delta_0)=d(\gamma_1^{-1}x_p,\Delta_0)\geq 1/\eta \,\,\,\text{for any}\,\,\, \gamma_1\in N_p.
\end{equation}

\begin{figure}
	\begin{center}
		\begin{tikzpicture}[scale=2.5]
		\path[fill=olive] (3.5,0) rectangle (5.5,1.5);
		\draw[fill= white, draw=red, very thick] (3.75,0.25) rectangle (5.25,1.25);
		\draw (0,0) rectangle (2,1.5);
		\path[fill=olive, draw=red, very thick] (1,0.375)circle [radius=0.125];
		\path[fill=olive, draw=red, very thick] (1,0.625)circle [radius=0.125];
		\path[fill=olive, draw=red, very thick] (0.875,0.5)circle [radius=0.125];
		\path[fill=olive, draw=red,very thick] (1.125,0.5)circle [radius=0.125];
		\path[fill=olive, draw=none] (1,0.375)circle [radius=0.125];
		\path[fill=olive, draw=none ] (1,0.625)circle [radius=0.125];
		\path[fill=olive, draw=none ] (0.875,0.5)circle [radius=0.125];
		\path[fill=olive, draw=none ] (1.125,0.5)circle [radius=0.125];
		\node[above] at (1,0.75) {$p=\gamma\infty$};
		\path[fill=black] (1,0.5) circle [radius=0.01];
		\draw [thick, <-] (2.25,0.5) -- (3.25,0.5);
		\node[above] at (2.75, 0.5) {$\gamma$};
		\node[above] at (0.2, 1.2) {$\Delta_{0}$};
		\path[ draw=blue, very thick] (4.5, 0.75)circle [radius=0.45];
		\path[fill=black] (4.5,0.75) circle [radius=0.01];
		\node[above] at (4.5, 0.75) {$x_p$};
		\path[ draw=blue, very thick] (1, 0.5)circle [radius=0.3];
		\end{tikzpicture}
	\end{center}
	\caption{The shaded region on the left hand side is $J_{p,\eta}$, which is the image under the action of $\gamma$ on the complement of the white rectangle on the right hand side. }\label{fig:flower}
\end{figure}


\begin{lem}\label{lem:jp}
	There exists $0<c_{\four}<1$ such that for any $\eta\in(0,1)$
	\[B(p,c_{\four}\eta h(p))\subset J_{p,\eta}\subset B(p,\eta h(p)), \]
	\[B(x_p,1/\eta)\subset (\gamma^{-1} J_{p,\eta})^c\subset B(x_p, 1/(c_{\four}\eta)). \]
\end{lem}
\begin{proof}
	Due to the compactness of $\Delta_0$, there exists $c_{\four}$ such that $(\gamma^{-1}J_{p,\eta})^c=(B_Y(2/\eta)\times R_{p,\eta})\subset B(x_p,1/(c_{\four}\eta))$. The first statement can be deduce from the second using Lemma \ref{lem:explicit}.
\end{proof}

\textbf{In the following, we abbreviate $J_{p,\eta}$ to $J_p$.} For $r>0$, let 
\begin{align}
\label{equ:def neighborhood 2}
&N_r(\partial J_p):=\{x\in J_p^c:\, d(x,\partial J_p)\leq r \},\\
&N_r(\partial\gamma^{-1} J_p):=\{x\in (\gamma^{-1}J_p)^c:\, d(x,\partial \gamma^{-1}J_p)\leq r \}.\nonumber
\end{align}

\begin{lem}\label{lem:jpr}
	Fix $C>1$. For every $0<\eta<1/4C^2$, there exists $0<c=c(\eta)<1$ depending on $\eta$ such that for any $r<h(p)$, 
	\begin{equation}
	\label{recdouble}
	\mu(N_{C\eta r}(\partial J_p))\leq c\mu(N_r(\partial J_p)). 
	\end{equation}
	Moreover, $c(\eta)$ tends to zero as $\eta$ tends to zero. 
\end{lem}
The proof of Lemma \ref{lem:jpr} will be given in the appendix.

We consider the general case. \textbf{Let $p$ be any parabolic fixed point in $\Delta $. Write $p=\gamma p_i$ with $\gamma\in \Gamma$ the representation of $p$.} If $p_i=\infty$, let $J_p$ and $N_p$ be defined as (\ref{flower}) and \eqref{flower group} respectively. Otherwise, we use the following commutative diagram to define $J_p$:
 \[ \begin{tikzcd}
\H^{d+1} \arrow{r}{g_i} \arrow[swap]{d}{\Gamma} & \H^{d+1} \arrow{d}{g_i \Gamma g_{i}^{-1}} \\%
\H^{d+1} \arrow{r}{g_i}& \H^{d+1}.
\end{tikzcd}
\]
Note that $g_ip=(g_i\gamma g_i^{-1})\infty\in g_i\Delta$. So for the action of $g_i\Gamma g_i^{-1}$ on $\partial \mathbb{H}^{d+1}$, we can define $J_{i,p}$ for the parabolic fixed point $g_ip$ as (\ref{flower}). Set
\begin{equation}
\label{def:jip}
J_p:=g_i^{-1}J_{i,p},\ \ 
N_p=\{\gamma_1\in \Gamma_{p_i}:\,\gamma \gamma_1\Delta_{p_i}\subset J_{p}\},
\end{equation}
where $\Gamma_{p_i}$ is a subgroup of $\Gamma$ defined in \eqref{pistabilizer}. The set $J_p$ enjoys the property
\begin{equation}
\label{flower2}
 J_{p}\cap\Lambda_\Gamma=\underset{\gamma_1\in N_p}{\sqcup}\left(\gamma\gamma_1\overline{\Delta}_{p_i}\cap\Lambda_\Gamma\right).
 \end{equation}
 On each set $\gamma\gamma_1\Delta_{p_i}$, we have an expanding map given by $(\gamma\gamma_1)^{-1}$ which maps this set to $\Delta_{p_i}$.

The following lemma is an analog of Lemma \ref{lem:jp}.
\begin{lem}
\label{equ:jp}
There exists some constant $C_{\five}>1$ such that for any $\eta\in(0,1)$
\begin{align*}
B(g_ix_p,1/\eta)&\subset (g_i\gamma^{-1}J_p)^c\subset B(g_ix_p,1/(c_{\four}\eta)),\\
B\left(p,\eta h(p)/C_{\five}\right)\subset g_i^{-1}B\left(g_ip, c_{\four}\eta h(g_i p)\right)&\subset J_p\subset g_i^{-1}B\left(g_ip,\eta h(g_i p)\right)\subset B\left(p, C_{\five}\eta h(p)\right),
\end{align*}
where $x_p=\gamma^{-1}p_i$.
\end{lem}

\begin{proof}
We use Lemma~\ref{lem:jp},~\ref{lem:height} and~\ref{lem:bilip} to obtain the lemma. 
\end{proof}


\subsection{Coding for $\Delta_0$}
\label{coding procedure}

The construction of the coding for the whole $\Delta_0$ is by induction. Let $\Omega_0:=\Delta_{0}$.


Since we already have a nice coding for flowers $J_p$, the idea is to find a collection of pairwise disjoint flowers $J_p$ to cover the intersection $\Lambda_{\Gamma}\cap\Omega_0$. Here $p$ is a parabolic fixed point. From the construction of $J_p$ (Lemma \ref{lem:jp}), we know that the higher the height $h(p)$ is, the larger $J_p$ is. So we start with parabolic fixed points with large heights. We want that the full flower $J_p$ is inside $\Omega_0$. Hence we only take parabolic fixed points $p$ away from the boundary of $\Omega_0$.

Take
\begin{align*}
&h_n=e^{-n},\\
& \eta\in (0,1)\text{ a sufficiently small constant to be specified at the end of the proof of Proposition \ref{keylemma}}. 
\end{align*}
All the constants appearing later will be independent of $\eta$ unless we state it explicitly.

\begin{itemize}
\item 
For $n\in\N$, let 
\begin{equation}
\label{good parabolic fixed points}
P_{n+1}=\{p\in\calP:\,\ \eta h(p)\in(h_{n+1},h_n],\ B(p,h_n/(4\eta))\subset\Omega_{n} \}.
\end{equation}

\item For any $p\in P_{n+1}$, write $p=\gamma p_i$ with $\gamma\in \Gamma$ the representation of $p$. Construct $J_p$ and $N_p$ as in the previous section.

\item Set
\begin{align*}
&\Omega_{n+1}=\Omega_n-D_{n+1}=\Omega_n-\cup_{p\in P_{n+1}}J_{p}.
\end{align*}

\end{itemize}


Using the definition of $J_p$, Lemma~\ref{equ:jp} and the separation property (Lemma \ref{lem:pdistance}), it can be shown that the sets $J_p$'s with $p\in P_n$ and $n\in \mathbb{N}$ are mutually disjoint (Lemma \ref{lem:separation}) and inside $\Delta_{0}$. In Proposition \ref{keylemma}, it will be shown that the union $\cup_n\cup_{p\in P_n}J_p$ is conull in $\Delta_0$ with respect to the PS measure $\mu$. By \eqref{flower2}, the countable disjoint union
\[ \bigcup_{n\in\N}\bigcup_{p=\gamma p_i\in P_n}\bigcup_{\gamma_1\in N_p}\gamma\gamma_1\Delta_{p_i} \]
is also conull in $\Delta_0$ with respect to PS measure. On each set $\gamma\gamma_1\Delta_{p_i}$ we have an expanding map given by $(\gamma\gamma_1)^{-1}$ which maps this set to $\Delta_{p_i}$. For one cusp case, these are the countable collection of disjoint open subsets and the expanding map. When there are multi-cusps, this is the first step to construct the coding and the rest will be provided in Section \ref{sec:exptail} and \ref{sec:codmulti}.

\bigskip{}
The main result of this section is the following proposition.
\begin{prop}\label{keylemma}
	There exist $\epsilon_0>0$ and $N>0$ such that for all $n>N$, we have
	\[\mu(\Omega_n)\leq (1-\epsilon_0)^n. \]
\end{prop}

For one cusp case, this yields Proposition \ref{prop:coding} (1). Moreover, the exponential tail property \eqref{sum} will follow from Proposition \ref{keylemma} rather directly and it will be proved in Section~\ref{sec:exptail}. To prove this proposition, we need a lot of preparations and we postpone its proof to the end of Section~\ref{sec:energy}.

\subsection{Separation}
\label{sec:sep}

\begin{lem}[Separation property]
\label{lem:pdistance}
For any two different parabolic fixed points $p, p'$, we have 
\begin{equation*}
d(p,p')>\sqrt{h(p) h(p')}.
\end{equation*}
\end{lem}
\begin{proof}
	Let $x,x'$ be the euclidean centers of $H_p$ and $H_{p'}$ repectively. By disjointness of horoballs, then $d_E(x,x')\geq (h(p)+h(p')/2$. By the Pythagoras' theorem, we obtain
	\[d(p,p')\geq \sqrt{d_E(x,x')^2-((h(p)-h(p'))/2)^2}\geq \sqrt{h(p)h(p')}. \qedhere \]
\end{proof}
This property plays a key role in the construction of the coding and the proof of Proposition~\ref{keylemma}.

\begin{lem}
\label{lem:separation}
	If $\eta<1/(4eC_{\five})$, then
	the sets $J_p$'s with $p\in P_n$ and $n\in \mathbb{N}$ are mutually disjoint, and the distance between any two connected components of $\partial\Omega_n$ is strictly greater than $h_n/(2\eta)$.
\end{lem}

\begin{proof}
Notice that $\Omega_{n}=\Omega_{n-1}-\cup_{p\in P_n} J_p$. By induction, we only need to prove two cases.

Case 1: We consider $J_p,J_{p'}$ with distinct $p,p'$ in $P_n$. Using Lemma~\ref{equ:jp} and Lemma~\ref{lem:pdistance}, we obtain
\begin{align*}
\sqrt{h(p)h(p')}-C_{\five}\eta (h(p)+h(p'))\geq \frac{h_n}{\eta}-2C_{\five} h_{n-1}\geq \frac{h_n}{2\eta}.
\end{align*}
Therefore, two sets $J_p, J_{p'}$ are disjoint and the distance between them is at least $h_n/(2\eta)$.

Case 2: We consider $J_p$ with $p\in P_n$ and $\Omega_{n-1}$. By \eqref{good parabolic fixed points} and Lemma~\ref{equ:jp}, the distance between $J_{p}$ and $\partial\Omega_{n-1}$ is also greater than 
\begin{align*}
\frac{h_{n-1}}{4\eta} -C_{\five}\eta h(p) \geq \frac{h_{n-1}}{4\eta} -C_{\five}h_{n-1}>\frac{h_n}{2\eta}.
\end{align*}
Hence $J_p$ is inside $\Omega_{n-1}$.

By induction, for different connected components of $\Omega_{n}$, their distance is at least $h_{n}/(2\eta)$.
\end{proof}

\subsection{Equivalence classes in $Q_n$}
\label{sec:equivalence classes}

\paragraph{Motivation of equivalence classes.}
 We introduce the notion of equivalence classes to attain the exponential tail property. 
Let's start with some definitions.
The visual map $\pi:\operatorname{T}^1(\mathbb{H}^{d+1})\to \partial \mathbb{H}^{d+1}$ is defined by
\begin{equation*}
\pi(x)=\lim_{t\to \infty} \mathcal{G}_t(x),
\end{equation*}
which maps $x$ to the forward endpoint in $\partial \mathbb{H}^{d+1}$ of the geodesic defined by $x$. 

Recall that we fix $p_1$ as $\infty$ and $H_{\infty}$ is the horoball based at $\infty$ given by $\mathbb{R}^d\times \{x\in \mathbb{R}:\,x>1\}$. Let $\widetilde{H_{\infty}}$ be the corresponding unstable horosphere. 
More precisely, let $x_o$ be the unit tangent vector based at $(0,1)\in \R^d\times \R$ with $\pi(x_o)=0$. Then $\widetilde{H_{\infty}}$ is the set of $x$ in $\operatorname{T}^1(\H^{d+1})$ such that $d(\calG_{-t}x_o,\calG_{-t}x)\rightarrow 0$ as $t\rightarrow +\infty$. For a set $E\subset \partial\H^{d+1}-\{\infty\}\simeq \R^d$, let
\begin{align*}
	\widetilde{E}=\pi|_{\widetilde{H_\infty}}^{-1}(E)
\end{align*}
 be the preimage of $E$ under the map $\pi$ restricted on $\widetilde{H_\infty}$.

Let 
\begin{align*}
&H_p(\eta)\,\,\, \text{be the horoball based at}\,\,\, p\,\,\, \text{with height equal to}\,\,\, \eta h(p),\\
& \mathcal C_\eta=\cup_{p\in \calP}\Gamma \T^1(H_p(\eta)).
\end{align*}
Then $\calC_\eta$ is the lift of the unit tangent bundle over proper horocusps of $M$. 

At the time $n$, the set $\calG_n\widetilde{\Omega_n}$ is a large sheet with many holes, consisting of ``flowers" of different sizes, corresponding to different $\calG_n \widetilde{J_p}$. Here is the source of the exponential tail: for $x\in \calG_n\widetilde{\Omega_n} $ with $\pi(x)\in\Lambda_{\Gamma}\cap\Omega_0$ and $x$ not in the cusp region $\calC_\eta$, the recurrence of geodesic flow implies that there exits a new flower $\calG_n\widetilde{ J_p}$ inside the neighborhood of $x$ with size bounded below (Lemma \ref{lem:parabolic}). So a fixed proportion of the neighborhood of $x$ will be coded in a fixed time (Lemma \ref{lem:sub1ii}).

Lemma \ref{lem:parabolic} doesn't apply to the set $\calG_n\widetilde{\Omega_n}\cap \calC_\eta$. 
Points in the set $\pi(\calG_n\widetilde{\Omega_n}\cap \calC_\eta)$ are contained in the balls centered at certain parabolic fixed points. We want to argue that from step $n$ to step $n+1$, the points near the outer edge of the balls will escape the cusps. We illustrate the scheme of the proof. 
For $n\in\N$, define
\begin{equation*}
Q_{n+1}=\{p\in\calP:\, \eta h(p)\in (h_{n+1},h_n],\ B(p,\eta h(p))\cap \Omega_n\neq\emptyset,\ d(p,\partial \Omega_n)\leq h_n/(4\eta)\}.
\end{equation*}
Then we can cover $\calC_\eta\cap\calG_n\widetilde{\Omega_n}$ by the union of balls centered at parabolic fixed points in $Q_m$ with $m\leq n$ 
(Lemma \ref{lem:np}, \ref{lem:cuspBn}):
$$ \calC_\eta\cap\calG_n\widetilde{\Omega_n}=\calG_n(\cup_{m\leq n}\cup_{p\in Q_m}\widetilde{B}(p,r(p,n)))\cap \calG_n\widetilde{\Omega_n},$$
where $r(p,n)=\sqrt{\eta h(p)h_n}$. 
From step $n$ to step $n+1$, the part 
\[\calG_{n+1}(\widetilde{B}(p,r(p,n))-\widetilde{B}(p,r(p,n+1))) \]
will leave the cusp region $\calC_\eta$. We want a lower bound for the measure of $({B}(p,r(p,n))-{B}(p,r(p,n+1)))\cap\Omega_{n}$. The main difficulty is that $B(p,r(p,n))\cap \Omega_{n}$ may not be a full ball, in which case its PS measure is hard to estimate. 

We introduce the notion of equivalence classes to resolve this issue.
Consider a subset of $Q_n$:
\begin{equation*}
Q_{n}'=\{p\in Q_{n}:\,\ B(p,r(p,n))\cap\partial \Omega_{n-1}\neq\emptyset \}. 
\end{equation*}
Pick any $p\in Q_n'$. As the ball $B(p,r(p,n))\cap \Omega_{n-1}$ is not a full ball, we will pair it with another partial ball and use the doubling property of the PS measure. Notice that there is a unique component in $\partial \Omega_{n-1}$ closest to $p$ (Lemma \ref{lem:separation}).
 If it is $\partial \Omega_0$, note that $\Lambda_{\Gamma}\cap\R^d$ is covered by the translations of $\Delta_0$, so the symmetry property of these translations gives the point $p'$ to pair with $p$ (if $p$ is around the corners of $\partial \Omega_0$, we may need more than one point to pair with $p$). If it is some $\partial J_q$, write $q=\gamma^{-1}p_i$ with $\gamma^{-1}\in \Gamma$ the representation of $q$. We map $B(p,r(p,n))$ and $\partial J_q$ by $g_i\gamma $ and get a picture similar to the previous case. 
We find the paring point for $g_i\gamma p$ and map it back to get the one for $p$. The work lies in modifying the radius $r(p,n)$: $r(p,n)$ is defined to $\sqrt{\eta h_n h(p)}$ depending on $h(p)$, and it may happen that the horosphere attached to the pairing point of $p$ has a different height.

\begin{figure}
	\def\svgwidth{350bp}
	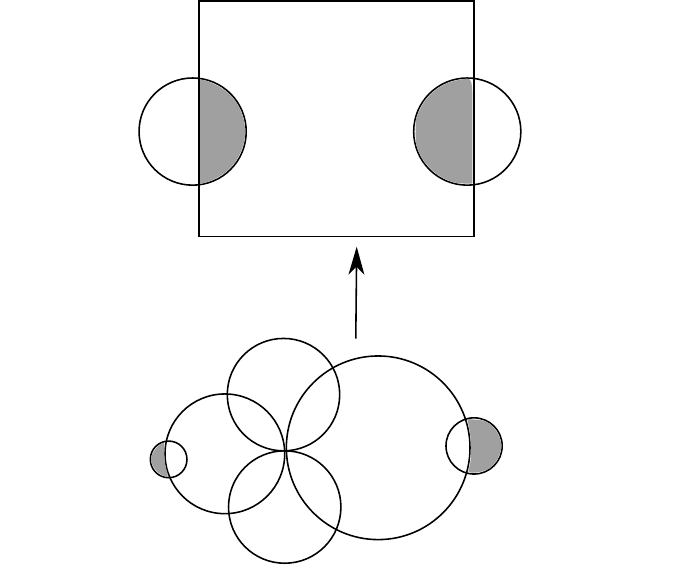
	\caption{Pairing partial balls: $q=\gamma^{-1}p_i$, $p'=g_i\gamma p$, $B_1=B(g_i\gamma p, \tilde r_{p,m}/C_{\six})$, $\gamma_1\in (g_i\Gamma g_i^{-1})_{\infty}$.}
\end{figure}


\bigskip{}
\paragraph{Finding the radius.}

\textbf{For Lemma~\ref{lem:rprq} - Lemma~\ref{lem:pjq}, we consider $p\in Q_n'$ such that the component in $\partial \Omega_{n-1}$ closest to $p$ is $\partial J_{q}$ with $q\in \cup_{l=1}^{n-1}P_l$. Write $q= \gamma^{-1}p_i$ for some $\gamma\in \Gamma$ the representation of $q$.} 

\begin{lem}
\label{lem:rprq}
There exists $C>1$ such that we have 
\begin{align*}
\eta h(p)\leq h_{n-1}\leq C\eta^3 h(q),\ \ 
 \frac{\eta h(g_i p)}{C}\leq h_{n-1}\leq C \eta^3 h(g_i q).
\end{align*}
\end{lem}

\begin{proof}

It follows from Lemma~\ref{lem:pdistance} that
\begin{equation*}
d(p,q)\geq \sqrt{h(p) h(q)}\geq \sqrt{h_{n-1} h(q)/(e\eta)}.
\end{equation*}
Meanwhile by Lemma~\ref{equ:jp}, we have 
\begin{equation*}
d(p,q)\leq d(p,\partial J_{q})+\max_{y\in\partial J_q}d(y, q)\leq r(n,p)+C_5\eta h(q)\leq h_{n-1}+C_{\five}\eta h(q)\leq (1+C_{\five})\eta h(q).
\end{equation*}
So the above two inequalities lead to first statement. The second statement follows easily from the first statement and Lemma~\ref{lem:height}.
\end{proof}

\begin{lem}
\label{lem:annulus}
There exists $C>1$ such that 
\begin{equation*}
B(g_ip, h(g_i p))\subset B(g_iq, C\eta h(g_i q))-B(g_iq, \eta h(g_i q)/C).
\end{equation*}
\end{lem}

\begin{proof}
For any $\xi\in \partial B(g_ip, h(g_i p))$, an upper bound for $d(\xi, g_iq)$ is given by
\begin{align*}
d(\xi, g_iq)\leq &d(\xi, g_ip)+d(g_ip,\partial J_{g_iq})+\max_{y\in\partial J_{g_iq} }d(y,g_iq)
\leq h(g_i p)+C h_{n-1}+\eta h(g_i q) \leq C\eta h(g_i q).
\end{align*}
A lower bound for $d(\xi, g_iq)$ is given by
\begin{align*}
d(\xi, g_iq)\geq &d(g_iq, \partial J_{g_iq}) -d(g_ip,\partial J_{g_iq})-h(g_i p)
\geq c_{\four} \eta h(g_i q) -Ch_{n-1}-h(g_i p) \geq c_{\four} \eta h(g_i q) -C\eta^2h(g_i q).
\end{align*}
Hence by taking $\eta$ sufficiently small, we reach the conclusion. 
\end{proof}

\begin{lem}
\label{lem:rgammaprp}
There exists $C>1$ such that we have
\begin{equation*}
\frac{ \h(g_i p)}{C\eta^2 \h(g_i q)}\leq \h(g_i \gamma p)\leq \frac{C\h(g_i p)}{\eta^2 \h(g_i q)}.
 \end{equation*}
\end{lem}

\begin{proof}
	
Apply Lemma~\ref{lem:rg1} to the horoball $H=H_{g_ip}$ based at $g_ip$ and the element $g=g_i\gamma g_i^{-1}$. Notice that $g_iq=(g_i\gamma g_i^{-1})^{-1}\infty=g^{-1}\infty$ and $h(g_i q)=h(H_{g_iq})=h(g^{-1}H_\infty)$.
We have
\[ \frac{h(g_i q)h(g_i p)}{d(g_iq,g_ip)^2+h(g_i p)^2}\leq h(g_i \gamma p)=h(gH)\leq \frac{h(g_i q)h(g_i p)}{(d(g_iq,g_ip)-h(g_i p) /2)^2}, \]

By Lemma~\ref{lem:rprq} and~\ref{lem:annulus}, we obtain the lemma.
\end{proof}

For any $m\in \N$, set
\begin{equation}
\label{equ:rpm}
\tilde{r}_{p,m}=\sqrt{\frac{h_mh(g_i \gamma p)}{\eta h(g_i q)}}.
\end{equation}
At the point $g_i\gamma p$, we will consider ball $B(g_i\gamma p, \tilde r_{p,m}/C_{\six})$, where $C_{\six}>1$ is a constant given in Lemma \ref{lem:pjq} such that $\tilde{r}_{p,m}/C_{\six}$ guarantees the equivalence classes we introduce later are well-defined. Once we have chosen the ball $B(g_i\gamma p, \tilde r_{p,m}/C_{\six})$, we will map it back by $(g_i\gamma)^{-1}$ to attain the ``correct" ball at $p$.


	By Lemma~\ref{lem:rgammaprp},~\ref{lem:height} and~\ref{lem:rprq}, we have 
\begin{equation}\label{equ:rpmless}
	\tilde{r}_{p,n}\leq \frac{C\sqrt{\eta \h_n\h(g_i p)}}{\eta^2\h(g_i q)}\leq \frac{C\h(g_i p)}{\eta \h(g_i q)}\leq C\eta h(g_i\gamma p).
\end{equation}


\begin{lem}\label{lem:pjq}
There exists $C_{\six}>1$ such that for any point $p'$, if $d(g_i\gamma p',g_i\gamma \partial J_q)\leq \tilde{r}_{p,n}/C_{\six}$, then $d(p',\partial J_q)\leq h_n$.
\end{lem}

\begin{proof}
	It follows from Lemma~\ref{lem:jp} applying $g_i\gamma g_i^{-1}$ that for any $C_{\six}>1$, if $d(g_i\gamma p', g_i\gamma \partial J_q)\leq \tilde{r}_{p,n}/C_{\six}$, then 
	$$g_i\gamma p'=g_i\gamma g_i^{-1}(g_ip')\in B(g_ix_q, 2/(c_{\four}\eta))-B(g_ix_q,1/(2\eta)),$$
	where $g_ix_q=g_i\gamma p_i=g_i\gamma g_i^{-1}\infty$.
 For any $x$ in the line segment between $g_i\gamma p'$ and $g_i\gamma \partial J_{q}$, we use Lemma \ref{lem:explicit} to obtain $|(g_i\gamma^{-1} g_{i}^{-1})'x|\leq 4\eta^2 h(g_i q)$. 
	By Lemma~\ref{lem:bilip} and \eqref{equ:rpmless}, we obtain
	\begin{align*}
	d(p',\partial J_q)\leq Cd(g_ip',g_i\partial J_{q})\leq C\eta^2 \h(g_i q) d(g_i\gamma p',g_i\gamma \partial J_q)\leq C\eta^2\tilde{r}_{p,n}\h(g_i q)/C_{\six} \leq C\eta\h(g_i p)/C_{\six}.
	\end{align*}
	By taking $C_{\six}>1$ large enough, we have $d(p', \partial J_q)\leq \h_n$.
\end{proof}

\bigskip{}
\paragraph{Definition of equivalence classes.}

Now we define equivalence classes in $Q_n$. We define them by induction. For $Q_1$,
\begin{itemize}
\item for $p\in Q_1-Q_1'$, set the equivalence class $C(p)$ of $p$ to be $\{p\}$.

\item for $p\in Q_1'$, set
\begin{equation*}
C(p)=\{\gamma_1 p:\,\gamma_1 B(p,\sqrt{\eta h_1 h(p)})\cap \partial \Omega_0\neq \emptyset,\,\,\, \gamma_1\in \Gamma_{\infty}\}.
\end{equation*}

\end{itemize}
Set 
\begin{equation*}
Q_1'':=\cup_{p\in Q_1} C(p),
\end{equation*}
and for any $p'\in Q_1''$ and $m\geq 1$, define
\begin{align*}
r_{p',m}=\sqrt{\eta h_m h(p')},\ \ B_{p',m}=B(p,r_{p',m}).
\end{align*}

Suppose we have defined $Q_n''$. We define the equivalence classes in $Q_{n+1}$ and the set $Q_{n+1}''$ as follows:
\begin{itemize}
\item[Case 1:] for $p\in Q_{n+1}'-\cup_{l\leq n}Q_l''$ such that the component in $\partial \Omega_n$ closest to $p$ is $\partial \Omega_0$, set
\begin{equation*}
C(p)=\{\gamma_1p:\, \gamma_1B(p,\sqrt{\eta h_{n+1}\h(p)})\cap \partial \Omega_0\neq\emptyset,\,\, \gamma_1\in \Gamma_{\infty}\}.
\end{equation*}
For any $p'\in C(p)$ and $m\geq n+1$, define 
\begin{align*}
r_{p',m}=\sqrt{\eta h_m \h(p')},\ \ B_{p',m}=B(p',r_{p',m}).
\end{align*}
\item[Case 2:] for $p\in Q_{n+1}'-\cup_{l\leq n}Q_l''$ such that the component in $\partial \Omega_n$ closest to $p$ is some $J_{q}$, write $q=\gamma^{-1}p_i$ with $\gamma^{-1}$ the representation of $q$. Let $\tilde{r}_{p,n}$ and $C_{\six}$ be as given in (\ref{equ:rpm}) and Lemma~\ref{lem:pjq} respectively. If $B(g_i\gamma p,\tilde{r}_{p,n}/C_{\six})\cap \,g_i\gamma\partial J_q\neq\emptyset$, set
\begin{equation}
\label{equ:cp}
C(p)=\{(g_i\gamma)^{-1} \gamma_1g_i\gamma p:\, \gamma_1B(g_i\gamma p,\tilde{r}_{p,n}/C_{\six})\cap \,g_i\gamma\partial J_q\neq\emptyset,\,\,\gamma_1\in (g_i\Gamma g_i^{-1})_{\infty}\}.
\end{equation}
Otherwise, set $C(p)=\{p\}$.


For any $p'\in C(p)$ and $m\geq n+1$, define
\begin{align}
\label{correct radius}
 &r_{p',m}=\frac{1}{C_{\six}}\sqrt{\frac{h_m h(g_i \gamma p')}{\eta h(g_i q)}}\,\,\,(\text{which equals}\,\,\,r_{p,m}),\\
 \label{correct ball}
 &B_{p',m}=(g_i\gamma)^{-1}B(g_i\gamma p', r_{p',m}).
 \end{align} 
\item[Case 3:] for $p\in Q_{n+1}- \cup_{l\leq n}Q_l''$ such that $p$ does not belong to the union of equivalence classes defined in the previous two cases, set $C(p)=\{p\}$ and for any $m\geq n+1$, define 
\begin{align*}
r_{p,m}=\sqrt{\eta h_m h(p)},\ \ B_{p,m}=B(p,r_{p,m}).
\end{align*}
\end{itemize}
Set
\begin{equation*}
Q_{n+1}''=\bigcup_{p\in Q_{n+1}-\cup_{l\leq n}Q_l''}C(p).
\end{equation*}
Then $\cup_{1\leq l\leq (n+1)}Q_l''\supset \cup_{1\leq l\leq (n+1)}Q_l$.

It is worthwhile to point out that under our definition of equivalence classes, it may happen that for $p\in Q_n-\cup_{l<n}Q_l''$, its equivalence class $C(p)$ may contain points $p'$ whose associated horospheres don't appear in the time interval $[n-1,n)$. This is our motivation to establish results like Lemmas \ref{lem:puniform} and \ref{lem:Pn}.

In the following discussion of the points $p$ in $Q_{n}''$, if the definition of $p$ involves a boundary component of $\partial \Omega_{n-1}$, we will need to consider this boundary component a lot of the times. For simplicity, we call the boundary component used to define $p\in Q_{n}''$ \textbf{the associated boundary component of $p$}. 



\bigskip{}
\paragraph{Uniformity among equivalence classes.}

For $p\in Q_{n}-\cup_{l<n} Q_l''$,
we show that, up to some constant, the points in the equivalence class $C(p)$ are ``uniform". 

\begin{lem}
\label{lem:puniform}
There exists $C_{\seven}>1$ such that for any $p\in Q_{n}-\cup_{l<n} Q_l''$ and any $p'\in C(p)$ we have
\begin{equation*}
1/C_{\seven}\leq \h(p)/\h(p')\leq C_{\seven}.
\end{equation*}
\end{lem}
It suffices to prove Lemma~\ref{lem:puniform} for the case when $\# C(p)\geq 2$ and the associated component of $p$ in $\partial \Omega_{n-1}$ is some $\partial J_q$. Write $q=\gamma^{-1}p_i$ with $\gamma^{-1}$ the representation of $q$. Let $r_{p,m}$ and $B_{p,m}$ be defined as in \eqref{correct radius} and \eqref{correct ball} respectively. We first show the following estimate.

\begin{lem}[Location of balls]
\label{lem:eclocation}

There exists a constant $C>1$ such that for $C(p)$ with the associated boundary component $\partial J_q$ and for $p'\in C(p)$
\begin{align}
&B(g_i\gamma p', r_{p',m})\subset B \left(g_ix_q, C/\eta\right)-B(g_ix_q,1/(C\eta)),\label{inverseinclusion} \\
&g_iB_{p',m}\subset B(g_iq, C\eta \h(g_i q))-B(g_iq,\eta \h(g_i q)/C). \label{inclusion}
\end{align}
\end{lem}

\begin{proof}
By Lemma~\ref{equ:jp}, we have
\begin{equation*}
g_i\gamma \partial J_q\subset B(g_i x_q, 1/(c_4\eta))-B(g_i x_q, 1/\eta).
\end{equation*}
We also have $r_{p,n}\leq C\eta$ by \eqref{equ:rpmless}. Meanwhile, the construction of the equivalence class $C(p)$ implies that $r_{p',m}=r_{p,m}$. Hence we obtain (\ref{inverseinclusion}). We use Lemma~\ref{lem:explicit} to obtain (\ref{inclusion}) from \eqref{inverseinclusion}.
\end{proof}

\begin{proof}[Proof of Lemma~\ref{lem:puniform}]
We prove the following explicit estimate:
\begin{equation}
 \label{bounded interval}
 \h(g_i p')\approx \eta^2 \h(g_i q) \h(g_i \gamma p').
 \end{equation}
 This together with Lemma~\ref{lem:height} and $h(g_i \gamma p')=h(g_i \gamma p)$ will lead to Lemma~\ref{lem:puniform}. Note that $\h(g_i \gamma p')\leq C$, with $C$ a constant depending on $\Gamma$. We apply Lemma~\ref{lem:rg1} to the horoball $H=H_{g_i\gamma p'}$ based at $g_i\gamma p'$ and the element $g=g_i\gamma^{-1}g_i^{-1}$.
 Notice that $g_ix_q =(g_i\gamma^{-1}g_i^{-1})^{-1}\infty =g^{-1}\infty$ and $h(g_i x_q)=h(g^{-1} H_\infty)$. We obtain
 \begin{equation*}
 \frac{\h(g_i x_q) \h(g_i \gamma p')}{d(g_ix_q,g_i\gamma p')^2+\h(g_i \gamma p')^2} \leq \h(g_i p')=h(gH) \leq \frac{\h(g_i x_q) \h(g_i \gamma p')}{(d(g_ix_q, g_i\gamma p')-\h(g_i \gamma p')/2)^2}.
\end{equation*}
Due to \eqref{equ:xpp}, we have $h(g_ix_q)=h(g^{-1}H_\infty)=h(gH_\infty)=h(g_iq)$.
By \eqref{inverseinclusion}, we have $d(g_i\gamma p',g_ix_q)\approx 1/\eta$. Therefore, we obtain
\begin{equation*}
 \frac{\eta^2 \h(g_i q) \h(g_i \gamma p')}{C}\leq \h(g_i p')\leq C\eta^2 \h(g_i q) \h(g_i \gamma p').\qedhere
\end{equation*}
\end{proof}


\begin{lem}
\label{lem:universal}
There exists $C_{\eight}>1$ such that for any $p\in Q''_{n}$ and any $m\geq n$, the ball $B_{p,m}$ satisfies
\begin{equation*}
B(p, \sqrt{\eta \h(p) \h_m}/C_{\eight})\subset B_{p,m} \subset B(p, C_{\eight}\sqrt{\eta \h(p) \h_m}).
\end{equation*}
\end{lem}

\begin{proof}
It is enough to prove the case when the associated component of $p$ in $\partial \Omega_{n-1}$ is some $\partial J_q$. Write $q=\gamma^{-1}p_i$ with $\gamma^{-1}$ the representation of $q$. By definition, 
\begin{equation*}
B_{p,m}=(g_i\gamma)^{-1}B(g_i\gamma p, r_{p,m})
\end{equation*}
with $r_{p,m}=\frac{1}{C_{\six}}\sqrt{\frac{h_m h(g_i\gamma p)}{\eta h(g_i q)}}$. 

Consider the action of $g_i\gamma^{-1} g_i^{-1}$ on $B(g_i\gamma p, r_{p,m})$. By Lemma \ref{lem:explicit} and \eqref{inverseinclusion}, we have 
\begin{equation}
\label{deri}
|(g_i\gamma^{-1}g_i^{-1})'x|\approx \eta^2 h(g_iq)\,\,\,\text{for any}\,\,x\in B(g_i\gamma p, r_{p,m}).
\end{equation}
Meanwhile, there exists a point $p'\in C(p)$ such that $p'$ satisfies Lemma \ref{lem:rgammaprp}. We have $h(g_i\gamma p)=h(g_i\gamma p')$ and $h(p)\approx h(p')$ (Lemma \ref{lem:puniform}). As a result, we obtain
\begin{equation}
\label{hei}
h(g_i\gamma p)\approx \frac{h(g_i p)}{\eta^2 h(g_i q)}.
\end{equation}

\eqref{deri} and \eqref{hei} yield there exists $C>1$ such that
\begin{equation*}
B(g_ip, \sqrt{\eta \h(g_ip) \h_m}/C)\subset g_iB_{p,m} \subset B(g_ip, C\sqrt{\eta \h(g_ip) \h_m}).
\end{equation*}
We use Lemma~\ref{lem:height} and~\ref{lem:bilip} to finish the proof.
 \end{proof}

\bigskip{}
\paragraph{Well-definedness of equivalence classes }

\begin{lem}\label{lem:welldefine}
	For any two equivalence classes $C(p')$ and $C(p'')$, they are either the same or disjoint.
\end{lem}
\begin{proof}
	Assume that these two equivalence classes are not the same and the intersection is nonempty.

\medskip{}
\textbf{Case 1}: 
Suppose one of these two equivalence classes just consists of one point, 
	say $\#C(p')=1$ and $\#C(p'')\geq 2$. We may assume that $p''\in Q_n'$ for some $n$. If the associated component of $p''$ is $\partial\Omega_0$, then we are in Case 1 of the definition of equivalence classes. For $p'\in C(p'')$, we obtain that $h(p')=h(p'')\in (h_n,h_{n-1}]/\eta$. As $\#C(p')=1$, we know that $p'$ is contained in $Q_l-\cup_{i<l}Q_i''$ for some $l<n$, which contradicts $h(p')\in (h_n,h_{n-1}]/\eta$.
	
	In the sequel, we assume further that the associated component of $p''$ is some $\partial J_q$.
	Write $q=\gamma^{-1}p_i$.
	 As $p'$ belongs to the equivalence class $C(p'')$, it follows from the definition that 
	\begin{equation}
	\label{eqn:welldef}
	B(g_i\gamma p', r_{p',n})\cap g_i \gamma \partial J_q\neq \emptyset,\,\,\, B(g_i\gamma p'', r_{p'',n})\cap g_i \gamma \partial J_q\neq \emptyset,
	\end{equation}
	where $r_{p',n}$ is defined as in \eqref{correct radius} and equals $r_{p'',n}$.
	
	The fact that $C(p')$ just consists of $p'$ implies $p'\in Q_l-\cup_{i<l}Q_i''$ for some $l<n$. Meanwhile, as $\partial J_q$ is the associated component of $p''$, by Lemma~\ref{lem:rprq} and~\ref{lem:puniform}, we have
	\begin{equation*}
	\h(q)\geq \h(p'')/(C\eta^2)\geq \h(p')/(C\eta^2)\geq h_l/(C\eta^3).
	\end{equation*}
Hence $\partial J_q\subset \partial\Omega_l$.

 (\ref{eqn:welldef}) allows us to apply Lemma~\ref{lem:pjq} to $p'$, and we obtain
$$d(p',\partial J_q)\leq \h_n< \sqrt{\eta\h_l h(p')}.$$
So $p'\in Q_l'-\cup_{i<l}Q_i''$. \eqref{eqn:welldef} yields
\begin{equation*}
B(g_i\gamma p',r_{p',l})\cap g_i\gamma \partial J_q\neq \emptyset,\,\,\, B(g_i\gamma p'',r_{p'',l})\cap g_i\gamma \partial J_q\neq \emptyset.
\end{equation*}
As $l<n$, $C(p')$ contains $p''$, which is a contradiction.

\textbf{Case 2}:	Suppose that $\# C(p'),\,\#C(p'')\geq 2$. Without loss of generality, we may assume that $p'\in Q_m'-\cup_{l<m}Q_l''$ and $p''\in Q_n'-\cup_{l<n}Q_l''$ and $m\leq n$. 

Let $p\in C(p')\cap C(p'')$. Then it follows from the construction of equivalence classes and Lemma~\ref{lem:pjq} that there are boundary components $\partial_1$ and $\partial_2$ in $\partial \Omega_{n-1}$ such that
	\[d(p,\partial_1)\leq \h_{m-1},\,\,\,\ d(p,\partial_2)\leq \h_{n-1}. \]
	On the one hand, as $\partial_1$ and $\partial_2$ are in $\partial \Omega_{n-1}$, if they are distinct, Lemma~\ref{lem:separation} states that their distance is greater than $\h_{n-1}/(2\eta)$. On the other hand, using Lemma~\ref{lem:puniform}, we obtain
	\begin{equation*}
	\h_n/\h_m\geq \h(p'')/(e\h(p'))=(\h(p'')/\h(p))(\h(p)/(e\h(p')))\geq 1/(eC_{\seven}^2).
	\end{equation*}
	Then
	\begin{equation*}
	d(\partial_1,\partial_2)\leq \h_{m-1}+\h_{n-1}\leq (1+eC_{\seven}^2)\h_{n-1}<\h_{n-1}/(2\eta).
	\end{equation*}
	We conclude that $\partial_1=\partial_2$. 
	
	There are two possibilities for $\partial_1$. One possibility is that $\partial_1$ is some $\partial J_q$. Write $q=\gamma^{-1}p_i$ with $\gamma^{-1}$ the representation of $q$. As $g_i\gamma p$ is related with $g_i\gamma p'$ and $g_i\gamma p''$ by elements in $(g_i\Gamma g_i^{-1})_{\infty}$, we have $\gamma_1 g_i\gamma p''=g_i\gamma p'$ for some $\gamma_1\in (g_i\Gamma g_i^{-1})_{\infty}$. As $m\leq n$, we have
	\begin{equation*}
	\emptyset\neq B(g_ip'', r_{p'',n})\cap g_i\gamma \partial J_{q}\subset B(g_ip'', r_{p'',m})\cap g_i \gamma \partial J_{q}.
	\end{equation*}
	As a result, we have $C(p')= C(p'')$.
	The other possibility is that $\partial_1=\partial \Omega_0$. It follows directly from the construction of equivalence classes that
	\begin{equation*}
	\emptyset\neq B(p'', r_{p'',n})\cap \partial \Omega_0\subset B(p'', r_{p'',m})\cap \partial \Omega_0.
	\end{equation*}
	Hence $C(p')=C(p'')$.
\end{proof}

\subsection{Auxiliary sets $A_n$ and $B_n$ in $\Omega_n$}
\label{sec:auxiliary sets}

We introduce auxiliary sets $A_n$ and $B_n$ in $\Omega_n$. By Lemma~\ref{lem:welldefine}, the set $Q_n''$ is disjoint with $\cup_{1\leq l\leq (n-1)}Q_n''$. For any $p\in Q_n''$ and any $m\geq n$, we have defined the ball $B_{p,m}$.
Note that it follows from the construction of $Q_{n}''$ that if $\#C(p)=1$, then the full ball $B_{p,n}$ is contained in $\Omega_n$.
For each $n$, we define
\begin{align*}
B_{n}=\Omega_n\cap \bigcup_{p\in \cup_{1\leq l\leq n}Q_l''} B_{p,n}\,\,\,\text{and}\,\,\,\ \ A_{n}=\Omega_{n}-B_{n}.
\end{align*}

In the followings, we will show how to use the set $B_n$ to detect whether a point is in the cusps of the manifold at time $t=n$ or not.

\subsubsection*{$B_n$ and cusps}


\begin{lem}\label{lem:np}
	For $x\in\widetilde{\Omega_0}$, if $\mathcal{G}_n x\in \mathcal C_\eta$, then there exists a parabolic fixed point $p$ with $\eta h(p)>h_n$ such that 
	\[d(\pi(x),p)<\sqrt{\eta h(p)h_n}. \]

\end{lem}
\begin{proof}
By assumption, in the universal cover $\operatorname{T}^1(\mathbb{H}^{d+1})$, the point $\calG_n x$ is contained in a horoball $H_p(\eta)$. Hence $h_n<\eta h(p)$. If $h_n\leq \eta h(p)/2$, then we can use Pythagorean's theorem to conclude that $d(\pi(x),p)\leq \sqrt{\eta h(p)h_n}$ (see Figure~\ref{fig:radius}). If $h_n\geq \eta h(p)/2$, then $d(\pi(x),p)\leq \eta h(p)/2\leq \sqrt{\eta h(p)h_n}$.
\end{proof}

\begin{figure}
	\begin{center}
		\begin{tikzpicture}[scale=2]
		\draw [domain=0:2*pi, samples=200] plot ({cos(\x r)},{sin(\x r)});
		\draw (-2,-1) -- (2,-1);
		\draw (0,0) node[above]{$O$};
		\draw [fill] (0,0) circle [radius=.02];
		\draw (0,-1/2) node[above left]{$D$};
		\draw [fill] (0,-1/2) circle [radius=.02];
		\draw (-{sqrt(3)/2},-1/2) -- ({sqrt(3)/2},-1/2);
		\draw (0,-1) node[below]{$p$};
		\draw [fill] (0,-1) circle [radius=.02];
		\draw (0,0) -- (0,-1);
		\draw (0,0) -- ({sqrt(3)/2},-1/2);
		\draw ({sqrt(3)/2},-1/2) node[right]{$C$};
		\draw (-2/3,-1/2) node[above left]{$\calG_nx$};
		\draw [fill] (-2/3,-1/2) circle [radius=.02];
		\draw (-2/3,-1) node[below left]{$\pi(x)$};
		\draw [fill] (-2/3,-1) circle [radius=.02];
		\draw [->, dashed] (-2/3,1.5) -- (-2/3,-1);
		\draw (-2/3,1.5) node[left]{$x$};
		\draw [fill] (-2/3,1.5) circle [radius=.02];
		\end{tikzpicture}
	\end{center}
	\caption{Radius}\label{fig:radius}
\end{figure}

\begin{lem}\label{lem:cuspBn}
Fix $c_{\nine}<\min\{1/C_{\five},1/C_{\eight}^2\}$, where $C_{\five}$ and $C_{\eight}$ are constants given in Lemma \ref{equ:jp} and Lemma \ref{lem:universal} respectively. For any $x\in\widetilde{\Omega_n}$, if $\mathcal{G}_n x\in \mathcal C_{c_{\nine}\eta}$, then $\pi(x)\in B_n$.

\end{lem}
\begin{proof}
	For $x\in \widetilde{\Omega_n}$, if $\mathcal{G}_n x\in\mathcal C_{c_{\nine}\eta}$, then it follows from Lemma~\ref{lem:np} that there exists a parabolic fixed point $p$ with $c_{\nine}\eta h(p)>h_n$ such that 
	\begin{equation}\label{equ:pix}
		d(\pi(x),p)<\sqrt{c_{\nine}\eta h(p) h_n}\leq c_9\eta h(p).
	\end{equation}
	By the definition of $P_n$ and $Q_n$, this $p$ must belong to $\bigcup_{j< n}(P_j\cup Q_j)$.
	If $p$ is in some $P_j$, then by Lemma~\ref{equ:jp} we have
	\[ \eta h(p)/C_{\five}<d(\pi(x),p), \]
	contradicting the assumption that $c_{\nine}<1/C_{\five}$. So $p$ must be in some $Q_j$. We use the construction of $B_n$, Lemma~\ref{lem:universal}, that is $B_n\supset B(p,\sqrt{\eta h(p)h_n}/C_{\eight})$, and \eqref{equ:pix} to conclude that $\pi(x)\in B_n$.
	\end{proof}

\begin{rem*}
	If $\pi(x)\in B_n$, then the point $\calG_nx$ is contained in $\calC_{C\eta}$. So the set $\calG_n \widetilde{B_n}$ is almost the same as the set of points in the cusps at time $t=n$, i.e. $\calC_\eta\cap\calG_n\widetilde{\Omega_{n}}$.
\end{rem*}

\subsubsection*{Parabolic fixed points, $B_n$ and different generations}

\begin{lem}\label{lem:Pn}
	We have $P_{n}\cap(\cup_{l\leq n} Q_l'')=\emptyset$.
\end{lem}
\begin{proof}
	If not, suppose $p\in P_{n}$ is also contained in an equivalence class $C(p')$ with $p'\in Q_m-(\cup_{1\leq l\leq m-1}Q_l'')$ and $m\leq n$. Due to $\# C(p)\geq 2$, by the construction of equivalence classes, we must have $p'\in Q_m'$. Let $\partial$ be the associated boundary component of $p'$. Recall the construction of the equivalence classes. If $\partial$ is $\partial \Omega_0$, it is easy to obtain $d(p, \partial \Omega_0)<h_m$. If $\partial$ is some $\partial J_q$, due to $\# C(p')\geq 2$, we use Lemma~\ref{lem:pjq} to deduce that $d(p,J_q)\leq h_m$. By Lemma~\ref{lem:puniform}, we have $h_m/h_n\leq h(p')/h(p)\leq C_{\seven}$. Hence by the definition of $P_{n}$
	\begin{equation*}
	 d(p,\partial)\geq h_{n-1}/(4\eta)\geq h_m/(4eC_{\seven}\eta)> h_m, 
	\end{equation*}
	which is a contradiction.
\end{proof}

\begin{lem}\label{lem:pbn}
There exists a constant $0<c_{\ten}<1$ such that for any $p\in P_{n+1}\cup Q_{n+1}''$, we have
	\begin{equation}\label{equ:pbn}
	d(p,B_n)\geq c_{\ten}h_n/\eta.
	\end{equation}
\end{lem}

\begin{proof}	
Let $p\in P_{n+1}\cup Q_{n+1}''$ and $B_{q,n}$ be a ball in $B_n$. By Lemma~\ref{lem:Pn}, $p$ and $q$ are two different parabolic fixed points. We have
\begin{align*}
d(p, B_{q,n})&\geq d(p,q)-C\sqrt{\eta h(q) h_n}\,\,\,\,\,\,\,\,\,\,\,\, \text{(by Lemma~\ref{lem:universal})}\\
&\geq \sqrt{h(p)h(q)}-C\sqrt{\eta h(q) h_n}\,\,\,\,\,\,\,\,\,\,\,\, \text{(by Lemma~\ref{lem:pdistance})}\\
&=\sqrt{h(q)}(\sqrt{h(p)}-C\sqrt{\eta h_n})\\
&\geq \sqrt{h_n/C\eta}\left(\sqrt{h_n/(Ce\eta)}-C\sqrt{\eta h_n}\right)
\geq h_n/(C\eta).\qedhere
\end{align*}
\end{proof}

Recall that $D_{n+1}=\cup_{p\in P_{n+1}} J_p$ and $\Omega_{n+1}=\Omega_n-D_{n+1}$.
\begin{lem}\label{lem:fullball}
If $\eta<\frac{c_{\ten}}{C_{\five}C_{\eight}}$, then:

\noindent 1. We have the followings:
\begin{align}
 \label{equ:generation}
 D_{n+1}\cap B_n&=\emptyset\,\,\,\text{and}\,\,\, \left(\cup_{p\in Q''_{n+1}} B_{p,n+1}\right)\cap B_n=\emptyset,
\\\nonumber
A_{n+1}&=(A_n-D_{n+1}-A_n\cap B_{n+1})\cup(A_{n+1}\cap B_n),
\\
\label{equ:anbn}	
A_n\cap B_{n+1}&=\cup_{p\in Q_{n+1}''}(B_{p,n+1}\cap\Omega_{n+1}),\ A_{n+1}\cap B_n=B_n-B_{n+1}.
\end{align}
2. For $p\in Q_n''$ and $m\geq n$, we have $B_{p,m}\cap \Omega_m=B_{p,m}\cap\Omega_n$. 

\end{lem}
\begin{proof}
	For $p\in P_{n+1}$, by Lemma \ref{equ:jp}, we have $J_p\subset B(p,C_{\five}\eta h(p))$. Then $C_{\five}\eta h(p)\leq C_{\five}h_n\leq c_{\ten}h_n/\eta$. By Lemma \ref{lem:pbn}, we have $J_p\cap B_n=\emptyset$. For $p\in Q_{n+1}''$, by Lemma \ref{lem:universal}, we have $B_{p,n+1}\subset B(p, C_{\eight}\sqrt{\eta h(p)h_n})$. Then $C_{\eight}\sqrt{\eta h(p)h_n}\leq C_{\eight} h_n< c_{\ten}h_n/\eta$.
	By Lemma \ref{lem:pbn}, we have $B_{p,n+1}\cap B_n=\emptyset$.
	The rest of the first statement can be obtained easily.
	
	For $m>l\geq n$, by $D_{l+1}\cap B_l=\emptyset$ and $B_{p,m}\cap \Omega_l\subset B_{l}$ when $p\in Q_n''$ we know that
	\[ B_{p,m}\cap\Omega_{l+1}=B_{p,m}\cap (\Omega_{l}-D_{l+1})=B_{p,m}\cap\Omega_{l}, \]
	which implies the second part of the statement.
\end{proof}

\subsection{Energy exchange argument}\label{sec:energy}
We are ready to prove Proposition \ref{keylemma}. 

\begin{lem}
\label{lem:sub2}
	There exists $c_{\ele}>0$ such that
	\begin{equation}
	\label{equ:sub2}
	\mu(B_{n}\cap A_{n+1})>c_{\ele}\mu(B_{n}). 
	\end{equation}
\end{lem}

The definition of equivalence classes is mainly used in this lemma. The idea is that the left hand side of \eqref{equ:sub2} can be expressed as a sum over equivalence classes. Over one equivalence class, we obtain a full ball whose measure we are able to estimate.
\begin{proof}[Proof of Lemma \ref{lem:sub2}]
We claim that for any distinct $p,p'\in \cup_{1\leq l\leq n}Q''_l$, we have $B_{p,n}\cap B_{p',n}=\emptyset$. The first equation in Lemma~\ref{lem:fullball} verifies the case when $p\in Q''_l$ and $p'\in Q''_j$ with $l\neq j$.

When $p,p'\in Q''_l$, using Lemma~\ref{lem:pdistance}, and~\ref{lem:universal}, we have
\begin{align*}
&d(B_{p,l}, B_{p',l})\geq d(p,p')-C\sqrt{\eta h(p) h_l}-C\sqrt{\eta h(p')h_l}\\
\geq & \sqrt{h(p) h(p')}-C\sqrt{\eta h(p) h_l}-C\sqrt{\eta h(p')h_l}
\geq h_l/(C\eta) -2Ch_{l-1}>0,
\end{align*}
showing the claim.

So $\mu(B_n\cap A_{n+1})$ can be divided into the sum over $p\in \cup_{1\leq l\leq n}Q''_l$. Using Lemma~\ref{lem:welldefine}, we can further group the sum into equivalence classes. Due to \eqref{equ:anbn}, $\mu(B_n\cap A_{n+1})=\mu(B_n-B_{n+1})$. Then the proof of \eqref{equ:sub2} is reduced to proving that there exists $c_{\ele}>0$ such that for each equivalence class $C(p)$, we have
\[ \sum_{p'\in C(p)} \mu((B_{p',n}-B_{p',n+1})\cap\Omega_n)\geq c_{\ele} \sum_{p'\in C(p)}\mu(B_{p',n}\cap \Omega_n).\]

We first consider the equivalence classes defined in Case 1 and Case 3 in page \pageref{correct radius}.
Then by the definition of equivalence class and Lemma \ref{lem:quasi-gammainfinity}, we obtain
\begin{align*}
\frac{\sum_{p'\in C(p)} \mu((B_{p',n}-B_{p',n+1})\cap\Omega_n)}{\sum_{p'\in C(p)}\mu(B_{p',n}\cap \Omega_n)}\geq \frac{\mu(B_{p,n}-B_{p,n+1})}{C\mu(B_{p,n})}=\frac{\mu(B(p,r_{p,n})-B(p,r_{p,n+1}))}{C\mu(B(p,r_{p,n}))}\geq \frac{c}{C},
\end{align*}
where the last inequality follows from Lemma~\ref{lem:rre} and $r_{p,n}=\sqrt{\eta h_mh_p}\ll \eta h(p)$.

\begin{figure}
	\def\svgwidth{350bp}
	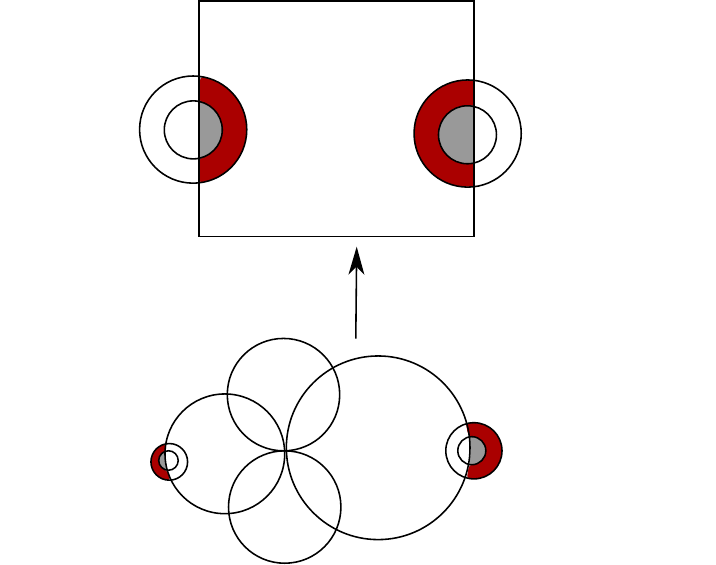
	\caption{Pairing partial balls: $q=\gamma^{-1}p_i$, $p'=g_i\gamma p$, $B_1=B(g_i\gamma p, \tilde r_{p,m}/C_{\six})$, $\gamma_1\in (g_i\Gamma g_i^{-1})_{\infty}$}
\end{figure}

Next we consider the equivalence classes defined in Case 2 in page \pageref{correct radius}. Suppose the associated boundary component of $p$ is $\partial J_{q}$ with $q=\gamma^{-1}p_i$ and $\gamma^{-1}$ is the representation of $q$. We first assume that $p_i=\infty$. By Lemma~\ref{lem:annulusquasi} and \eqref{inclusion} for any Borel subset $E\subset B_{p,n}$, we have
\begin{align*}
h^{\delta}_q\mu(\gamma E)/C\leq\mu(E)\leq Ch^{\delta}_q\mu(\gamma E).
\end{align*}
We have 
\begin{align}
\label{energyba}
&\frac{\sum_{p'\in C(p)} \mu((B_{p',n}-B_{p',n+1})\cap\Omega_n)}{\sum_{p'\in C(p)}\mu(B_{p',n}\cap \Omega_n)}
\geq \frac{\sum_{p'\in C(p)} \mu((B (\gamma p',r_{p',n})-B(\gamma p',r_{p',n+1}))\cap\gamma J_q^c)}{C\sum_{p'\in C(p)}\mu(B (\gamma p',r_{p',n})\cap \gamma J_q^c)}.
\end{align}
For each $p'\in C(p)$, we can write $p'=\gamma^{-1} \gamma_1 \gamma p$ with $\gamma_1\in \Gamma_{\infty}$. We have
\begin{align*}
&\mu (B (\gamma p',r_{p',n})\cap \gamma J_q^c)=\mu (\gamma_1 B (\gamma p, r_{p,n})\cap \gamma J_q^c)
\approx \mu (B (\gamma p, r_{p,n})\cap \gamma^{-1}_1\gamma J_q^c),
\end{align*}
where we use Lemma \ref{lem:quasi-gammainfinity} and (\ref{inverseinclusion}) to compute $|\gamma_1(x)|$ and $|x|$ for $x$ in $B (\gamma p, r_{p,n})$. Summing over $p'\in C(p)$, we can get a full ball. Similarly, we have
\begin{align*}
\mu((B (\gamma p', r_{p',n})-B (\gamma p',r_{p',n+1}))\cap \gamma J_q^c)
\approx \mu((B(\gamma p, r_{p,n})-B(\gamma p,r_{p,n+1}))\cap \gamma_1^{-1}\gamma J_q^c).
\end{align*}
We use these two observations and Lemma~\ref{lem:rre}, $r_{p,n}\ll\eta h(\gamma p)$ \eqref{equ:rpmless} to conclude
\begin{align*}
(\ref{energyba})\geq &\frac{\mu(B(\gamma p, r_{p,n})-B(\gamma p, r_{p,n+1}))}{C\mu(B (\gamma p, r_{p,n}))}\geq \frac{c}{C}.
\end{align*}

For general $p_i$, let $g_ip_i=\infty$ and $\Gamma_i=g_i\Gamma g_i^{-1}$. Using \eqref{conjugation}, we obtain
\begin{align}
\label{energyba2}
\frac{\sum_{p'\in C(p)} \mu((B_{p',n}-B_{p',n+1})\cap \Omega_n)}{\sum_{p'\in C(p)} \mu (B_{p',n}\cap \Omega_n)}
\approx\frac{\sum_{p'\in C(p)} \mu_{\Gamma_i}(g_i(B_{p',n}-B_{p',n+1})\cap g_i\Omega_n)}{\sum_{p'\in C(p)} \mu_{\Gamma_i} (g_iB_{p',n}\cap g_i\Omega_n)}.
\end{align}
This fraction can be estimated the same way as we estimate (\ref{energyba}). So
\begin{equation*}
(\ref{energyba2})\geq c/C.\qedhere
\end{equation*}
\end{proof}

Let $C_{\twl}=2C_{\five} C_{\three}+4C_{\eight}$, where $C_{\three}$, $C_{\five}$ and $C_{\eight}$ are constants given by Proposition \ref{double}, Lemma \ref{equ:jp} and Lemma \ref{lem:universal} respectively. 
Let 
\begin{equation*}
\Omega_n'=\{x\in\Omega_n:\ d(x,\partial\Omega_n)\leq C_{\twl}h_n \}. 
\end{equation*}
This is the set of points with distance less than $C_{\twl} h_n$ to the boundary of $\Omega_{n}$.
\begin{lem}[Boundary estimate]\label{lem:bou}
	There exists $c_{\thi}>0$ depending on $C_{\twl}\eta$ such that
	\begin{equation*}
	\mu(\Omega_n')\leq c_{\thi}\mu(\Omega_n) 
	\end{equation*}
and $c_{\thi}$ tends to 0 as $C_{\twl}\eta$ tends to 0.	
\end{lem}
\begin{proof}
	
	The boundary $\partial\Omega_n$ consists of 
	$\partial\Omega_0$ and $\partial J_p$ with $p\in\cup_{1\leq l\leq n}P_l$. For any $p\in \cup_{1\leq l\leq n}P_l$, write $p=\gamma p_i$ with $\gamma\in \Gamma$ the representation of $p$ and $\Gamma_i=g_i\Gamma g_i^{-1}$. Recall the definitions \eqref{equ:def neighborhood1} and \eqref{equ:def neighborhood 2}. Note that $h_{n}/(4\eta)\leq h(p)$. It follows from Lemmas~\ref{lem:height} and~\ref{lem:bilip} that there exists $C>1$ such that $h_{n}/(C\eta)<h(g_i p)$ and
	\begin{equation*}
	g_iN_{C_{\twl}h_n}(\partial J_p)\subset N_{CC_{\twl} h_n}(\partial J_{i,p}),\,\,\, N_{h_n/(C\eta)}(\partial J_{i,p})\subset g_i N_{h_n/(4\eta)}(\partial J_p),
	\end{equation*}
	where $J_{i,p}$ is defined as in \eqref{def:jip}. It follows from \eqref{conjugation}, Lemma \ref{lem:boud} and 
	\ref{lem:jpr} that there exists $c>0$ such that 
	\begin{align*}
		&\mu(\Omega_n')=\mu(N_{C_{\twl} h_n}(\partial{\Omega_0}))+\sum_{p\in \cup_{1\leq l\leq n}P_l}\mu(N_{C_{\twl}h_n}(\partial J_p))\\
		\leq &c\mu(N_{h_n/(4\eta)}(\partial \Omega_0))+C'\sum_{p\in \cup_{1\leq l\leq n}P_l}\mu_{\Gamma_i}(N_{CC_{\twl}h_n}(\partial J_{i,p}))\\
		\leq &c\mu(N_{h_n/(4\eta)}(\partial \Omega_0))+cC'\sum_{p\in \cup_{1\leq l\leq n}P_l}\mu_{\Gamma_i}(N_{h_n/(C\eta)}(\partial J_{i,p}))\\
		\leq &c\mu(N_{h_n/(4\eta)}(\partial \Omega_0))+cC'^2\sum_{p\in \cup_{1\leq l\leq n}P_l}\mu(N_{h_n/(4\eta)}(\partial J_p))
		\leq cC'^2\mu(\Omega_n),
	\end{align*}
	where the last inequality is due to Lemma~\ref{lem:separation} and $C'=\max_{i}e^{\delta d(o,g_io)}$. 
\end{proof}

\begin{lem}
\label{lem:sub1i}
	There exists $0<c_{\fou}<1$ such that
	\[\mu(A_n\cap (D_{n+1}\cup B_{n+1} ))\leq c_{\fou}\mu(A_n)+\mu(\Omega_n'). \]
\end{lem}
\begin{proof}
By Lemma~\ref{lem:fullball}, we have 
\begin{equation*}
A_n\cap B_{n+1}\cap (\Omega_n-\Omega_n')\subset \bigcup_{p\in Q_{n+1}''}B_{p,n+1}.
\end{equation*}
We consider the points $p\in Q_{n+1}''$ such that $B_{p,n+1}$ intersects the set on the left. Denote the set of such points by $Q_{n+1}'''$. By Lemma~\ref{lem:universal}, we have
\begin{equation*}
B_{p,n+1}\subset B(p,C_{\eight}\sqrt{\eta h(p)h_{n+1}})\subset B(p,C_{\eight}h_n).
\end{equation*}
Then its distance to $\partial\Omega_n$ is greater than $(C_{\twl}-2C_{\eight})h_n\geq C_{\twl}h_n/2$. So $Q_{n+1}'''$ must be a subset of points in Case 3 in page \pageref{correct radius}, 
and $B_{p,n+1}=B(p,\sqrt{\eta h(p)h_{n+1}})\subset B(p,h_n)$. 

For $p\in P_{n+1}$, by Lemma \ref{equ:jp}, we have $J_p\subset B(p,C_{\five}\eta h(p))\subset B(p,C_{\five}h_n)$.


	By \eqref{equ:pbn}, for $p\in P_{n+1}\cup Q_{n+1}'''$
	\begin{equation}\label{equ:Hp}
	d(p,B_n)\geq c_{\ten}h_n/\eta\geq C_{\twl}h_n/2.
	\end{equation}
	Hence
	\[B(p, C_{\five}h_{n})\subset B(p,C_{\twl} h_n/2)\subset A_n. \]
	Then by doubling property in Proposition~\ref{double}
	\[\mu(B(p, C_{\five}h_{n}))\leq c_{\fou}\mu(B(p,C_{\twl} h_n/2)). \]
	By Lemma \ref{lem:pdistance}, the points in the set $P_{n+1}\cup Q_{n+1}'''$ are of distance $h_{n+1}/\eta$ apart from each other. Hence
	the balls $B(p,C_{\twl}h_n/2)$ are disjoint. 
	Adding them together, we obtain
	\begin{align*}
		&\mu(A_n\cap(D_{n+1}\cup B_{n+1})-\Omega_n')\leq \sum_{p\in P_{n+1}\cup Q_{n+1}'''}\mu(B(p,C_{\five}h_n))\\
		\leq& c_{\fou} \sum_{p\in P_{n+1}\cup Q_{n+1}'''}\mu(B(p,C_{\twl}h_n/2))\leq c_{\fou}\mu(A_n).\qedhere
	\end{align*}
\end{proof}

\bigskip{}
Set $A_n'=A_n-\Omega_n'$ which is the set of points in $A_n$ with distance at least $C_{\twl}h_n$ to the boundary $\partial\Omega_n$.
\begin{lem}
\label{lem:sub1ii}
	There exist $N\in\N$ and $c_{\fif}>0$ depending on $\eta$ such that
	\begin{equation*}
	\mu(\cup_{l=1}^N D_{n+l})\geq c_{\fif}\mu(A_n').
	\end{equation*}
\end{lem}

Let $\widetilde{A_n}$ be the subset of $\widetilde{\Omega_n}$ such that $\pi(\widetilde{A_n})=A_n$. The key point of the proof is that we can use the recurrence property of the geodesic flow on $\calG_n(\widetilde{A_n})$, since Lemma~\ref{lem:cuspBn} tells us that $\calG_n(\widetilde{A_n})$ stays in a compact subset. Recall that we introduce some notations. We assumed that there are $j$ cusps in $M$ and $\{p_i\}_{1\leq i\leq j}$ is a complete set of inequivalent parabolic fixed points. We used the notation $H_{p_i}$ to denote the horoball based at $p_i$. Now let $H^s_{p_i}\subset \operatorname{T}^1(\mathbb{H}^{d+1})$ be the strong stable horosphere, that is,
\begin{equation*}
H_{p_i}^s:=\{x\in \operatorname{T}^1(H_{p_i}):\,\text{the basepoint of}\,x\,\text{is at}\,\partial H_{p_i}\,\text{and}\,\pi(x)=p_i\}.
\end{equation*}
By abusing the notation, we also use $H_{p_i}^s$ to denote its image in the quotient space $\operatorname{T}^1(M)$.

For every $x\in \operatorname{T}^1(M)$ and $\epsilon>0$, set $W^u(x,\epsilon)$ to be the local strong unstable manifold at $x$, that is,
\begin{equation*}
W^u(x,\epsilon):=\{y\in\operatorname{T}^1(M):\,\lim_{t\to -\infty}d(\mathcal{G}_{t}x,\mathcal{G}_ty)=0,\ d^u(y,x)\leq \epsilon\},
\end{equation*}
where $d(\cdot,\cdot)$ is the Riemannian metric on $\operatorname{T}^1(M)$ and $d^u(\cdot,\cdot)$ is the Riemannian metric restricted on the strong unstable manifold.

Denote by $W$ in $\operatorname{T}^1(M)$ the non-wandering set of the geodesic flow. 
\begin{lem}\label{lem:recurrence}
Let $K$ be any compact subset in $W$. Then there exists $U_0>0$ such that for every $x$
	 in $K$ and every $H^s_{p_i}$ in $\operatorname{T}^1(M)$, there exists a time $t\in[0,U_0]$ such that $\calG_t(W^u(x,1))$ meets $H^s_{p_i}$. 
\end{lem}
\begin{proof}
	
	Let $\epsilon<1/10$ and consider $Z_1=\cup_{x\in \partial H^s_{p_i}}W^u(x,\epsilon)$ and $Z_2=\cup_{x\in \partial H^s_{p_i}}W^u(x,5\epsilon)$ in $\operatorname{T}^1(M)$. Then $Z_1$ is a transversal section to the geodesic flow. By ergodicity of geodesic flow on non-wandering set $W$, there exists a point $y$ such that its negative time orbit is dense and there exists $t_0\geq 0$ such that 
	$\calG_{t_0} y\in Z_1 $.
	We can cover the compact set $K$ with a finite number of balls of radius $\epsilon$. There exists $t_1>0$ such that $\calG_{[-t_1,0]}y$ intersects every ball.
	
	Fix any $x$ in $K$. There exists $x'\in W^u(x,\epsilon)$ and $-s\in [-t_1,0]$ such that $d(x',\mathcal{G}_{-s}y)\leq 2\epsilon$ and $\calG_{-s}y$ are in the same strong stable manifold (that is to say, $\lim_{t\to \infty}d(\mathcal{G}_t x', \mathcal{G}_t (\mathcal{G}_{-s}y))=0$). Therefore
	\[d(\calG_{s+t_0}x',\calG_{t_0} y)\leq 2\epsilon. \]
	Using $\calG_{t_0} y\in Z_1$ and local product structure, we have $\calG_{s+t_0}x'\in \calG_{s_1}Z_2$ for some $s_1\in[-\epsilon,\epsilon]$. Due to the definition of $Z_2$, we can find $x''\in W^u(x,6\epsilon)$ such that $\calG_{s+t-s_1}x''\in H^s_{p_i}$.
\end{proof}

The following lemma is a straightforward corollary of Lemma~\ref{lem:recurrence}. Recall that $c_{\nine}>0$ is the constant given in Lemma \ref{lem:cuspBn}. Let $K_{c_{\nine}\eta}=W-\Gamma\backslash\calC_{c_{\nine}\eta}$. The base of non-wandering set in $M$ is the convex core $C(M)$ and the base of $\Gamma\backslash\calC_{c_{\nine}\eta}$ is a union of proper horocusps. By Definition~\ref{def:geofinite}, we know that $K_{c_{\nine}\eta}$ is compact.
\begin{lem}\label{lem:parabolic}
	Let $U_0$ be the constant in Lemma~\ref{lem:recurrence} with $K=K_{c_{\nine}\eta}$. For every $x$ in $\Delta_0\cap\Lambda_\Gamma$ and $n\in\N$, if $\calG_n \tilde{x}$ is in $K_{c_{\nine}\eta}$, where $\tilde{x}$ is the point in $\widetilde{\Omega_0}$ such that $\pi(\tilde{x})=x$, then the ball $B(x,h_n)$ contains a parabolic fixed point with height in $h_n[e^{-U_0},1]$.
\end{lem}
\begin{proof}
	Let $\tilde{B}$ be the set in $\widetilde{\Omega_0}$ such that $\pi(\tilde{B})=B(x,h_n)$. We have $\calG_n\tilde{B}=W^u(\calG_n\tilde{x},1)$. As $\calG_n\tilde{x}\in K_{c_{\nine}\eta}$, by Lemma~\ref{lem:recurrence}, there exists $t\in[0,U_0]$ such that $\calG_tW^u(\calG_n\tilde{x},1)$ intersects some $H^s_{p_i}$. Hence in the universal cover $\T^1(\H^{d+1})$, the unstable leaf $\calG_tW^u(\calG_n\tilde{x},1)$ is tangent to a horoball. Let $q$ be the basepoint of the horoball. Then $q$ is in $B(x,h_n)$.
\end{proof}

\begin{proof}[\textbf{Proof of Lemma~\ref{lem:sub1ii}}]
	Set $N=U_0+2\lfloor -\log \eta\rfloor+2$. We claim that: There exists $C'>1$ depending on $\eta$ such that $\mathop{\cup}_{1\leq l\leq N}P_{n+l}$ is a $C'h_n$ dense set in $A_n'\cap\Lambda_\Gamma$. That is to say, for every $x\in A_n'\cap\Lambda_\Gamma$, there exists some $p\in \bigcup_{1\leq l\leq N}P_{n+l}$ such that $d(x,p)\leq C'h_n$.
	
	\medskip{}
	Let $k=\lfloor -\log \eta\rfloor$. Fix any point $x\in A_n'\cap\Lambda_\Gamma$. We consider the position of $x$ in $\Omega_{n+k}$.
	
	\begin{itemize}
	
	 \item[Case 1] Suppose $x\notin \Omega_{n+k}$. Then $x\in J_{p}$ for some $p\in \cup_{1\leq l\leq k}P_{n+l}$. So $d(x,p)\leq C_{\five}\eta h(p)\leq C_{\five}h_n$.
	 
	 \item[Case 2] Suppose $x\in \Omega_{n+k}$ and $d(x,\partial \Omega_{n+k})<3h_{n+k}/\eta$. As $x\notin \Omega_n'$, we have $d(x,\partial \Omega_n)\geq C_{\twl}h_n$. Meanwhile, we have $3h_{n+k}/\eta< C_{\twl} h_n$. Consequently, the connected component in $\partial \Omega_{n+k}$ closest to $x$ is some $\partial J_p$ with $p\in \cup_{1\leq l\leq k} P_{n+l}$. Hence
	 \begin{equation*}
	 d(x,p)\leq d(x,\partial J_p)+d(\partial J_p,p)\leq 3h_{n+k}/\eta+C_{\five} \eta h(p)\leq Ch_n.
	 \end{equation*}
	
	\item[Case 3] 
		
		Suppose $x\in A_{n+k}\cap\Lambda_\Gamma$ and $d(x,\partial\Omega_{n+k})>2h_{n+k}/\eta$. By Lemma~\ref{lem:cuspBn}, $\mathcal{G}_{n+k}\tilde x\in K_{c_{\nine}\eta}$. It then follows from Lemma~\ref{lem:parabolic} that $B(x, h_{n+k})$ contains a parabolic fixed point $p$ with height in $h_{n+k}[e^{-U_0},1]$. Let $j=\lfloor-\log (\eta h(p))\rfloor$, then $j\in n+2k+[0,U_0+1]$. Let's consider the position of $p$. 
		\begin{itemize}
			\item Suppose $p\in P_{j+1}$. Then $d(x,p)\leq h_{n+k}$. 
			\item Suppose $p\notin P_{j+1}$ and $p\notin \Omega_j$. Note that the conditions on $x$ and $p\in B(x,h_{n+k})$ imply that $p\in \Omega_{n+k}$. So there exists some $q\in \cup_{l=1}^{j-n-k}P_{n+k+j}$ such that $p\in J_q$. We obtain
			\begin{equation*}
			d(x,q)\leq d(x,p)+d(p,q)\leq h_{n+k}+C_{\five}\eta h(q)\leq C h_{n+k}. 
			\end{equation*}
			\item Suppose $p\notin P_{j+1}$ and $p\in\Omega_j$. 
			Because $\eta h(p)\in (h_{j+1},h_j]$, we must have $p\in Q_{j+1}$.
			By the definition of $Q_{j+1}$, we have $d(p,\partial\Omega_j)\leq h_j/\eta$. Observe that
			\begin{equation*}
			d(p,\partial\Omega_{n+k})\geq d(x,\partial \Omega_{n+k})-d(x,p)>2h_{n+k}/\eta-h_{n+k}>h_j/\eta. 
			\end{equation*}
			So there exists $q\in \cup_{l=1}^{j-n-k}P_{n+k+l}$ such that $d(p,J_{q,\eta})\leq h_j/\eta$. This implies
			\begin{equation*}
			d(x,q)\leq d(x,p)+d(p,q)\leq h_{n+k}+h_j/\eta+C_{\five}\eta h(q)\leq Ch_{n+k}.
			\end{equation*}		
		\end{itemize}
			
		\item[Case 4]
		Suppose $x\in B_{n+k}$ and $d(x,\partial \Omega_{n+k})\geq 3h_{n+k}/\eta$. As $x\in A_n$, we have $x\in B_{n+k}-B_n$. So there exists $p\in \cup_{1\leq l \leq k}Q_{n+l}''$ such that $x\in B_{p, n+k}$. By Lemma \ref{lem:universal}, we have $$B_{p,n+k}\subset B(p,C_{\eight}\sqrt{\eta h(p) h_{n+k}})\subset B(p,C_{\eight}\sqrt{h_n h_{n+k}}).$$
			Since $h_{k}\geq \eta$, for any $y\in B_{p,n+k}$, using Lemma~\ref{lem:universal}, we have
		\begin{equation*}
		d(y,\partial \Omega_{n+k})\geq d(x,\partial \Omega_{n+k})-d(x,y)\geq 3h_{n+k}/\eta-2C_{\eight}\sqrt{h_n h_{n+k}}\geq 2h_{n+k}/\eta.
		\end{equation*}
		So the full ball $B_{p,n+k}$ is contained in $\Omega_{n+k}$. By a similar argument as in the proof of Lemma \ref{lem:sub2}, we have $\mu(B_{p,n+k}-B_{p,n+k+1})>0$. We can find a point $y\in \Lambda_{\Gamma}\cap (B_{p,n+k}-B_{p,n+k+1})$. By \eqref{equ:anbn}, we know that in fact $y$ is in $A_{n+k+1}\cap \Lambda_{\Gamma}$.
		\begin{itemize}
			\item If $d(y,\partial \Omega_{n+k+1})>2h_{n+k+1}/\eta$, the point $y$ belongs to Case 3.
			\item Otherwise, $d(y,\partial \Omega_{n+k+1})\leq 2h_{n+k+1}/\eta$. But $d(y,\partial\Omega_{n+k})>2h_{n+k}/\eta$, there exists $J_q$ with $q\in P_{n+k+1}$ such that $d(y,\partial J_q)\leq 2h_{n+k+1}/\eta$.
		\end{itemize}
	 It follows that there exists $q\in \cup_{1\leq l\leq N}P_{n+l}$ such that
		\begin{equation*}
		d(x,q)\leq d(x,y)+d(y,q)\leq 2C_{\eight}\sqrt{h_{n}h_{n+k}}+d(y,q)\leq Ch_n. 
		\end{equation*}	
	\end{itemize}

	Finally, by Lemma \ref{equ:jp} we know that for $p\in\cup_{1\leq l\leq N}P_{n+l} $, the balls $B(p,\eta h(p)/C_{\five})$ are disjoint. 
	Using the claim and doubling property (Proposition~\ref{double}), 
	\begin{align*}
		&\mu(\cup_{l=1}^N D_{n+l})\geq \sum_{p\in\cup_{1\leq l\leq N}P_{n+l}}\mu(B(p,\eta h(p)/C_{\five}))\\
		\geq& c_{\fif}\sum_{p\in\cup_{1\leq l\leq N}P_{n+l}}\mu(B(p,C' h_n ))\geq c_{\fif}\mu(A_n'\cap\Lambda_\Gamma)=c_{15}\mu(A_n'),
	\end{align*}
	finishing the proof.
\end{proof}

\begin{proof}[\textbf{Proof of Proposition~\ref{keylemma}}]
	We will prove the following statement and Proposition~\ref{keylemma} is a direct consequence of this:	for $\eta$ sufficiently small, there exist $N$ and $c_0>0$ depending on $\eta$ such that
	\[\mu(\cup_{l=1}^N D_{n+l})\geq c_0\mu(\Omega_n). \]
	
	Recall that $c_{\ele},c_{\thi}$ and $c_{\fou}$ are the constants given in Lemma \ref{lem:sub2}, \ref{lem:bou} and \ref{lem:sub1i} respectively. We can take $c_{\thi}$ small enough such that $c_{\thi}<c_{\ele}$ and $c_{\thi}+c_{\fou}<1$.
	Write $t_n=\frac{\mu(A_n)}{\mu(B_n)}$, which makes sense even if $\mu(B_n)=0$. Then by Lemma~\ref{lem:fullball},~\ref{lem:sub2},~\ref{lem:sub1i}, and~\ref{lem:bou}
	\begin{align*}
		t_{n+1}&=\frac{\mu(A_{n+1})}{\mu(B_{n+1})}=\frac{\mu(A_n)+\mu(B_n\cap A_{n+1})-\mu(A_n\cap(D_{n+1}\cup B_{n+1}))}{\mu(B_n)-\mu(B_n\cap A_{n+1})+\mu(A_n\cap B_{n+1})}\\
		&\geq \frac{\mu(A_n)+c_{\ele}\mu(B_n)-(c_{\fou}\mu(A_n)+c_{\thi}\mu(\Omega_n))}{\mu(B_n)-c_{\ele}\mu(B_n)+(c_{\fou}\mu(A_n)+c_{\thi}\mu(\Omega_n))}=\frac{t_n-(c_{\fou}+c_{\thi})t_n+(c_{\ele}-c_{\thi})}{1+(c_{\fou}+c_{\thi})t_n-(c_{\ele}-c_{\thi})}=f(t_n).
	\end{align*}
	Here $f$ is fractional function and of the form $f(t)=\frac{a_1t+a_2}{b_1t+b_2}$ with $a_i,b_i>0$, which is a monotone function. Hence 
	$$\inf_{t\in\R^+ }f(t)\geq\min\{\frac{a_1}{b_1},\ \frac{a_2}{b_2} \}=\min\{\frac{1-(c_{\fou}+c_{\thi})}{c_{\fou}+c_{\thi}},\frac{c_{\ele}-c_{\thi}}{1-(c_{\ele}-c_{\thi})} \}=q(c_{\thi}). $$
	By $t_n>0$, 
	 there is a uniform lower bound of $t_n$ for all $n\in\N$.
	
	Then use Lemma~\ref{lem:sub1ii} to obtain the desired statement:
	\begin{align*}
		&\mu(\cup_{l=1}^N D_{n+l})\geq c_{\fif}\mu(A_n')\geq c_{\fif}(\mu(A_n)-\mu(\Omega_n'))\\
		\geq &c_{\fif}(\mu(A_n)-c_{\thi}\mu(\Omega_n))=c_{\fif}(\frac{t_n}{1+t_n}-c_{\thi})\mu(\Omega_n).
	\end{align*}
	If $c_{\thi}$ is small enough, then $\frac{t_n}{1+t_n}\geq\frac{q(c_{\thi})}{1+q(c_{\thi})}>c_{\thi}$. Then we can fix a small $\eta$ in Lemma~\ref{lem:bou} such that $c_{\thi}$ satisfies these restriction. 
\end{proof}

\subsection{Exponential tail}\label{sec:exptail}


For one cusp case, we have described how to construct the countable collection of disjoint open subsets in $\Delta_0$ and the expanding map in Section \ref{coding procedure}. When there are multi-cusps, the coding is constructed in two steps and we describe the first step here and finish the rest in Section \ref{sec:codmulti}.

 Suppose that there are $j$ cusps. Recall the regions $\Delta_{p_i}$ introduced in Section \ref{sec:multi}. We claim: there is a countable collection of disjoint open subsets $\sqcup_{i,k}\Delta_{p_i,k}$ in $\Delta(=\sqcup_i \Delta_{p_i})$ and an expanding map $T_0:\sqcup_{i,k}\Delta_{p_i,k}\to \Delta$ such that
 \begin{itemize}
 	\item $\sum_{i,k}\mu(\Delta_{p_i,k})=\mu(\Delta)$.
 	\item For each $\Delta_{p_i,k}$, it is a subset in $\Delta_{p_i}$. And there exists an element $\gamma_0\in \Gamma$ such that $\Delta_{p_i,k}=\gamma_0 \Delta_{p_l}$ for some $1\leq l\leq j$ and $T|_{\Delta_{p_i,k}}=\gamma_0^{-1}$. 
 \end{itemize}
Denote by $\mathcal{H}_0$ the set of inverse branches of $T_0$.

 The construction is as follows: for each $\Delta_{p_i}$, we apply the construction in Section \ref{coding procedure} to the group $\Gamma_i=g_i\Gamma g_i^{-1}$ and the region $g_i\Delta_{p_i}=\Delta_{p_i}'$. In particular, we attain a countable collection of disjoint open sets $\Delta_{p_i,k}'$. Moreover, Proposition \ref{keylemma} holds for $\Delta_{p_i}'$. We set $\Delta_{p_i,k}=g_i^{-1}\Delta_{p_i,k}'$. 
 
 
 For an element $\gamma_0$ in $\calH_0$, if $\gamma_0$ maps $\Delta_{p_l}$ into some $\Delta_{p_i,k}$, then we define
\begin{equation}\label{infty}
|\gamma_0'|_\infty=\sup_{x\in\Delta_{p_l}}|\gamma_0'(x)|.
\end{equation}
The infinity norm of the derivative of a composition map is defined similarly.
We prove the following.
\begin{lem}
\label{lem:h0}
 There exists $\epsilon>0$ such that
\begin{equation}\label{h0}
\sum_{\gamma_0\in\calH_0} |\gamma_0'|_{\infty}^{\delta-\epsilon}<\infty.
\end{equation}
\end{lem}
For one cusp case, this gives the exponential tail \eqref{sum}. When there are multi-cusps, \eqref{h0} can be understood as that the map $T_0$ satisfies the exponential tail property. 

We start the proof of Lemma \ref{lem:h0} with the following result. Denote by $\cup_n P_n$ the set of ``good parabolic fixed points" which appear in the first step of the construction of the coding for multi-cusp case and are defined similarly as \eqref{good parabolic fixed points}. 

\begin{lem}\label{lem:ng}
There exists $C>0$ such that for any parabolic fixed point $p=\gamma p_i\in \Delta_0\cap \cup_n P_n$, we have for any $\epsilon\in (0,\delta-k/2)$,
\begin{equation*}
\sum_{\gamma_1\in N_p}|(\gamma \gamma_1)'|_{\infty}^{\delta-\epsilon}\leq C(2\delta-k-2\epsilon)^{-1}h(p)^{-\epsilon} \eta^{-2\epsilon}\mu(J_{p}),
\end{equation*} 
where $k$ is the rank of the parabolic fixed point $p$ and $N_p$ is defined in \eqref{flower group}.

\end{lem}
\begin{proof}
We first consider the case when $p=\gamma \infty$. By Lemma~\ref{lem:explicit}, we have for every $x\in \Delta_0$ and every $\gamma_1\in N_p$,
\begin{equation*}
|(\gamma \gamma_1)'(x)|=|\gamma'(\gamma_1x)|=\frac{h(p)}{d(\gamma_1x,x_{\gamma})^2}.
\end{equation*}
As $\cup_{\gamma_1\in N_p}\gamma_1\Delta_0\subset B(x_{\gamma}, 1/\eta)^c$ where $x_{\gamma}=\gamma^{-1}\infty$, we use general polar coordinates to obtain
\begin{equation}
\label{equ:tail1}
\sum_{\gamma_1\in N_p}|(\gamma \gamma_1)'|^{\delta-\epsilon}_{\infty}\ll h(p)^{\delta-\epsilon}\sum_{\gamma_1\in N_p}\frac{1}{d(\gamma_1\Delta_0,x_{\gamma})^{2\delta-2\epsilon}}\ll \frac{h(p)^{\delta-\epsilon} \eta^{2\delta-2\epsilon-k}}{2\delta-2\epsilon-k}.
\end{equation}
Meanwhile, by the quasi-invariance of PS measure and \eqref{equ:change}, we have for every $\gamma_1\in N_p$
\begin{equation*}
\mu(\gamma \gamma_1 \Delta_0)=\int_{x\in \Delta_0}|(\gamma \gamma_1)'(x)|^{\delta}_{\mathbb{S}^n}\dd\mu(x)\approx \int_{x\in \Delta_0}|(\gamma \gamma_1)'(x)|^{\delta}\dd\mu(x)\approx\frac{\mu(\Delta_0)h(p)^{\delta}}{d(\gamma_1\Delta_0,x_{\gamma})^{2\delta}}.
\end{equation*}
Therefore,
\begin{equation}
\label{equ:tail2}
\mu(\cup_{\gamma_1\in N_p} \gamma\gamma_1\Delta_0)\gg \mu(\Delta_0)h(p)^{\delta}\sum_{\gamma_1\in N_p}\frac{1}{d(\gamma_1\Delta_0,x_{\gamma})^{2\delta}}\gg \frac{h(p)^{\delta}\eta^{2\delta-k}}{2\delta-k}.
\end{equation}
Hence (\ref{equ:tail1}) and (\ref{equ:tail2}) together yield the statement for the case when $p=\gamma \infty$.

For the general case when $p=\gamma p_i$ with $g_i p_i=\infty$. Note that for every $\gamma_1\in N_p$, we have
$\gamma \gamma_1=g_i^{-1}(g_i\gamma \gamma_1 g_i^{-1}) g_i$. Hence by Lemma~\ref{lem:height} and~\ref{lem:bilip}
\begin{equation*}
h(g_i p)\approx h(p),\ \ |(\gamma \gamma_1)'|_{\infty}=\sup_{x\in\Delta_{p_i}}|(\gamma\gamma_1)'(x)|\approx \sup_{x\in g_i\Delta_{p_i}}|(g_i \gamma \gamma_1 g_i^{-1})'(x)|= |(g_i \gamma \gamma_1 g_i^{-1})'|_{\infty}.
\end{equation*}
Write $\Gamma_i=g_i\Gamma g_i^{-1}$. Using \eqref{conjugation}, we obtain
\begin{equation*}
\mu(\gamma \gamma_1 \Delta_{p_i})\approx \mu_{\Gamma_i}(g_i\gamma \gamma_1g_i^{-1}(g_i\Delta_{p_i})).
\end{equation*}
We have $g_ix_p=(g_i\gamma g_i^{-1})^{-1}\infty\in g_i\overline{\Delta}_{p_i}$.
Because $g_ip=g_i\gamma g_i^{-1}\infty$ and $g_i \gamma_1 g_i^{-1}(g_i\Delta_{p_i})\subset B(g_ix_p,1/\eta)^c$ for any $\gamma_1\in N_p$, we are able to compare $\sum_{\gamma_1} |(g_i \gamma \gamma_1 g_{i}^{-1})'|^{\delta-\eps}_{\infty}$ with $\mu_{\Gamma_i}(\cup_{\gamma_1} g_i \gamma \gamma_1 g_i^{-1}(g_i \Delta_{p_i}))$ as above and this will prove Lemma \ref{lem:ng} for the general case.
\end{proof}

\begin{proof}[\textbf{Proof of Lemma \ref{lem:h0}}]
We only need to sum the inverse branches in $\mathcal{H}_0$ whose images are in $\Delta_{0}$. For a general inverse branches whose image is in $\Delta_{p_j}$, we consider the group $g_j\Gamma g_j^{-1}$ and the inequality can be proved in the same fashion. By Lemma~\ref{lem:ng} and Proposition~\ref{keylemma}, for any sufficiently small $\epsilon\in (0,1)$,
\begin{align*}
&\sum_{n\in\N}\sum_{p=\gamma p_i\in P_n\cap \Delta_0}\sum_{\gamma_1\in N_p}|(\gamma \gamma_1)'|_{\infty}^{\delta-\epsilon}
\ll \sum_{n\in \N}\sum_{p\in P_n}\mu(J_{p})h(p)^{-\epsilon}\eta^{-2\epsilon}\\
\leq& \eta^{-2\epsilon}\sum_{n\in \N}\mu(\Omega_n)e^{\epsilon(n+1)}\leq \eta^{-2\epsilon}\sum_{n\geq N}(1-\epsilon_0)^ne^{\epsilon(n+1)}+\eta^{-2\epsilon}\sum_{n< N}\mu(\Omega_n)e^{\epsilon(n+1)}.
\end{align*}
By choosing an $\epsilon$ small enough such that $(1-\epsilon_0)e^{\epsilon}<1$, the above sum is finite.
\end{proof}

\subsection{Coding for multi-cusps}\label{sec:codmulti}
We caution the readers that the symbol $\gamma$ will be used to denote an inverse branch in this section.

In Section \ref{sec:exptail}, we have found a countable collection of disjoint open subsets $\sqcup_{i,k}\Delta_{p_i,k}$ in $\Delta(=\sqcup_i \Delta_{p_i})$ and an expanding map $T_0:\sqcup_{i,k}\Delta_{p_i,k}\to \Delta$.
Without loss of generality, we may suppose that $T_0$ is irreducible, which means there doesn't exist a nonempty subset of $I_1\subsetneq \{1,\cdots,j \}$ such that 
$$T_0(\cup_{i\in I_1}\Delta_{p_i})\subset \cup_{i\in I_1}\Delta_{p_i}.$$
Otherwise, we can consider the restriction of $T_0$ to the union $\cup_{i\in I_1}\Delta_{p_i}$.


For $x\in \Delta_0=\Delta_{p_1}$, define the first return time
\[n(x)=\inf\{n\in\N:\,T_0^n(x)\in\Delta_{p_1} \}.\]
Set $n(x)=\infty$ if $T_0^n(x)$ doesn't come back to $\Delta_{p_1}$ for all $n\in\N$ or $T_0^n(x)$ lies outside of the domain of definition of $T_0$ for some $n$.

The expanding map $T$ in Proposition \ref{prop:coding} is defined by 
\begin{equation*}
T(x)=T_0^{n(x)}(x)\,\,\,\text{for}\,\,\, x\,\,\,\text{such that}\,\,\, n(x)<\infty. 
\end{equation*}
By the definition of $T_0$, 
we have 
\begin{itemize}
\item either $T(x)=\gamma^{-1}x$ with $\gamma\in \mathcal{H}_0$ and $\gamma: \Delta_{p_1}\to \Delta_{p_1}$, 
\item or $T(x)=\gamma_{n(x)}^{-1}\cdots \gamma_1^{-1} x$ with $\gamma_{l}\in \mathcal{H}_0$ for $l=1,\ldots,n(x)$, 
where $\gamma_l$ maps $\Delta_{p_{k(l+1)}}$ to $\Delta_{p_{k(l)}}$ with $ p_{k(1)}=p_{k(n(x)+1)}=p_1$ and $p_{k(l)}\neq p_1$ for $1<l\leq n(x)$.

\end{itemize}
The string $\gamma_{n(x)}^{-1}\cdots \gamma_1^{-1}$ gives an open subset $\gamma_1\cdots \gamma_{n(x)}\Delta_{p_1}\subset \Delta_{p_1}$. They consist of open subsets described in Proposition \ref{prop:coding}. 

To prove (1), (3) and (4) in Proposition \ref{prop:coding}, we start with a preliminary version of Lemma~\ref{lem:l1}. Define
\begin{equation*}
 U_i=g_i^{-1}B(g_i\Delta_{p_i},1/(2\eta))^c\,\,\,\text{for}\,\,\,1\leq i\leq j.
 \end{equation*}
\begin{lem}\label{lem:l1'}
	 If $\gamma$ is an inverse branch in $\calH_0$ which maps $\Delta_{p_i}$ into $\Delta_{p_l}$, then $\gamma^{-1}U_l\subset U_i$.
\end{lem}

\begin{proof}
Due to the construction, we know that $\gamma p_i$ is a parabolic fixed point inside $\Delta_{p_l}$. The definition of $U_l$ implies
\begin{equation*}
g_lU_l\subset B(g_l\gamma p_i,1/(2\eta))^c.
\end{equation*}
Because the maps $g_i$'s are bi-Lipschitz (Lemma~\ref{lem:bilip}), we obtain
\begin{equation}
\label{invcon1}
g_i U_l\subset B(g_i \gamma p_i,1)^c=B(g_i\gamma g_i^{-1}\infty,1)^c.
\end{equation}
By Lemma~\ref{lem:explicit}, we have
\begin{align}
\label{invcon2}
(g_i\gamma g_i^{-1})^{-1} B(g_i\gamma g_{i}^{-1}\infty, 1)^c=B((g_i\gamma g_i^{-1})^{-1}\infty, h(g_i \gamma p_i))
\subseteq B((g_i\gamma g_i^{-1})^{-1}\infty,1).
\end{align}
By \eqref{flower3}, we obtain
\begin{equation}
\label{invcon3}
d(g_i\Delta_{p_i}, (g_i\gamma g_i^{-1})^{-1}\infty)\geq 1/\eta. 
\end{equation}
Combining (\ref{invcon1})-(\ref{invcon3}) together, we conclude that
\begin{equation*}
g_i \gamma^{-1}U_l \subset B((g_i\gamma g_i^{-1})^{-1}\infty,1)\subset B(g_i\Delta_{p_i},1/(2\eta))^c.\qedhere
\end{equation*}
\end{proof}

We prove Proposition \ref{prop:coding} (1) and \eqref{sum}.
The proof is to consider an induced map and reduce the number of cusps by $1$ at a time.

Let $q=p_j$. Denote $\cup_{1\leq i\leq j-1}\Delta_{p_i}=\Delta-\Delta_{q}$ by $X_1$ and for $x\in X_1$, define
\[n_1(x)=\inf\{n\in\N:\, T_0^n(x)\in X_1 \}.\] The map $T_1$ is given by $T_1(x)=T_0^{n_1(x)}(x)$ for $x$ such that $n_1(x)<\infty$. Since $T_0$ is irreducible, this induced system is also irreducible on $X_1$. 
Write
 \begin{align*}
 &\mathcal{H}_q:=\text{the set of the inverse branches of}\,\,\,T_0\,\,\,\text{which are from}\,\,\,\Delta_q\,\,\,\text{to}\,\,\,\Delta_q,\\
 &\mathcal{H}_p:=\text{the set of the inverse branches of}\,\,\,T_0\,\,\,\text{which are from}\,\,\,X_1\,\,\,\text{to}\,\,\,X_1,\\
 &\mathcal{H}_{pq}:=\text{the set of the inverse branches of}\,\,\,T_0\,\,\,\text{which are from}\,\,\,X_1\,\,\,\text{to}\,\,\,\Delta_q,\\
 &\mathcal{H}_{qp}:=\text{the set of the inverse branches of}\,\,\,T_0\,\,\,\text{which are from}\,\,\,\Delta_q\,\,\,\text{to}\,\,\,X_1.
 \end{align*}
 As $T_1$ is a composition of multiples of $T_0$'s, we have
 \begin{itemize}
\item either $T_1(x)=\gamma^{-1}x$ with $\gamma\in \mathcal{H}_p$,
\item or $T_1(x)=\gamma_{n_1(x)}^{-1}\cdots \gamma_1^{-1} x$ with $\gamma_1\in \mathcal{H}_{qp}$, $\gamma_{n_1(x)}\in \mathcal{H}_{pq}$ and $\gamma_l\in \mathcal{H}_q$ for $l=2,\ldots, n_1(x)-1$. 
 \end{itemize}
The string $\gamma_{1}\cdots \gamma_{n_1(x)}$ is an inverse branch of $T_1$. Set
\begin{align*} 
&\calH_1:=\text{the set of all inverse branches of}\,\,\, T_1,\\
 &\mathcal{H}_q^n:=\{\gamma_1\cdots \gamma_n:\,\gamma_i\in\mathcal{H}_q\,\,\,\text{for}\,\,\,1\leq i\leq n\}\,\,\,\text{for every}\,\,\,n\in \mathbb{N}.
\end{align*}
 \begin{lem}\label{lem:gammacoding}
 There exists $C>0$ such that for every $n\in \mathbb{N}$ and for every $\gamma\in\calH_q^n$, we have 
 \begin{equation*}
 |\gamma' (x)|\geq |\gamma'|_\infty /C \,\,\,\text{for any}\,\,\,x\in \Delta_q.
 \end{equation*}
 \end{lem}
 \begin{proof}
 	We first notice that $|\gamma'(x)|\approx |(g_j\gamma g_j^{-1})'(g_jx)|$.
 Write $p=\gamma p_j$. By Lemma~\ref{lem:explicit}, we have
 \[| (g_j\gamma g_j^{-1})' (y)|=\frac{h(g_j p)}{d(y,g_j\gamma^{-1}p_j)^2}.\]
 By \eqref{flower3}, we have $d(g_j\Delta_q,g_j\gamma^{-1}p_j)=d(g_j\Delta_q,(g_j\gamma g_j^{-1})^{-1}\infty)>1/(2\eta)$. Then for every $y\in g_j\Delta_{p_j}=g_j\Delta_q$, the distance
$d(y,g_j\gamma^{-1}p_j)\in [d(g_j\Delta_q,g_j\gamma^{-1}p_j)\pm \operatorname{diam}(g_j\Delta_q)]$, which implies the lemma.
 \end{proof}

\begin{lem}\label{lem:hl}
	There exist $C>0$, $\epsilon>0$ such that for every $l\in\N$
	\begin{equation*}
	\sum_{1\leq k\leq l,\ \gamma_k\in\calH_q}|(\gamma_1\cdots \gamma_l)'|_\infty^{\delta}<C(1-\epsilon)^l.
	\end{equation*}
\end{lem}
\begin{proof}
	Claim: there exists $\epsilon>0$ such that for every $n\in \mathbb{N}$ and for any $h\in \calH_q^n$, we have
	\begin{equation}
	\label{measure decrease}
	\sum_{\gamma\in \calH_q}\mu(h\gamma\Delta_q)\leq (1-\epsilon)\mu(h\Delta_q).
	\end{equation}
	Proof of the claim: for a measurable set $E\subset \Delta_q$, by Lemma~\ref{lem:gammacoding}
	\begin{equation}\label{equ:hE}
	\mu(hE)=\int_E |h'(x)|_{\S^d}^\delta\dd\mu(x)\approx \int_E |h'(x)|^\delta\dd\mu(x) \in \mu(E) |h'|^{\delta}_\infty[1/C,1].
	\end{equation}
	Write $F=\cup_{\gamma\in\calH_q} \gamma\Delta_q$. Since $T_0$ is irreducible, we have $\mu(F)<\mu(\Delta_q)$. By \eqref{equ:hE}
	\[\sum_{\gamma\in\calH_q}\mu(h\gamma\Delta_q)=\mu(hF)\leq |h'|^\delta_\infty \mu(F)=\frac{\mu(F)}{\mu(\Delta_q-F)}|h'|^\delta_\infty \mu(\Delta_q-F)\leq C' \mu(h(\Delta_q-F)).\]
	So we have
	\begin{equation*}
	(1+1/C')\sum_{\gamma\in \mathcal{H}_q}\mu(h\gamma \Delta_q)\leq \mu(hF)+\mu (h(\Delta_q-F))\leq \mu(h\Delta_q).
	\end{equation*}
	
	Using \eqref{measure decrease}, Lemma~\ref{lem:gammacoding} and \eqref{equ:hE} with $E=\Delta_q$, we obtain
	\begin{align*}
	&\sum_{1\leq k\leq l,\ \gamma_k\in\calH_q}|(\gamma_1\cdots \gamma_l)'|_\infty^{\delta}\leq C \sum_{1\leq k\leq l,\ \gamma_k\in\calH_q}\mu(\gamma_1\cdots \gamma_l\Delta_q)\\
	\leq &C \sum_{1\leq k\leq l-1,\ \gamma_k\in\calH_q}(1-\epsilon)\mu(\gamma_1\cdots\gamma_{l-1}\Delta_q) \leq C (1-\epsilon)^l.\qedhere
	\end{align*}
\end{proof}

\begin{proof}[\textbf{Proof of Proposition \ref{prop:coding} (1) and \eqref{sum}}] We first use Proposition \ref{keylemma}, \eqref{h0}, Lemma \ref{lem:gammacoding} and \ref{lem:hl} to prove that for the expanding map $T_1$, we have
\begin{enumerate}
\item There exists $\epsilon_1>0$ such that 
\begin{equation}
\label{induce map 1}
 \sum_{\gamma\in\calH_1}|\gamma'|_\infty^{\delta-\epsilon_1}<\infty, 
 \end{equation}
 where $|\gamma'|_\infty$ is defined as in \eqref{infty}.
 \item There exists $\epsilon>0$ such that for every $n\in \mathbb{N}$,
 \begin{equation}
 \label{induce map 2}
 \mu(\{x\in X_1:\,n_1(x)>n+1\})\ll (1-\epsilon)^n.
 \end{equation}
 \end{enumerate}
 The second statement in particular implies that the map $T_1$ is defined almost everywhere in $\Delta-\Delta_q$.
 
	Due to Lemma~\ref{lem:hl}, we can find a large $l_0$ such that $\sum_{1\leq k\leq l_0,\ \gamma_k\in\calH_q}|(\gamma_1\cdots \gamma_{l_0})'|_\infty^{\delta}<1$. Then using \eqref{h0} and submultiplicativity $|(\gamma_1\gamma_2)'|_\infty\leq |(\gamma_1)'|_\infty|(\gamma_2)'|_\infty$, we obtain 
$$\sum_{1\leq k\leq l_0,\ \gamma_k\in\calH_q}|(\gamma_1\cdots \gamma_{l_0})'|_\infty^{\delta-\epsilon}<\infty,$$ 
where $\epsilon>0$ is the constant given by \eqref{h0}. Hence we can find $0<\epsilon_1<\epsilon$ small such that 
\begin{equation*}
\sum_{1\leq k\leq l_0,\ \gamma_k\in\calH_q}|(\gamma_1\cdots \gamma_{l_0})'|_\infty^{\delta-\epsilon_1}<1.
\end{equation*}
Using submultiplicativity, we obtain constants $C>0,\rho<1$ such that for $l\in\N$
\begin{equation}\label{equ:1kl}
	\sum_{1\leq k\leq l,\ \gamma_k\in\calH_q}|(\gamma_1\cdots \gamma_{l})'|_\infty^{\delta-\epsilon_1}\leq C\rho^l.
\end{equation}

Denote $\sum_{\gamma\in \mathcal{H}_1}|\gamma'|_{\infty}^{\delta-\epsilon_1}$ by $E_q$. For every inverse branch of $T_1$, it can be uniquely decomposed as $\gamma_0\gamma_1\cdots\gamma_{l}\gamma_{l+1}$ with $\gamma_{l+1}\in\calH_{pq}$, $\gamma_{i}\in\calH_q$ with $i=1,\cdots l$ and $\gamma_0\in\calH_{qp}$. 
Using this expression and submultiplicativity,
we obtain
\begin{equation}\label{equ:ep}
E_q\leq \sum_{\gamma\in\calH_p}|\gamma'
|_\infty^{\delta-\epsilon_1}+\sum_{l\geq 1} (\sum_{\gamma_0\in\calH_{qp}} |\gamma_0'|_\infty^{\delta-\epsilon_1})(\sum_{\gamma_{l+1}\in\calH_{pq}} |\gamma_{l+1}'|_\infty^{\delta-\epsilon_1})
(\sum_{\gamma_i\in\calH_q,1\leq i\leq l}|(\gamma_1\cdots \gamma_l)'|_\infty^{\delta-\epsilon_1}).
\end{equation}
 Therefore $E_p$ is also finite due to \eqref{equ:ep}, \eqref{equ:1kl} and \eqref{h0}.
 
The set of $x$ such that $T_0^n(x)$ is outside of domain of definition of $T_0$ for some $n$ has zero PS measure by Proposition~\ref{keylemma}. We only need to consider the set of $x$ such that $T_0^n(x)$ is in the domain of definition for every $n$. 
If $x\in X_1$ with $n_1(x)>n+1$, then $x$ must be in $\gamma_0\gamma_1\cdots \gamma_n\Delta_q$ with $\gamma_0\in\calH_{qp}$ and $\gamma_i\in\calH_q$ for $1\leq i\leq n$. Therefore
Lemma~\ref{lem:hl} implies 
\begin{align*}
&\mu(\{x\in X_1:\,n_1(x)>n+1\})\leq \sum_{\gamma_0\in\calH_{qp}}\sum_{\gamma_i\in\calH_q,1\leq i\leq n}\mu( \gamma_0\gamma_1\cdots\gamma_n\Delta_q)\\
\leq& (\sum_{\gamma_0\in\calH_{qp}} |\gamma_0'|_\infty^{\delta})(\sum_{\gamma_i\in\calH_q,1\leq i\leq n}|(\gamma_1\cdots \gamma_n)'|_\infty^{\delta})\ll (1-\epsilon)^n.
\end{align*}
 
 We keep reducing the number of cusps by considering the set $X_2:=X_1-\Delta_{p_{j-1}}$ and the induced map $T_2:X_2\to X_2$ which is constructed similar to $T_1$: in particular, the inverse branches of $T_2$ are compositions of elements in $\mathcal{H}_1$. Analogs of Lemma~\ref{lem:gammacoding} and \ref{lem:hl} for $T_2$ also hold. The replacements of Proposition \ref{keylemma} and \eqref{h0} are \eqref{induce map 2} and \eqref{induce map 1} respectively. Using these three ingredients, we can show the properties like \eqref{induce map 1} and \eqref{induce map 2} also hold for $T_2$. The proof of Proposition~\ref{prop:coding}(1) and \eqref{sum} will be finished by repeating this.
\end{proof}


Now, we will finish proving the rest of results for the coding except Lemma~\ref{lem:uni} (UNI).

\begin{proof}[\textbf{Proof of Lemma~\ref{lem:l1}}]
	Take 
	\begin{equation}\label{equ:Lambda-}
	\Lambda_{-}=\Lambda_\Gamma\cap \{|x|>1/(2\eta) \}=\Lambda_\Gamma\cap U_1.
	\end{equation}
%
	The contracting map $\gamma$ from $\Delta_0$ to $\Delta_0$ is a composition of maps in $\calH_0$, so the inclusion follows directly from Lemma~\ref{lem:l1'}.
	Write $p=\gamma \infty$. By Lemma~\ref{lem:explicit},
	\[|(\gamma^{-1})'(x)|_{\S^d}=\frac{h(p)}{d(x,p)^2}\frac{1+|x|^2}{1+|\gamma^{-1}x|^2}. \]
	For $x\in\Lambda_{-}$,
	as $p=\gamma\infty\in\Delta_0$, we have
	\begin{equation*}
	\frac{1+|x|^2}{d(x,p)^2}\leq \frac{1+|x|^2}{(|x|-\operatorname{diam}(\Delta_0))^2}. 
	\end{equation*}
	The right hand side of the inequality is around $1$ as $|x|\geq 1/(2\eta)$.
	For $\gamma^{-1}x\in \Lambda_{ -}$, we have $|\gamma^{-1}x|\geq 1/(2\eta)$. Hence $|(\gamma^{-1})'(x)|_{\S^d}\leq\lambda$ for some $\lambda$ independent of $\gamma$.
\end{proof}

\begin{proof}[\textbf{Proof of Proposition~\ref{prop:coding} (3)}]
	By Lemma~\ref{lem:explicit}, we have
	\[|\gamma'(x)|=\frac{h(p)}{d(x,\gamma^{-1}\infty)^2}. \]
	By Lemma~\ref{lem:l1}, we have $\gamma^{-1}\infty\in\Lambda_{-}$, which implies $d(\gamma^{-1}\infty,x)\geq 1/(2\eta)$. Hence 
	\[
	|\gamma'(x)|\leq (2\eta)^2h(p)\leq 4\eta^2.\qedhere\]
\end{proof}

{\textbf{Proof of Proposition~\ref{prop:coding} (4).}} We need the following lemma, which will also be needed in later sections. 
\begin{lem}\label{lem:der}
	Let $\gamma$ be any element in $\Gamma$ which does not fix $\infty$. For any $x\in \Delta_0$ and any unit vector $e\in\R^d$, we have
	\[\partial_e\log | \gamma'(x)|=-\frac{2\langle x-\xi,e \rangle}{|x-\xi|^2} \]
	where $\xi=\gamma^{-1}\infty$.
\end{lem}
\begin{proof}
	It can be shown using Lemma~\ref{lem:explicit} and elementary computation. 
\end{proof}

For $\gamma\in\calH$, Proposition~\ref{prop:coding} (4) can be deduced using Lemma~\ref{lem:der} and the observation that $|\gamma^{-1}\infty|\geq 1/(2\eta)$ (Lemma~\ref{lem:l1} and~\eqref{equ:Lambda-}). 

\subsection{Verifying UNI}\label{sec:UNI}


We prove Lemma~\ref{lem:uni} in this part. 
Let $\Gamma_f$ be the semigroup generated by $\gamma$ in $\calH$ and $\Gamma_b$ be the semigroup generated by $\gamma^{-1}$ with $\gamma\in\calH$. Let $\Lambda_f$ and $\Lambda_b$ be the limit set of $\Gamma_f$ and $\Gamma_b$ on $\partial\H^{d+1}$, that is the set of accumulation points of orbit $\Gamma_f o$ and $\Gamma_b o$ for some $o\in\H^{d+1}$ respectively. It follows from the definition that the limit set $\Lambda_f$ is $\Gamma_f$-invariant and $\Lambda_b$ is $\Gamma_b$-invariant. Due to~\cite[Proposition 3.19]{Kap} (convergence property of M\"obius transformation), we have that $\Lambda_f$ is a $\Gamma_f$-minimal set and $\Lambda_{b}$ is a $\Gamma_b$-minimal set.

\begin{lem}\label{lem:dense}
	The limit set $\Lambda_b$ is not contained in an affine subspace in $\R^d\cup\{\infty \}$ or a sphere in $\R^d$.
\end{lem}
\begin{proof}
	Let $A$ be an affine subspace or a sphere with minimal dimension which contains $\Lambda_b$. Because $\Lambda_b$ is $\Gamma_b$ invariant, the semigroup $\Gamma_b$ must preserve $A$, so does the Zariski closure of $\Gamma_b$. The Zariski closure of a semigroup is a group (see for example~\cite[Lemma 6.15]{BQ}). The Zariski closures of $\Gamma_f$ and $\Gamma_b$ are the same. Hence $\Gamma_f$ also preserves $A$ and $\Lambda_f$ is in $A$. \textbf{We claim: $\mu(\Lambda_f)=\mu(\Lambda_{\Gamma}\cap \Delta_0)>0$}. Then because $\Gamma$ is Zariski dense, by~\cite[Corollary 1.4]{FS}, we conclude that $\mu(A)$ is non zero if and only if $A=\R^d$, finishing the proof.
	
	Proof of the claim: Let $x$ be any point in $\Lambda_{\Gamma}\cap \Delta_0$ such that $T^nx\in\Lambda_{\Gamma}\cap \Delta_0$ for every $n\in\N$. We can write $x=\gamma_n T^n(x)\in \gamma_n\Delta_0$ for some $\gamma_n\in\calH^n$. Fix any $y\in \Lambda_f$, it follows from Proposition~\ref{prop:coding} (3) that $d(\gamma_n y, \gamma_n T^n(x))\to 0$. So $\gamma_n y\rightarrow x$ and $x\in\Lambda_f$. Due to Proposition~\ref{prop:coding} (1), the set of $x$'s such that $T^nx\in\Lambda_{\Gamma}\cap \Delta_0$ for every $n\in\N$ is a conull set in $\Lambda_{\Gamma}\cap \Delta_0$. Hence $\mu(\Lambda_f)=\mu(\Lambda_{\Gamma}\cap \Delta_0)$.
\end{proof}
\begin{lem}
\label{nonvanish}
	For every $x\in\Lambda_\Gamma\cap\overline\Delta_0$, there exist pairs of points $(\xi_{1m},\xi_{2m}),\ m=1,\cdots, k_x$ in the limit set $\Lambda_b$ and $\epsilon_x'>0$ such that for every unit vector $e\in \R^d$ there exists $m$,
	\[|\langle \frac{x-\xi_{1m}}{|x-\xi_{1m}|^2}-\frac{x-\xi_{2m}}{|x-\xi_{2m}|^2},e \rangle|>2\epsilon_x'>0. \]
\end{lem}
\begin{proof}
	The map $inv_x:\xi \mapsto \frac{x-\xi}{|x-\xi|^2}$ is an inversion and this map is injective. If there exists a unit vector $e\in\R^d$ such that 
	\[\langle \frac{x-\xi_{1}}{|x-\xi_{1}|^2}-\frac{x-\xi_{2}}{|x-\xi_{2}|^2},e \rangle=0 \]
	for all $\xi_1,\xi_2$ in $\Lambda_b$, then $inv_x(\Lambda_b)$ is contained in an affine subspace parallel to $e^\perp$. Hence $\Lambda_b$ itself is contained in an affine subspace in $\R^d\cup\{\infty \}$ or a sphere in $\R^d$, which contradicts Lemma~\ref{lem:dense}. Therefore, for every unit vector $e\in\R^d$, there exist $\xi_1,\xi_2$ in $\Lambda_b$ such that 
	\[\langle \frac{x-\xi_{1}}{|x-\xi_{1}|^2}-\frac{x-\xi_{2}}{|x-\xi_{2}|^2},e \rangle\neq0. \]
	We use continuity and compactness to finish the proof.
\end{proof}
\begin{lem}\label{lem:Lam2}
	Let $\xi$ be any point in $\Lambda_b$. For any $\epsilon_2,\epsilon_3>0$, there exists $n_\xi\in\N$ such that for any $n\geq n_\xi$, there exists $\gamma$ in $\calH^n$ satisfying
	\[ d_{\S^d}(\gamma^{-1}\infty,\xi)\leq \epsilon_2,\ |\gamma'|_\infty\leq \epsilon_3. \]
\end{lem}
\begin{proof}
	Since $\Lambda_b$ is $\Gamma_b$ minimal, for any point $\xi'\in\Lambda_b$, there exists a sequence $\{\gamma_n^{-1}\}$ in $\Gamma_b$ such that $\gamma_n^{-1}\xi'$ converges to $\xi$ and $|\gamma_n'|_\infty$ tends to zero.
	By Lemma~\ref{lem:l1}, we know that $\gamma_n^{-1}\Lambda_{-}$ also converges to $\xi$. Hence we can always find a $\gamma$ in $\Gamma_f$ with $|\gamma'|_\infty\leq \epsilon_3$ and $d_{\S^d}(\gamma^{-1}\Lambda_{-},\xi)\leq \epsilon_2$. Let $n_\xi$ be the unique number such that $\gamma\in\calH^{n_\xi}$.
	
	For any $\gamma_1\in \cup_{n\geq 1}\calH^{n}$, we have $|(\gamma_1\gamma)'|_\infty\leq |\gamma_1'|_\infty|\gamma'|_\infty\leq |\gamma'|_\infty$ and 
	\[d_{\S^d}((\gamma_1\gamma)^{-1}\infty,\xi)=d_{\S^d}(\gamma^{-1}(\gamma_1^{-1}\infty),\xi)\leq d_{\S^d}(\gamma^{-1}\Lambda_{-},\xi)\leq\epsilon_2. \]
	Therefore, for any $m= n_\xi+n$, choose any $\gamma_1\in\calH^n$ and then $\gamma_1\gamma\in\calH^m$ and it satisfies Lemma~\ref{lem:Lam2}.
\end{proof}

Combining the above two lemmas, by Lemma~\ref{lem:der} and the formula $R_n(\gamma x)=-\log|\gamma'(x)|$ for $\gamma\in\calH^n$, we have
\[|\partial_e(R_n\circ \gamma_{1m}-R_n\circ \gamma_{2m})(y)|=|\langle \frac{y-\gamma_{1m}^{-1}\infty}{|y-\gamma_{1m}^{-1}\infty|^2}-\frac{y-\gamma_{2m}^{-1}\infty}{|y-\gamma_{2m}^{-1}\infty|^2},e \rangle|. \]
Using this expression, Lemma \ref{nonvanish}, \ref{lem:Lam2} and continuity, we obtain 
\begin{lem}\label{lem:gamma}
	For every $x\in\Lambda_\Gamma\cap\overline\Delta_0$, there exist $\epsilon_x, \epsilon_x'>0$ such that for any $\epsilon_3>0$, there exists $n_x\in\N$ such that the following holds for any $n\geq n_x$. There exist $k_x\in\N$, $\gamma_{im}\in\calH^n$ with $i=1,2$ and $m=1,\ldots, k_x$ satisfying
	\begin{itemize}
	\item $|\gamma'_{im}|<\epsilon_3$ for every $i=1,2$ and $m=1,\ldots,k_x$.
	\item for any unit vector $e\in\R^d$, there exists $m\in\{1,\ldots,k_x\}$ such that for any $y\in B(x,\epsilon_x)$,
	\[|\partial_e(R_n\circ \gamma_{1m}-R_n\circ \gamma_{2m})(y)|\geq\epsilon_x'>0. \]
	\end{itemize}
\end{lem}

\begin{proof}[\textbf{Proof of Lemma~\ref{lem:uni}}]
	For every $x\in \Lambda_\Gamma\cap\overline\Delta_0$, we apply Lemma \ref{lem:gamma} to $x$ and get two constants $\epsilon_x, \epsilon_x'>0$. Since $\Lambda_\Gamma\cap\overline\Delta_0$ is compact, we can find a finite set $\{x_1,\cdots,x_l \}$ such that $\cup B(x_j,\epsilon_{x_j}/2)\supset \Lambda_\Gamma\cap\overline\Delta_0$. Let $\epsilon_0=\inf\{\epsilon_{x_j}' \}$ and $\r=\inf\{\epsilon_{x_j}/2 \}$. Take $\epsilon_3=\epsilon_0/C$
	and $n_0\geq \sup_{1\leq j\leq l}\{n_{x_j} \}$. Then for every $x_j$, there exists a finite set $\{\gamma_{im}\}$ in $\mathcal{H}^{n_0}$ satisfying results in Lemma~\ref{lem:gamma}. We put all these $\gamma_{im}$'s together and this is the finite set in $\mathcal{H}^{n_0}$ described in Lemma \ref{lem:uni}. For any $x\in\Lambda_\Gamma\cap\overline\Delta_0$, it is contained in some $B(x_j,\epsilon_{x_j}/2)$. Then $B(x,\r)\subset B(x_j,\epsilon_{x_j})$. The family $\{\gamma_{im}\}$ for $x_j$ will satisfy nonvanishing condition on $B(x,\r)$, that is for every unit vector $e\in\R^d$ there exists $m$ such that
	for any $y\in B(x,\r)$
	\[|\partial_e(R_{n_0}\circ \gamma_{1m}-R_{n_0}\circ \gamma_{2m})(y)|\geq\epsilon_0>0. \]
	
	Finally, the inequality $|\D\tau_{im}|_\infty\leq C_2$ is due to \eqref{uniform contraction}.
\end{proof}


\section{Spectral gap and Dolgopyat-type spectral estimate}\label{sec:spegap}
In this section, we prove a Dolgopyat-type spectral estimate and the main result is Proposition~\ref{L2contracting}. Our argument is influenced by the one in~\cite{ArMe, AGY, BaVa,Dol, Nau,Sto} and there is some technical variation in the current setting. The proof involves proving a cancellation lemma (Lemma~\ref{key}) and using it to obtain $L^2$ contraction. The rough idea is as follows. Denote the set $\Delta_0\cap \Lambda_{\Gamma}$ by $\Lambda_0$. With the UNI property (Lemma~\ref{lem:uni}) available, for each ball $B(y,r)$ with $y\in \Lambda_0$, one uses the doubling property of the PS measure to find a point $x\in B(y,r)\cap \Lambda_0$ such that cancellation happens on $B(x,r')$. Then, to run the classical argument, one needs to find finitely many such pairwise disjoint balls $B(x_i,r')$'s contained in $\Delta_0$ such that $\sqcup B(x_i,Dr')$ covers $\Delta_0$ for some $D>1$. The difficulty lies in that the balls $B(x_i,r)$'s are produced using PS measure so the position of $B(x_i,r')$'s is in some sense random and some $B(x_i,r')$ may not be fully contained in $\Delta_0$. To overcome this, we find $B(x_i,r')$'s which only cover a subset of $\Delta_0$ and divide the proof of Proposition~\ref{L2contracting} into the cases when the iteration is small and when the iteration is large.

\subsection{Twisted transfer operators}

For $s\in \mathbb{C}$, let $L_s$ be the twisted transfer operator defined by
\begin{equation}
L_s(u)(x)=\sum_{\gamma\in\calH} |\gamma'(x)|^{\delta+s}u(\gamma x).
\end{equation} 
For $u:\Delta_0\to \mathbb{C}$, define
\begin{equation*}
\uLip=\max \{|u|_{\infty},\,\ulip\},
\end{equation*}
where $\ulip=\sup_{x\neq y} |u(x)-u(y)|/d(x,y)$, where $d(\cdot, \cdot)$ is the Euclidean distance. Denote by $\text{Lip}(\Delta_0)$ the space of functions $u:\Delta_0\to \mathbb{C}$ with $\uLip <\infty$. We also introduce a family of equivalent norms on $\text{Lip}(\Delta_0)$:
\begin{equation*}
\ub=\max\{|u|_{\infty},\,\ulip/(1+|b|)\},\,\,\,b\in \mathbb{R}.
\end{equation*}

With Proposition~\ref{prop:coding} available, we obtain the following lemma by a verbatim of the proof of~\cite[Proposition 2.5]{ArMe}.
\begin{lem}
\label{well-defined}
Write $s=\sigma+i b$. 
The family $s\mapsto L_s$ of operators on $\operatorname{Lip}(\Delta_0)$ is continuous on $\{s\in\C:\ \sigma>-\epsilon_o\}$, where $\epsilon_o$ is given as in Proposition~\ref{prop:coding} (4). Moreover, $\sup_{|\sigma|<\epsilon_o} \lVert L_s\rVert_b<\infty$.
\end{lem}

\vspace{2mm}
Define the PS measure $\mu_E$ on $\Delta_0$ with respect to the Euclidean metric by
\[\dd\mu_E(x)=(1+|x|^2)^\delta\dd \mu(x). \]
Using the quasi-invariance of the PS measure $\mu$, we obtain that the dual operator of $L_0$ preserves the measure $\mu_E$ by a straightforward computation. Our main result of Dolgopyat argument is the following $L^2$ contracting proposition.
\begin{prop}\label{L2contracting}
	There exist $C>0,\ \beta <1, \epsilon>0$ and $b_0>0$ such that for all $v$ in $\operatorname{Lip}(\Delta_0)$, $m\in\N$ and $s=\sigma+ib$ with $|\sigma|\leq\epsilon$ and $|b|>b_0$, we have
	\begin{equation*}
	\int|L_s^{m}v|^2\dd\mu_E\leq C\beta^m\|v\|^2_b.
	\end{equation*}
\end{prop}
The proof will be given at the end of Section~\ref{sec:L2contractoin}.

Recall that $\nu$ is the unique $T$-invariant ergodic probability measure on $\Delta_0\cap \Lambda_{\Gamma}$ which is absolutely continuous with respect to the PS measure $\mu$ with a positive Lipschitz density function $\bar f_0$. So $\nu$ is also absolutely continuous with respect to $\mu_{E}$ with a positive Lipschitz density function $f_0$. Based on these, it is a classical result that the operator $L_0$ acts on $\operatorname{Lip}(\Delta_0)$ and has a spectral gap and a simple isolated eigenvalue at $1$ with $f_0$ the corresponding eigenfunction. 

For $\sigma\in \mathbb{R}$ close enough to $0$, $L_{\sigma}$ acting on $\text{Lip}(\Delta_0)$ is a continuous perturbation of $L_0$ (see Lemma~\ref{well-defined}). Hence, it has a unique eigenvalue $\lambda_{\sigma}$ close to $1$, and the corresponding eigenfunction $f_{\sigma}$ (normalized so that $\int f_{\sigma}=1$) belongs to $\text{Lip}(\Delta_0)$, strictly positive, and tends to $f_0$ in $\text{Lip}(\Delta_0)$ as $\sigma\to 0$. Choose a sufficiently small $\epsilon\in (0,\epsilon_o)$ such that for $\sigma\in (-\epsilon,\epsilon)$, $f_{\sigma}$ is well defined and
\begin{equation*}
1/2\leq \lambda_{\sigma}\leq 2,\,\,\, f_{0}/2\leq f_{\sigma}\leq 2 f_{0},\,\,\, |f_0|_{\text{Lip}}/2\leq |f_{\sigma}|_{\text{Lip}}\leq 2|f_0|_{\text{Lip}}.
\end{equation*}
For $s=\sigma+ib$ with $|\sigma|<\epsilon$ and $b\in \mathbb{R}$, define a modified transfer operator $\tilde{L}_{s}$ by
\begin{equation}
\tilde{L}_{s}(u)=(\lambda_{\sigma}f_{\sigma})^{-1}L_{s}(f_{\sigma}u).
\end{equation}
It satisfies $\tilde{L}_{\sigma}1=1$, and $|\tilde{L}_{s}u|\leq \tilde{L}_{\sigma}|u|$.

\begin{lem}[Lasota-Yorke inequality]
\label{Lasota-Yorke}
There is a constant $C_{\sixt}>1$ such that
\begin{equation}
|\LL^n_s v|_{\operatorname{Lip}} \leq C_{\sixt} (1+|b|) |v|_{\infty} +C_{\sixt} \lambda^n |v|_{\operatorname{Lip}}
\end{equation}
holds for any $s=\sigma+ib$ with $|\sigma|<\epsilon$, and all $n\geq 1,\,v\in \operatorname{Lip}(\Delta_0)$, where $\lambda$ is given as in Proposition~\ref{prop:coding}.
\end{lem}
The proof of this lemma is a verbatim of proof of \cite[Lemma 2.7]{ArMe}. The following lemma can be deduced from 
Lemma~\ref{Lasota-Yorke} by a straightforward computation.
\begin{lem}\label{lem:Lb}
We have $\lVert \LL^n_s\rVert_b\leq 2C_{\sixt}$ for all $s=\sigma+ib$ with $|\sigma|<\epsilon$ and all $n\geq 1$.
\end{lem}

\subsection{Cancellation lemma}

The main result of this subsection is the cancellation lemma (Lemma~\ref{key}) and the proof is inspired by the proof of analogous results in~\cite{Nau} and~\cite{Sto}. We start with detailing all the constants. 

Let $C_{\sev}$ be the constant which will be specified in \eqref{cone const 3}. We define the cone
\begin{defn}
For $b\in \mathbb{R}$, let
\begin{align*}
\mathcal{C}_b=&\{(u,v):\,u,\,v\in \text{Lip}(\Delta_0),\,u>0,\,0\leq |v|\leq u,\,|\log u|_{\text{Lip}}\leq C_{\sev} |b|,\\
&|v(x)-v(y)|\leq C_{\sev} |b| u(y)d(x,y)\,\,\,\text{for all}\,\,\,x,y\in \Delta_0\}.
\end{align*}
\end{defn}

Let $r>0$ and $\epsilon_0>0$ be the same constants as the ones in Lemma~\ref{lem:uni}. We apply Lemma~\ref{lem:uni} with $C=16C_{\sev}$. Let $n_0$ be a sufficiently large integer which satisfies Lemma~\ref{lem:uni} and the inequality
\begin{equation}
\label{cone const 1}
\lambda^{n_0}C_{\sev}(1+\operatorname{diam}(\Delta_0))\leq 1.
\end{equation} 
Let $\gamma_{mj}$, with $m=1,2$, $j=1,\ldots,j_0$ be the inverse branches given by Lemma~\ref{lem:uni}.

Let $k\in\N$ be such that
\begin{equation}
\label{equ:k}
k\epsilon_0>16(C_2+\epsilon_0),
\end{equation}
where $C_2$ is given in \eqref{constant c2}.

Note that the measure $\nu$ is absolutely continuous with respect to the PS measure $\mu$. Let $D>0$ be such that for all $x\in\Lambda_\Gamma\cap \Delta_0$ and $r'\leq 1/C_{\three}$ with $C_{\three}$ given in Proposition \ref{double}
\begin{equation}
\label{equ:D}
 \nu(B(x,Dr'))>\nu(B(x,(k+2)r')). 
\end{equation}

Let $\eps_2>0$ be such that
\begin{equation}
\label{equ:epsilon3}
(2C_{\sev}\eps_2+1/4)e^{2C_{\sev}\eps_2}\leq 3/4. 
\end{equation}

Let $\eps_3>0$ be such that
\begin{equation}\label{equ:epsilon2}
	\eps_3(D+2)< \min\{\eps_2,\ r,\ 1/C_{\three}\},\ \eps_3(D+2)(C_2+\epsilon_0)<3\pi/2,\ \eps_3k \epsilon_0<\pi.
\end{equation}

Recall the notation $\tau_{mj}$ introduced in Lemma \ref{lem:uni}. For $s=\sigma+ib\in\C$, define
$$A_{s,\gamma_{mj}}(v)(x)=e^{(s+\delta)\tau_{mj}(x)}f_\sigma(\gamma_{mj}x)v(\gamma_{mj}x).$$


\begin{lem}\label{key}
	There exists $0< \eta_0<1$ such that the following holds. For $s=\sigma+ib$ with $|\sigma|\leq \epsilon$, $|b|>1$, for $(u,v)\in C_b$, and for any $y\in\Lambda_0$, there exists $x\in B(y,\eps_3D/|b|)\cap \Lambda_\Gamma$ such that we have the following:
	there exists $j\in \{1,\ldots,j_0\}$ such that one of the following inequalities holds on $B(x,\eps_3/|b|)$:
	\begin{align*}
	\textbf{type } \gamma_{1j}:\ 	|A_{s,\gamma_{1j}}(v)+A_{s,\gamma_{2j}}(v)|\leq\eta_0A_{\sigma,\gamma_{1j}}(u)+A_{\sigma,\gamma_{2j}}(u),\\
	\textbf{type } \gamma_{2j}:\ 	|A_{s,\gamma_{1j}}(v)+A_{s,\gamma_{2j}}(v)|\leq A_{\sigma,\gamma_{1j}}(u)+\eta_0 A_{\sigma,\gamma_{2j}}(u).
	\end{align*}
\end{lem}
We first prove a quick estimate. 
\begin{lem}\label{lem:inf}
Let $\eps_2$ be the constant defined in \eqref{equ:epsilon3}.
	For any $|b|>1$, for $(u,v)\in C_b$ and for a ball $Z$ of radius $\eps_2/|b|$, we have
	\begin{enumerate}
	\item
	$\inf_Zu\geq e^{-2C_{\sev}\eps_2}\sup_Zu;$
	\item 
	either $|v|\leq \frac{3}{4}u $
	for all $x\in Z$
	or $|v|\geq \frac{1}{4}u $
	for all $x\in Z$. 
\end{enumerate}
\end{lem}
\begin{proof}
	The first inequality is due to $|\log u(x)-\log u(y)|
	\leq C_{\sev}|b||x-y|$ for every $x,y\in\Delta_0$.
	
	Suppose there exists $x_0\in Z$ such that $|v(x_0)|\leq \frac{1}{4}u(x_0)$. Then by \eqref{equ:epsilon3}
	\begin{align*}
	|v(x)|&\leq |v(x)-v(x_0)|+\frac{1}{4}u(x_0)\leq C_{\sev}|x-x_0||b|\sup_Zu+\frac{1}{4}\sup_Zu\\
	&\leq (2C_{\sev}\eps_2+\frac{1}{4})\sup_Zu\leq (2C_{\sev}\eps_2+\frac{1}{4})e^{2C_{\sev}\eps_2}\inf_Z u\leq \frac{3}{4}u(x).\qedhere
	\end{align*}
\end{proof}

\begin{proof}[Proof of Lemma~\ref{key}]
It follows \eqref{equ:D} that there exists $x_0\in (B(y,\eps_3 D/|b|)-B(y,(k+2)\eps_3/|b|))\cap\Lambda_\Gamma$. 
	 Let $B_1=B(y,\eps_3/|b|)$, $B_2=B(x_0,\eps_3/|b|)$ and $\hat{B}$ the smallest ball containing $B_1\cup B_2$. For all $x\in B_1,x'\in B_2$, we have
	\begin{equation}\label{equ:xx'}
	d(x,x')\in \frac{\eps_3}{b}[k,D+2].
	\end{equation}
	In view of \eqref{equ:epsilon2}, the radius of $\hat{B}$ is smaller than $\eps_2/|b|$ and it is contained in $B(y,r)$. Let $e_0=(y-x_0)/|y-x_0|$.
	
	 By Lemma~\ref{lem:uni} for the point $y$ there exists $j$ in $\{1,\cdots,j_0\}$ such that \eqref{equ:uni} holds for $B(y,r)$ with $e=e_0$. From now on, $j$ is fixed, so we abbreviate $(\gamma_{1j},\gamma_{2j})$ to $(\gamma_1,\gamma_2)$ and $(\tau_{1j},\tau_{2j})$ to $(\tau_1,\tau_2)$.
	
	Due to $|\gamma_m'|_\infty\leq \lambda\leq 1$, the radius of $\gamma_{m}\hat{B}$ is smaller than $\epsilon_{2}/|b|$. So we can apply Lemma~\ref{lem:inf} to $\gamma_{m}\hat{B}$ and we have that 
	either
	$|v(\gamma_mx)|\geq\frac{1}{4} u(\gamma_mx)$ for all $x\in \hat{B}$ or $|v(\gamma_mx)|\leq\frac{3}{4} u(\gamma_mx)$ for all $x\in \hat{B}$. 	Suppose that 
	\[|v(\gamma_mx)|\leq \frac{3}{4}u(\gamma_mx) \]
	holds for some $m\in\{1,2 \}$ for all $x\in \hat B$. Then Lemma~\ref{key} can be proved by a straightforward computation.

	Suppose that for $x\in \hat{B}$ and $m=1,2$
	\begin{equation}\label{equ:notsmall}
	|v(\gamma_mx)|\geq \frac{1}{4}u(\gamma_mx).
	\end{equation}
	
	\textbf{Claim:} Under the assumption of \eqref{equ:notsmall}, there exists $C_{\eig}>0$ independent of $b$ and $(u,v)$ such that for $l\in \{1,2\}$, we have 
	\begin{equation}\label{claim}
	\text{either}\,\,\,\left|\frac{A_{s,\gamma_1}(v)}{A_{s,\gamma_2}(v)}\right|\leq C_{\eig}\text{ for all }x\in B_l\text{ or }\left|\frac{A_{s,\gamma_2}(v)}{A_{s,\gamma_1}(v)}\right|\leq C_{\eig}\text{ for all }x\in B_l. 
	\end{equation}
	\begin{proof}[Proof of the claim]
	 Fix any $x_0\in \Delta_0$. Due to $|\tau_m'|_{\infty}\leq C_2$ (see \eqref{uniform contraction}), we have for any $x\in \hat{B}$, $$|\tau_1(x)-\tau_2(x)|\leq |\tau_1(x_0)-\tau_2(x_0)|+2C_2|x-x_0|.$$ Hence there exists a constant $C({\tau_1,\tau_2})$ depending on $\tau_1,\tau_2$ such that
	\[\left|\frac{A_{s,\gamma_1}(v)}{A_{s,\gamma_2}(v)}\right|\leq C(\tau_1,\tau_2)\frac{f_\sigma(\gamma_1x)u(\gamma_1x)}{f_\sigma(\gamma_2x)u(\gamma_2x)}. \]
	For the middle term,
	\[\frac{f_\sigma(\gamma_1x)}{f_\sigma(\gamma_2x)}\leq 4\frac{\sup f_0}{\inf f_0}. \]
	Since the radius of $\gamma_2B_l$ is less than $\eps_2/|b|$, using Lemma~\ref{lem:inf}, we have for every $x$ in $B_l$
	\[ \frac{u(\gamma_1x)}{u(\gamma_2x)}\leq \frac{\sup_{B_l}u(\gamma_1)}{\inf_{B_l}u(\gamma_2)}\leq e^{2C_{\sev}\eps_2}\frac{\sup_{B_l}u(\gamma_1)}{\sup_{B_l}u(\gamma_2)}. \]
	Putting these together, we have
	\[\left|\frac{A_{s,\gamma_1}(v)}{A_{s,\gamma_2}(v)}\right|\leq C_{\eig}\frac{\sup_{B_l}u(\gamma_1)}{\sup_{B_l}u(\gamma_2)} \]
	where $C_{\eig}=4C(\tau_1,\tau_2)e^{2C_{\sev}\eps_2}\frac{\sup f_0}{\inf f_0}$.	We have a similar inequality for $\left|\frac{A_{s,\gamma_2}(v)}{A_{s,\gamma_1}(v)}\right|$. Note that either $\frac{\sup_{B_l}u(\gamma_1)}{\sup_{B_l}u(\gamma_2)}\leq 1$ or $\frac{\sup_{B_l}u(\gamma_2)}{\sup_{B_l}u(\gamma_1)}\leq 1$. The proof of the claim finishes.
\end{proof}

	Now we start to compute the angle and our definitions are only for $x\in\hat{B}$. The function $\arg(v(\gamma_mx))$ is well defined because $|v(\gamma_mx)|\geq u(\gamma_mx)/4>0$. Let 
	\[\Theta(x)=b(\tau_1(x)-\tau_2(x)), \ V(x)=\arg(v(\gamma_1x))-\arg(v(\gamma_2x)), \]
	and let $$\Phi(x)=\Theta(x)+V(x).$$
	We apply Lemma~\ref{lem:uni} to $\hat{B}$ and obtain that for $x\in \hat{B}$, 
	\[|\partial_{e_0}\Theta(x)|\geq |b|\epsilon_0, \,\,\,|\Theta'(x)|\leq |b|C_2. \]
	For the angle function, 
	by \eqref{equ:notsmall} and \eqref{equ:hm}, we have for $i\in\{1,2\}$ and $x,x'\in\hat{B}$ 
	\begin{align*}
	|\arg v(\gamma_ix)-\arg v(\gamma_ix')|&=|\operatorname{Im}(\log v(\gamma_ix)-\log v(\gamma_ix'))|\leq \frac{|v(\gamma_ix)-v(\gamma_ix')|}{|v(\gamma_ix)|}\\
	&\leq C_{\sev}|b|\frac{u(\gamma_ix)|\gamma_ix-\gamma_ix'|}{|v(\gamma_ix)|}\leq |b|\epsilon_0|x-x'|/4.
	\end{align*}
	This implies that for $x,x'\in\hat{B}$
	\[ |V(x)-V(x')|\leq |b|\epsilon_0|x-x'|/2. \]
	Combining the estimates for $\Theta$ and $V$, we obtain for $x,x'\in\hat{B}$
	\begin{equation}\label{equ:phileq}
		|\Phi(x)-\Phi(x')|\leq b(C_2+\epsilon_0)|x-x'|,
	\end{equation}
	and for $x, x+te_0\in \hat{B}$ with $t\in\R^+$,
	\[|\Phi(x)-\Phi(x+te_0)|\geq b\epsilon_0 t/2.\]
	Hence for $x_1=y,\ x_2=x_0$ which are the centers of $B_1$ and $B_2$ respectively, by \eqref{equ:xx'},
	\begin{equation}\label{equ:phixx'}
	|\Phi(x_1)-\Phi(x_2)|\in \eps_3[k\epsilon_0/2,(D+2)(C_2+\epsilon_0)].
	\end{equation}
	Let 
	$\epsilon_4=\eps_3k\epsilon_0/8.$
	We claim that there exists $l\in \{1,2\}$ such that
	\begin{equation}\label{equ:phiz}
		d(\Phi(x_l),2\pi\Z)>\epsilon_4.
	\end{equation}
	If not so, then both the distance from $\Phi(x_1)$ to $2\pi \Z$ and that from $\Phi(x_2)$ to $2\pi\Z$ are less than $\epsilon_4$. By \eqref{equ:phixx'} and \eqref{equ:epsilon2}
	\[|\Phi(x_1)-\Phi(x_2)|\leq \eps_3(D+2)(C_2+\epsilon_0) \leq 3\pi/2< 2\pi-2\epsilon_4. \]
	Hence $\Phi(x_1),\Phi(x_2)$ are in a ball $(2n\pi-\epsilon_4,2n\pi+\epsilon_4)$ with $n\in\Z$. This implies that 
	\[|\Phi(x_1)-\Phi(x_2)|\leq 2\epsilon_4=\eps_3k\epsilon_0/4, \]
	contradicting with \eqref{equ:phixx'}. 
	
	Without loss of generality, we may assume \eqref{equ:phiz} holds for $x_1$. For any $x$ in the ball $B_1$, by \eqref{equ:phileq} and \eqref{equ:k}
	\[|\Phi(x)-\Phi(x_l)|\leq (C_2+\epsilon_0)\eps_3\leq k\eps_3\epsilon_0/16=\epsilon_4/2. \]
	Combined with \eqref{equ:phiz}, we have
	\begin{equation}
	\label{equ:awayinteger}
		d(\Phi(x),2\pi\Z)\geq\epsilon_4/2. 
	\end{equation}
	In conclusion, there exists $l\in \{1,2\}$ such that for all $x\in B_l$, $d(\Phi(x),2\pi\Z)>\epsilon_4/2$ and \eqref{claim} holds. Without loss of generality, we may assume $|A_{s,\gamma_1}(v)(x)/A_{s,\gamma_2}(v)(x)|\leq C_{\eig}$ for all $x\in B_l$. By an elementary inequality \cite[Lemma 5.12]{Nau}, there exists $0<\eta_0<1$ depending on $\epsilon_4$ and $C_{\eig}$ such that on $B_l$
	\begin{equation*}
	|A_{s,\gamma_1}(v)+A_{s,\gamma_2}(v)|\leq \eta_0|A_{s,\gamma_1}(v)|+|A_{s,\gamma_2}(v)|\leq\eta_0A_{\sigma,\gamma_1}(u)+A_{\sigma,\gamma_2}(u).\qedhere
		\end{equation*}
\end{proof}

For $b$ with $|b|\geq 1$, let
\begin{equation}\label{equ:deltab}
	\Delta_b=\{x\in\Delta_0|\ d(x,\partial\Delta_0)>\frac{\eps_3(D+1)}{|b|} \}.
\end{equation}
For any $(u,v)\in \calC_b$, we can find $\{x_i\}_{1\leq i\leq l_0}\subset\Lambda_0:=\Lambda_{\Gamma}\cap \Delta_0$ such that $B(x_i,\eps_3/|b|)$'s are disjoint balls contained in $\Delta_0$, 
\[\Lambda_0\cap\Delta_b\subset \cup_{1\leq i\leq l_0}B(x_i,2\eps_3D/|b|), \]
 and on each $B(x_i,\eps_3/|b|)$ one of the $2j_0$ inequalities in Lemma~\ref{key} holds. In fact, suppose we have already found some points $x_i$'s but $\cup B(x_i,2\eps_3D/|b|)$ don't cover the set $\Lambda_0\cap\Delta_b$. Then for a point $y\in \Lambda_0\cap \Delta_b-\cup B(x_i,2\eps_3D/|b|)$, we apply Lemma~\ref{key} to $y$ and obtain a point $x\in B(y,\eps_3 D/|b|)\cap \Lambda_0$ such that Lemma~\ref{key} holds on $B(x,\eps_3/|b|)$. Moreover, the ball $B(x,\eps_3/|b|)$ is contained in $\Delta_0$ and it is disjoint from $\cup B(x_i,\eps_3/|b|)$. 


Let $B_i=B(x_i,\eps_3/|b|)$ and $\tilde{B_i}=B(x_i,\eps_3/(3|b|))$ for $i=1,\cdots,l_0$. Let $\eta\in[\eta_0,1)$ and define a $C^1$ function $\chi:\Delta_0\rightarrow[\eta,1]$ as follows: 
it equals 1 outside of $\cup_{m,j,i} \gamma_{mj} B_i$; for each $B_i$, if $B_i$ is of type $\gamma_{mj}$, let $\chi(\gamma_{mj}(y))=\eta$ for $y\in \tilde{B_i}$ and $\chi\equiv 1$ on other $\gamma_{m'j'}B_i$. We can choose $\eta$ close to 1 and independent of $b$ such that $|\chi'(x)|\leq |b|$ for all $x\in \Delta_0$.

\begin{cor}
\label{cone ineq}
	Under the same assumptions as in Lemma~\ref{key}, for $(u,v)\in \mathcal{C}_b$ and $\chi=\chi(b,u,v)$ a $C^1$ function described as above, we have
	\begin{equation*}
	\label{equ:con ineq}
	 |\LL_s^{n_0} v|\leq \LL_\sigma^{n_0}(\chi u). 
	\end{equation*}
\end{cor}

Define $J_i=B(x_i,2\eps_3D/|b|)$ for $i=1,\cdots,l_0$ and let $\tilde{B}=\cup\tilde{B_j}$.

\begin{prop}
\label{contr int}
	Suppose that $w$ is a positive Lipschitz function with $|\log w(x)-\log w(y)|\leq K|b||x-y|$ for some $K>0$. Then
	\begin{equation}
	\label{equ:contr int}
	\int_{\tilde{B}}w\dd\nu \geq \epsilon_4\int_{\Delta_b} w\ \dd\nu, 
	\end{equation}
	with $\epsilon_4=\epsilon_5 e^{-4\eps_3DK}$, where $\epsilon_5$ comes from doubling property only depending on $D$ and $\nu$.
\end{prop}
\begin{proof}
	Since $\cup_i J_i$ covers $\Delta_b$, it is sufficient to prove for each $i$ we have a similar inequality.
	Due to hypothesis, we obtain $\inf_{\tilde{B_i}}w\geq e^{-4\eps_3DK}\sup_{J_i}w$. By doubling property, there exists $\epsilon_5$ depending on $D$ such that
	\[\nu(\tilde{B_i})\geq\epsilon_5\nu(J_i). \]
	Therefore
	\[\int_{\tilde{B_i}}w\ \dd\nu\geq \nu(\tilde{B_i})\inf_{\tilde{B_i}}w\geq \epsilon_5\nu(J_i) e^{-4\eps_3DK}\sup_{J_i}w\geq \epsilon_4\int_{J_i}w\ \dd\nu.\qedhere \]
\end{proof}

\subsection{Invariance of Cone Condition}
We define the constants
\begin{align}
\label{cone const 2}
&C_{\sev}'=16(\delta+\epsilon)C_2|f_0|_{\infty} |f_0^{-1}|_{\infty}+16 |f_0^{-1}|_{\infty}|f_0|_{\text{Lip}}+4C_2+2,\\
\label{cone const 3}
&C_{\sev}=\max\{8|f_0^{-1}|_{\infty} |f_0|_{\text{Lip}} +(\delta+3)C_2+1+4|f_0|_{\infty}|f_0^{-1}|_{\infty}C_{\sev}', 6C_{\sixt}\}.
\end{align}

\begin{lem}
\label{cone invariance}
Let $C_{\sev}>0$ be the constant defined in (\ref{cone const 3}) and $n_0$ be the constant defined in (\ref{cone const 1}). For $s=\sigma+ib$ with $|\sigma|<\epsilon$ and $|b|>1$, for $(u,v)\in \mathcal{C}_b$, 
 we have 
\begin{equation}
(\LL^{n_0}_{\sigma}(\chi u),\,\LL^{n_0}_s v)\in \mathcal{C}_b,
\end{equation}
where $\chi=\chi (b,u,v)$ is the same as the one in Corollary~\ref{cone ineq}. 
\end{lem}
The proof is a verbatim of the proof of~\cite[Lemma 2.12]{ArMe}.

\subsection{$L^2$ contraction for bounded iterations}
In this part, we will prove Proposition~\ref{L2contracting} for the case when $m$ bounded by $\log |b|$. Compared with~\cite{AGY}, where they can finish the proof of an analog of Proposition~\ref{L2contracting} at this stage, we have the difficulty about the boundary. More precisely, Proposition \ref{contr int} is one of the ingredients to obtain Proposition \ref{L2contracting}. Now the integration region of the right hand side of \eqref{equ:contr int} is $\Delta_b$, which is smaller than $\Delta_0$, so it just enables us to obtain $L^2$ contraction for bounded iterations. For large iteration, we will use a Lipschitz contraction lemma (Lemma \ref{lem:Lipcontracting}) to obtain $L^2$ contraction in the next subsection.

\begin{lem}\label{lem:vLip}
	For $|b|>1$ and $v\in \operatorname{Lip}(\Delta_0)$, if $|v|_{\operatorname{Lip}}\geq C_{\sev} |b| |v|_{\infty}$, then
	\[\|\LL_s^{n_0}v\|_b\leq \frac{9}{10}\|v\|_b. \]
\end{lem}
\begin{proof}
	We have
	\[|\LL_s^{n_0}v|_\infty\leq |v|_\infty\leq \frac{1}{C_{\sev}|b|}|v|_{Lip}\leq \frac{2}{C_{\sev}}\|v\|_b. \]
	By Lemma~\ref{Lasota-Yorke}, we obtain
	\begin{align*}
|\LL_s^{n_0}v|_{\operatorname{Lip}}& \leq C_{\sixt}(1+|b|)|v|_\infty+C_{\sixt}\lambda^{n_0}|v|_{\operatorname{Lip}}\leq (1+|b|)(\frac{C_{\sixt}(1+|b|)}{C_{\sev}|b|}+C_{\sixt}\lambda^{n_0})\|v\|_b\\&\leq (1+|b|)(\frac{1}{3}+\frac{1}{6})\|v\|_b=(1+|b|)\frac{1}{2}\|v\|_b,
	\end{align*}	
	where the last inequality is due to $C_{\sev}\geq 6C_{\sixt}$ and $\lambda^{n_0}C_{\sev}\leq 1$.
\end{proof}

\begin{lem}
\label{L2bounded}

There exist $C_{\nin}>0$ and $\beta<1$ such that for all $s=\sigma+ib$ with $|\sigma|<\epsilon$ and $|b|$ large enough and $m\leq [C_{\nin}\log|b|]$
\begin{equation}\label{equ:L2Linfinty}
\int |\LL^{mn_0}_{s} v|^2\dd\nu \leq \beta^m\|v\|_b^2.
\end{equation}
\end{lem}

\begin{proof}
If for all $0\leq p\leq m-1$, we have $|\LL_s^{pn_0}v|_{\text{Lip}}\geq C_{\sev} |b| |\LL_s^{pn_0}v|_{\infty}$, then by Lemma~\ref{lem:vLip},
\[\int |\LL_s^{mn_0}v|^2\dd\nu\leq \|\LL_s^{mn_0}v\|_b^2\leq (\frac{9}{10})^m\|v\|_b^2. \]

Otherwise, suppose $p$ is the smallest integer such that $|\LL_s^{pn_0}v|_{\operatorname{Lip}}\leq C_{\sev}|b||\LL_s^{pn_0}v|_\infty$. We consider $v'=\LL_s^{pn_0}v$. Then Lemma~\ref{lem:vLip} implies $\|v'\|_b\leq (\frac{9}{10})^p\|v\|_b$. We only need to show that 
\[\int|\LL_s^{(m-p)n_0}v'|^2\dd\nu\leq \beta^{m-p}\|v'\|_b^2. \]
We reduce to the case when $p=0$, that is $|v|_{\operatorname{Lip}}\leq C_{\sev}|b||v|_\infty$.
Define $u_0\equiv 1,\,v_0=v/|v|_{\infty}$ and induitively,
\begin{equation*}
u_{m+1}=\LL^{n_0}_{\sigma}(\chi_{m}u_{m}),\,\,\, v_{m+1}=\LL^{n_0}_{s}(v_m),
\end{equation*}
where $\chi_m=\chi (b,u_m,v_m)$. It is immediate that $(u_0,v_0)\in \mathcal{C}_b$, and it follows from Lemma~\ref{cone invariance} that $(u_m,v_m)\in \mathcal{C}_b$ for all $m$. Hence in particular the $\chi_m$'s are well defined.

We will show that there exist $\beta_1\in (0,1)$, $\kappa>0$ and $C>0$ such that for all $m$
\begin{equation}
\label{induction eq}
\int u_{m+1}^2 \dd\nu \leq \beta_1 \int u_m^2 \dd\nu+C|b|^{-\kappa}.
\end{equation}
Then note that
\begin{equation*}
|\LL^{mn_0}_s v|=|v|_{\infty} |\LL^{mn_0}_s v_0|=|v|_{\infty} |v_m|\leq |v|_{\infty} u_m.
\end{equation*}
As a result,
\begin{equation*}
\int |\LL^{mn_0}_s v|^2\dd\nu\leq |v|_{\infty}^2 \int u_m^2 \dd\nu \leq |v|^2_{\infty}(\beta_1^m \int u^2_0 \dd\nu+C|b|^{-\kappa}\sum_{0\leq l\leq m-1}\beta_1^l ) \leq(\beta_1^m+C|b|^{-\kappa}/(1-\beta_1)) |v|_{\infty}^2.
\end{equation*}
We can find $C_{\nin}>0$ and $\beta<1$ such that for any large enough $|b|$, \eqref{equ:L2Linfinty} holds for all $m\leq [C_{\nin}\log|b|]$.

Now we prove \eqref{induction eq}. By definition
\begin{align*}
u_{m+1}(x)&=\lambda_{\sigma}^{-n_0} f_{\sigma}^{-1}(x) \sum_{\gamma\in \calH^{n_0}} |\gamma'(x)|^{\delta+\sigma} f_{\sigma}(\gamma x) \chi_{m}(\gamma x) u_m(\gamma x)\\
&= \lambda_{\sigma}^{-n_0} f_{\sigma}^{-1}(x)\sum_{\gamma\in \calH^{n_0}} \left(|\gamma'(x)|^{\delta/2} f_{\sigma}^{1/2}(\gamma x) u_m(\gamma x)\right) \left(|\gamma'(x)|^{\delta/2+\sigma} f_{\sigma}^{1/2}(\gamma x) \chi_m(\gamma x)\right),
\end{align*}
so by Cauchy-Schwarz
\begin{align*}
u_{m+1}^2(x)\leq& (\lambda_{\sigma}^{-n_0} f_{\sigma}(x))^{-2}\left(\sum_{\gamma\in \calH^{n_0}} |\gamma'(x)|^{\delta} f_{\sigma}(\gamma x) u_{m}^2(\gamma x)\right) \left(\sum_{\gamma\in \calH^{n_0}} |\gamma'(x)|^{\delta+2\sigma} f_{\sigma}(\gamma x) \chi_m^2(\gamma x)\right)\\
\leq & \xi (\sigma) \LL^{n_0}_0 (u^2_m) \LL^{n_0}_{2\sigma} (\chi_m^2)
\end{align*}
where (noting that $\lambda_0=1$)
\begin{equation*}
\xi (\sigma) =(\lambda^{-2}_{\sigma} \lambda_{2\sigma})^{n_0} \left|\frac{f_0}{f_{\sigma}}\right|_{\infty} \left|\frac{f_{2\sigma}}{f_{\sigma}}\right|_{\infty} \left| \frac{f_{\sigma}}{f_0}\right|_{\infty} \left|\frac{f_{\sigma}}{f_{2\sigma}}\right|_{\infty}.
\end{equation*}
As in Proposition~\ref{contr int}, we write $\Delta_0=\tilde{B}\sqcup \tilde{B}^c$. Let $\mathcal{H}_c$ be the set of inverse branches given by Lemma~\ref{lem:uni}. If $y\in \tilde{B}$, then there exists $\gamma_i\in \mathcal{H}_{c}$ such that
\begin{align*}
\LL^{n_0}_{2\sigma}(\chi^2_{m})(y)\leq & \lambda^{-n_0}_{2\sigma} f_{2\sigma} (y)^{-1} \left(\eta^2 |\gamma'_i(y)|^{\delta+2\sigma} f_{2\sigma}(\gamma_iy)+\sum_{\gamma\in \calH^{n_0}\backslash \{\gamma_i\}} |\gamma'(y)|^{\delta+2\sigma} f_{2\sigma}(\gamma y) \right)\\
=& \LL^{n_0}_{2\sigma} (1)(y)-(1-\eta^2) \lambda^{-n_0}_{2\sigma} f_{2\sigma}(y)^{-1} |\gamma'_i(y)|^{\delta+2\sigma}f_{2\sigma}(\gamma_iy)\\
\leq & 1-(1-\eta^2) 2^{-(n_0+2)} \inf f_0 \cdot |f_0|_{\infty}^{-1} \cdot \inf_{\{\gamma_i\in \mathcal{H}_c\}} |\gamma'_i|^{\delta+2\sigma}=\eta_1<1,
\end{align*}
In this way we obtain that there exists $\eta_1<1$ such that 
\begin{equation*}
u^2_{m+1}(y)\leq 
\begin{cases}
\eta_1 \xi (\sigma) \LL^{n_0}_0 (u_m^2)(y), & y\in \tilde{B},\\
\xi (\sigma) \LL^{n_0}_0 (u_m^2)(y), & y\in \tilde{B}^c.
\end{cases}
\end{equation*}

Since $(u_m,v_m)\in \mathcal{C}_b$, it follows in particular that $|\log u_m|_{\text{Lip}}\leq C_{\sev} |b|$. Hence by \eqref{cone const 1},
\begin{equation*}
u_m^2(\gamma x)/u_m^2(\gamma y)\leq \exp (2C_{\sev}\lambda^{n_0} |b|d(x,y))\leq \exp (2|b|d(x,y)).
\end{equation*}
Let $w=\LL^{n_0}_0(u_m^2)$. Then
\begin{align*}
\frac{w(x)}{w(y)}=\frac{f_0(y)\sum_{\gamma\in \calH^{n_0}}|\gamma'(x)|^{\delta} f_0(\gamma x)u_m^2(\gamma x) }{f_0(x) \sum_{\gamma\in \calH^{n_0}} |\gamma'(y)|^{\delta} f_0(\gamma y) u_m^2(\gamma y)}\leq \exp \left(\left(2|f_0^{-1}|_{\infty} |f_0|_{\text{Lip}} +\delta C_2 +2|b|\right) d(x,y)\right).
\end{align*}
Hence $|\log w|_{\text{Lip}}\leq K|b|$ with $K=2|f_0^{-1}|_{\infty} |f_0|_{\text{Lip}} +\delta C_2+2$. Using Proposition~\ref{contr int}, we have
\begin{equation*}
(1-\eta_1) \int_{\tilde{B}} w\dd\nu \geq \epsilon_4 (1-\eta_1)\int_{\Delta_b} w \dd\nu. 
\end{equation*}
Setting $\beta'=1-\epsilon_4(1-\eta_1)$, we can further write
\begin{equation*}
\eta_1\int_{\tilde{B}} w\dd\nu+\int_{\Delta_b-\tilde{B}} w\dd\nu\leq \beta' \int_{\Delta_b} w\dd\nu\leq \beta'\int_{\Delta_0} w \dd\nu.
\end{equation*}
Hence
\begin{align}
\int_{\Delta_b} u^2_{m+1}\dd\nu \leq & \xi (\sigma) \left(\eta_1\int_{\tilde{B}} \LL^{n_0}_0(u^2_m)\dd\nu +\int_{\Delta_b-\tilde{B}} \LL^{n_0}_0 (u^2_m)\dd\nu\right)\nonumber\\
\label{bounded contraction 1}
\leq &\xi(\sigma) \beta' \int_{\Delta_0} \LL^{n_0}_0(u^2_m)\dd\nu =\xi (\sigma) \beta' \int_{\Delta_0} u^2_m\dd\nu. 
\end{align}
By \eqref{equ:bou}, \eqref{equ:deltab} and $|u_{m+1}|\leq 1$,
\begin{equation}
\label{bounded contraction 2}
	\int_{\Delta_0-\Delta_b}u_{m+1}^2 \dd\nu\leq \nu(\Delta_0-\Delta_b)\leq C|b|^{-\kappa}.
\end{equation}

Finally we can shrink $\epsilon$ if necessary so that $\xi (\sigma)\beta'<1$ for $|\sigma|<\epsilon$ and then \eqref{bounded contraction 1} and \eqref{bounded contraction 2} imply \eqref{induction eq}.
\end{proof}

\subsection{Proof of Proposition~\ref{L2contracting}}\label{sec:L2contractoin}
\begin{lem}\label{lem:Linfty2}
There exist $\epsilon\in (0,1),\, \tau\in (0,1)$ and $C_{\twi}>0$ such that for all $s=\sigma+ib$ with $|\sigma|<\epsilon$, $n\geq 1$ and $v\in \operatorname{Lip}(\Delta_0)$, we have
\begin{equation*}
|\LL^n_s v|^2_{\infty} \leq C_{\twi} (1+|b|)\tau^n |v|_{\infty} \lVert v\rVert_b+C_{\twi}B^n|v|_{\infty} \int |v| \dd\nu
\end{equation*} 
where $B>1$ is a constant depending on $\epsilon$ and it tends to $1$ as $\epsilon\to 0$.
\end{lem}

\begin{proof}
We have
\begin{align*}
|\LL^n_s v(x)|\leq & \lambda_{\sigma}^{-n} f_{\sigma}^{-1}(x) \sum_{\gamma\in \calH^n} |\gamma'(x)|^{\delta+\sigma} f_{\sigma}(\gamma x) |v|(\gamma x)\\
=& \lambda^{-n}_{\sigma} f_{\sigma}^{-1}(x) \sum_{\gamma\in \calH^n} \left(|\gamma'(x)|^{\delta/2+\sigma} f_{\sigma}^{1/2}(\gamma x) |v|^{1/2}(\gamma x)\right) \left(|\gamma'(x)|^{\delta/2} f^{1/2}_{\sigma}(\gamma x) |v|^{1/2}(\gamma x)\right).
\end{align*}
Using Cauchy-Schwarz, we obtain
\begin{align*}
|\LL^n_s v(x)|^2 \leq (\lambda_{\sigma}^{-2}\lambda_{2\sigma})^n \xi(\sigma) \LL^n_{2\sigma}(|v|) (x) \cdot \LL^n_0(|v|)(x),
\end{align*}
where $\xi(\sigma)=|f_0/f_{\sigma}|_{\infty} |f_{2\sigma}/f_{\sigma}|_{\infty} |f_{\sigma}/f_0|_{\infty} |f_{\sigma}/f_{2\sigma}|_{\infty}\leq 64$. Hence
\begin{equation}
\label{bound}
|\LL^n_s v|^2_{\infty} \leq 64 B^n |v|_{\infty} |\LL^n_0(|v|)|_{\infty},
\end{equation}
where $B>1$ is a constant depending on $\epsilon$ with $B\to 1$ as $\epsilon\to 0$.

Since $\LL_0$ is a normalized transfer operator for the uniformly expanding map $T$, there exists $\tau_1\in (0,1)$ such that $|\LL^n_0 v|_{\infty}\leq C\tau_1^n\lVert v\rVert_{\text{Lip}}$ for all $v\in \text{Lip}(\Delta_0)$ with $\int v \dd\nu=0$. (This is a consequence of spectral gap of quasi-compact operator $\LL_0$.) Hence by decomposing $|v|$ into $(|v|-\int |v|\dd\nu)+\int |v|\dd\nu$, we obtain
\begin{equation*}
| \LL^n_0(|v|)|_{\infty} \leq 2C\tau_1^n \lVert v\rVert_{\text{Lip}}+\int |v|\dd\nu.
\end{equation*}
Substituting into (\ref{bound}), we have
\begin{equation*}
|\LL^n_s v|_{\infty}^2\leq 128 C(B\tau_1)^n (1+|b|)|v|_{\infty} \lVert v\rVert_b +64 B^n |v|_{\infty} \int |v| \dd\nu.
\end{equation*}
Finally, shrink $\epsilon$ if necessary so that $\tau=B\tau_1<1$.
\end{proof}

\begin{lem}\label{lem:Lipcontracting}
There exist $C>0,\ \epsilon \in (0,1),\, A>0$ and $\beta\in (0,1)$ such that
\begin{equation*}
\lVert \LL^{mn_0}_s v\rVert_b\leq C\beta^m \lVert v\rVert_b
\end{equation*} 
for all $m\geq A\log |b|,\, s=\sigma+ib$ with $|\sigma|<\epsilon$ and $|b|$ large enough, and all $v\in \operatorname{Lip}(\Delta_0)$.
\end{lem}

\begin{proof}
Let $N=[C_{\nin}\log|b|]n_0$. Using Lemma~\ref{lem:Linfty2} for $\LL_s^N v$ and $n=lN$, Lemma~\ref{lem:Lb} and \eqref{equ:L2Linfinty}, we obtain
\begin{align*}
	|\LL_s^{(l+1)N}v |_\infty^2&\leq C_{\twi}(1+|b|)\tau^{lN}|\LL_s^N v|_\infty\|\LL_s^Nv\|_b+C_{\twi}B^{lN}|\LL_s^Nv|_\infty (\int |\LL_s^Nv|^2\dd \nu)^{1/2}.
	\\
	&\leq 2C_{\sixt}C_{\twi}(1+|b|)\tau^{lN}|v|_\infty\|v\|_b+2C_{\sixt}C_{\twi}B^{lN}|v|_\infty \beta^{N/2}\|v\|_b.
\end{align*}
We fix $l$ depending on $\tau, C_{\nin}$ and $n_0$ such that $(1+|b|)\tau^{lN/2}\leq 1$. Then by shrinking $B$ if necessary, there exists $\beta_1<1$, such that
\begin{equation}\label{equ:l+1}
|\LL_s^{(l+1)N}v|_\infty\leq \beta_1^{(l+1)N}\|v\|_b.
\end{equation}

For Lipschitz norm, we have 
\begin{align*}
	|\LL_s^{(l+2)N}v|_{\operatorname{Lip}}&\leq C_{\sixt}(1+|b|)|\LL_s^{(l+1)N}v|_\infty+C_{\sixt}\lambda^N|\LL_s^{(l+1)N}v|_{\operatorname{Lip}}\\
	&\leq C_{\sixt}(1+|b|)\beta_1^{(l+1)N}\|v\|_b+C_{\sixt}^2\lambda^N((1+|b|)|v|_\infty+\lambda^{(l+1)N}|v|_{\operatorname{Lip}})\\
	&\leq C_{\sixt}^2(1+|b|)\|v\|_b(\beta_1^{(l+1)N}+\lambda^N+\lambda^{(l+2)N})\leq 3C_{\sixt}^2(1+|b|)\beta_2^N\|v\|_b,
\end{align*}
for some $\beta_2<1$, where we use Lemma~\ref{Lasota-Yorke} to get the first inequality and \eqref{equ:l+1} to get the second one. For the infinity norm, by \eqref{equ:l+1} and Lemma~\ref{lem:Lb}, we obtain
\[|\LL_s^{(l+2)N}v|_\infty\leq 2C_{\sixt}\beta_1^{(l+1)N}\|v\|_b. \]
Combining these two norm estimates, we obtain
\begin{equation}\label{equ:l+2}
\|\LL_s^{(l+2)N}v\|_b\leq C_{\sixt}^2(2\beta_1^{(l+1)N}+3\beta_2^{N})\|v\|_b\leq \beta_3^{(l+2)N/n_0}\|v\|_b, 
\end{equation}
for some $\beta_3<1$ if $|b|$ is large enough to absorb the constant $6C_{\sixt}^2$.

Let $A=2(l+2)C_{\nin}$ and $N_1=(l+2)N/n_0=(l+2)[C_{\nin}\log|b|]\leq A\log|b|$. For $m\geq A\log|b|$, we can write $m=dN_1+r$ with $r\in\N$ and $r< N_1$. Therefore by \eqref{equ:l+2} and Lemma~\ref{lem:Lb},
\begin{equation*}
	\|\LL_s^{mn_0}v\|_b=\| \LL_s^{dN_1n_0}(\LL_s^{rn_0}v)\|_b\leq\beta_3^{dN_1}\|\LL_s^{rn_0}v\|_b\leq 2C_{\sixt}\beta_3^{dN_1}\|v\|_b\leq 2C_{\sixt}(\sqrt{\beta_3})^m\|v\|_b.\qedhere
\end{equation*}

\end{proof}

\begin{proof}[\textbf{Proof of Proposition~\ref{L2contracting}}]
	It is sufficient to prove that for all $m\in \mathbb{N}$,
	\begin{equation}\label{LLs}
		\int |\LL_s^{mn_0}v|^2\dd\nu\leq C\beta^m\|v\|^2_b.
	\end{equation}
	Then for any $k\in \N$, suppose $k=mn_0+r$ with $0\leq r<n_0$. We have
	\[\int |L_s^kv|^2\dd\mu_E\leq C\lambda_\sigma^{2k}\int |\LL_s^k(f_{\sigma}^{-1}v)|^2\dd\nu\leq C\lambda_\sigma^{2k} \beta^m\|\LL_s^r(f_{\sigma}^{-1}v)\|_b^2\leq C\lambda_\sigma^{2k}\beta^m\|f_{\sigma}^{-1}v\|_b^2\leq C\lambda_{\sigma}^{2k}\beta^m \|v\|_b^2.\]
	By choosing $\epsilon$ small such that $\lambda_\sigma^{2n_0}\beta<1$ for any $|\sigma|<\epsilon$, we obtain Proposition~\ref{L2contracting}.
	
	It remains to prove \eqref{LLs}. For $m> A\log|b|$, by Lemma~\ref{lem:Lipcontracting}, we obtain
	\[\int|\LL_s^{mn_0}v|^2\dd\nu\leq \|\LL_s^{mn_0}v \|_b^2\leq C\beta^m\|v\|_b^2. \]
	
	For $A\log|b|\geq m\geq C_{\nin}\log|b|$, by \eqref{equ:L2Linfinty} and Lemma~\ref{lem:Lb}, we know
	\[\int|\LL_s^{mn_0}v|^2\dd\nu\leq \beta^{[C_{\nin}\log|b|]}\|\LL_s^{(m-[C_{\nin}\log|b|])n_0}v \|_b^2\leq 2C_{\sixt}\beta^{[C_{\nin}\log|b|]}\|v\|^2_b\leq 2C_{\sixt}\beta_1^m\|v\|^2_b
	\]
	for some $\beta_1=\beta^{C_{\nin}/A}<1$. 
	
	The case when $m\leq C_{\nin}\log |b|$ has been verified in Lemma~\ref{L2bounded}.
\end{proof}

\section{Exponential mixing}
\label{sec:expmix}

In this section, we prove Theorem~\ref{thm:skew}. As a first step, an analogous result concerning expanding semiflow will be proved. Let $T:\Lambda_+\to \Lambda_+$ be the uniformly expanding map and $R:\Lambda_+\to \mathbb{R}_+$ be the roof function as defined in Proposition~\ref{prop:coding}. Set $\Lambda_+^{R}=\{(x,t)\in \Lambda_+\times \mathbb{R}: 0\leq t<R(x)\}$. We define a semi-flow $T_t:\Lambda_+^R\to \Lambda_+^R$ by $T_s(x,t)=(T^nx, t+s-R_n(x))$ where $n$ is the unique integer satisfying $R_n(x)\leq t+s<R_{n+1}(x)$. Recall that $\nu$ is the unique $T$-invariant ergodic probability measure on $\Lambda_+$. Then the flow $T_t$ preserves the probability measure $\nu^{R}=\nu\times \operatorname{Leb}/(\nu\times \operatorname{Leb})(\Lambda_+^R)$. We will also use the probability measure $\mu_E^{R}=\mu_E\times \operatorname{Leb}/(\mu_E\times \operatorname{Leb})(\Lambda_+^R)$ on $\Lambda_+^R$. We show that $T_t$ is exponentially mixing. 


For a bounded function on $\Lambda_+^R$, we define two norms. Set
\begin{align*}
&\|U\|_{\calB_0}=|U|_\infty+\sup_{ (x,a)\neq(x',a')\in\Lambda_+^R}\frac{|U(x,a)-U(x',a')|}{d(x,x')+|a-a'|},\\
&\|V\|_{\calB_1}=|V|_\infty+\sup_{x\in \Lambda_+} \frac{\operatorname{Var}_{(0,R(x))}\{t\mapsto V(x,t) \}}{R(x)},
\end{align*}
where $\operatorname{Var}_{(0,R(x))}\{t\mapsto V(x,t) \}$ is the total variation of the function $t\mapsto V(x,t)$ on the interval $(0,R(x))$.
\begin{thm}\label{semiflow}
There exist $C>0,\ \epsilon>0$ such that for all $t>0$ and for any two functions $U,\ V$ on $\Lambda_{+}^R$ with $\|U\|_{\calB_0},\ \|V\|_{\calB_1}$ finite, we have
\[\left|\int U\cdot V\circ T_t\dd\mu_E^R-\left(\int U\dd\mu_E^R\right)\left(\int V\dd\nu^R\right)\right|\leq Ce^{-\epsilon t}\|U\|_{\calB_0}\|V\|_{\calB_1}. \]
\end{thm}

\begin{rem}
Applying this theorem to the function $U(x,t)\frac{\dd\nu^R}{\dd\mu_E^R}(x)$, we obtain
\begin{equation}
\label{semiflownu}
\left|\int U\cdot V\circ T_t\dd\nu^R-\left(\int U\dd\nu^R\right) \left(\int V\dd\nu^R\right)\right|\leq Ce^{-\epsilon t}\|U\|_{\calB_0}\|V\|_{\calB_1}. 
\end{equation}
\end{rem}

With Proposition~\ref{L2contracting} available, Theorem~\ref{semiflow} can be proved essentially along the same lines as the proof of \cite[Theorem 7.3]{AGY} (see also \cite[Section 7.5]{AGY}). We provide a sketch of the proof here. For a pair of functions $U,V$, let $\rho(t)=\int U\cdot V\circ T_t\dd\mu_E^R$ be the correlation function and the observation is that the Laplace transform of $\rho$, denoted by $\hat{\rho}$, can be expressed as a sum of twisted transfer operators $L_s$ \cite[Lemma 7.17]{AGY}. One shows that $\hat{\rho}$ admits an analytic continuation to a neighborhood of each point $s=ib$ and this part of the argument uses the quasi-compactness of the twisted transfer operators \cite[Lemma 7.21, 7.22]{AGY}. When $|b|$ is large, the Dolgopyat-type estimate (Proposition~\ref{L2contracting}), which is a replacement of~\cite[Proposition 7.7]{AGY} in the current setting, is used to imply that $\hat{\rho}$ admits an analytic extension to a strip $\{s=\sigma+ib\in \mathbb{C}: |\sigma|<\sigma_0\}$ for all sufficiently small $\sigma_0$ \cite[Corollary 7.20]{AGY}. The result of exponential mixing then follows from the classical Paley-Wiener theorem \cite[Theorem 7.23]{AGY}.

The difference between our result and that in~\cite{AGY} is the classes of functions in concern. The only adjustment we need to make is~\cite[Lemma 7.18]{AGY}, which is a norm estimate for $C^1$ functions in their paper, but for functions with finite $\calB_0$ norm in the current setting. The precise statement is as follows. For a function $U:\Lambda_+\to \mathbb{R}$ with $\| U\|_{\calB_0}<\infty$ and $s\in \mathbb{C}$, set $\hat{U}_{s}(x)=\int_0^{R(x)} e^{-ts}U(x,t)\dd t $.
\begin{lem}
	There exists $C>0$ such that for $s=\sigma+ib$ with $|\sigma|\leq\epsilon_o/4$ ($\epsilon_o$ is given as in Proposition \ref{prop:coding} (5)), the function $L_s\hat{U}_{-s}$ is Lipschitz on $\Delta_0$ and
	\[\|L_s\hat{U}_{-s}\|_b\leq \frac{C\|U\|_{\calB_0}}{\max\{ 1,|b|\}}. \]
\end{lem}
\begin{proof}
	We first prove for $x\in\Lambda_{+}$ we have
	\begin{equation*}
	\label{laplaceU}
	 |\hat{U}_{-s}(x)|\leq\frac{Ce^{\epsilon_oR(x)/2}}{\max\{1,|b|\}}\|U\|_{\calB_0}. 
	 \end{equation*}
	By definition, we have
	\begin{equation*}
		\hat{U}_{-s}(x)=\int_0^{R(x)}U(x,t)e^{ts}\dd t. 
	\end{equation*}	
	The case when $|b|\leq 1$ is easy. When $|b|>1$, one uses integration by parts and the fact that $U$ is Lipschitz with respect to $t$ to obtain
	\[ |\hat{U}_{-s}(x)|\leq (2|U|_\infty e^{\epsilon_oR(x)/4}+|U|_{\operatorname{Lip}}R(x)e^{\epsilon_o R(x)/4})/\max\{1,|b|\}.\]
	Then
	\[|L_s\hat{U}_{-s}|\leq \frac{C\|U\|_{\calB_0}}{\max\{1,|b|\}}L_\sigma(e^{\epsilon_oR/2}). \]
	Observe that by \eqref{sum}
	\begin{equation}
	\label{boundedeigen}
	 L_\sigma(e^{\epsilon_oR/2})=\sum_{\gamma\in\calH}|\gamma'(x)|^{\delta+\sigma}e^{\epsilon_o R(\gamma x)/2}\leq \sum_{\gamma\in\calH}|\gamma'(x)|^{\delta-3\epsilon_o/4}<\infty.
	 \end{equation}
	 So 	$|L_s\hat{U}_{-s}|\leq \frac{C\|U\|_{\calB_0}}{\max\{1,|b|\}} $.
	
	We estimate the Lipschitz norm of $L_s\hat{U}_{-s}$. We have
	\[L_s\hat{U}_{-s}(x)-L_s\hat{U}_{-s}(y)=\sum_{\gamma\in\calH}|\gamma'(x)|^{\delta+s}(\hat{U}_{-s}(\gamma x)-\hat{U}_{-s}(\gamma y))+(|\gamma'(x)|^{\delta+s}-|\gamma'(y)|^{\delta+s})\hat{U}_{-s}(\gamma y). \]
	The term $|\gamma'(x)|^{\delta+s}-|\gamma'(y)|^{\delta+s}$ can be estimated using Proposition~\ref{prop:coding} (4). 
	For the term $\hat{U}_{-s}(\gamma x)-\hat{U}_{-s}(\gamma y)$, suppose that $R(\gamma x)\geq R(\gamma y)$, we use Proposition~\ref{prop:coding} (4) again and get
	\begin{align*}
	|\hat{U}_{-s}(\gamma x)-\hat{U}_{-s}(\gamma y)|&\leq |R(\gamma x)-R(\gamma y)||U|_\infty e^{\sigma R(\gamma x)}+\int_0^{R(\gamma y)}|U(\gamma x,t)-U(\gamma y,t)|e^{t\sigma}\dd t\\
	&\leq (C_1 e^{\sigma R(\gamma x)}+R(\gamma y)e^{\sigma R(\gamma y)})\|U\|_{\calB_0}d(x,y). 
	\end{align*}
	Then we use \eqref{boundedeigen} to conclude that there exists some $C$ (independent of $U$) such that
	\[|L_s\hat{U}_{-s}|_{\operatorname{Lip}}\leq C \|U\|_{\calB_0}.\qedhere \]
\end{proof}

\begin{proof}[\textbf{Proof of Theorem~\ref{thm:skew}}]
Now Theorem~\ref{thm:skew} can be proved using the same lines as the proof of~\cite[Theorem 2.7]{AGY} (see also~\cite[Section 8.2]{AGY}). In particular, in the proof of~\cite[Lemma 8.3]{AGY}, we use \eqref{semiflownu} to replace~\cite[Theorem 7.3]{AGY} and Proposition~\ref{prop:dis} (2) to relate the measures $\hat{\nu}^{R}$ and $\nu^R$.
\end{proof}

\section{Resonance-free region}
\label{sec:res}
Recall that $\Gamma$ is a geometrically finite discrete subgroup in $G=\operatorname{SO}(d+1,1)^{\circ}$. We begin by defining the measures $m^{\operatorname{BR}}$, $m^{\operatorname{BR_*}}$ and $m^{\operatorname{Haar}}$. Recall the definition of the BMS measure on $\operatorname{T}^1(\mathbb{H}^{d+1})\cong \partial^2(\mathbb{H}^{d+1})\times \mathbb{R}$:
\begin{equation*}
\dd\tilde{m}^{\operatorname{BMS}}(x,x_-,s)=e^{\delta \beta_x(o,x_*)} e^{\delta \beta_{x_-}(o,x_*)} \dd\mu(x) \dd \mu(x_-)\dd s,
\end{equation*}
where $x_*$ is the based point of the unit tangent vector given by $(x,x_-,s)$. We define the measures $\tilde{m}^{\operatorname{BR}}$, $\tilde{m}^{\operatorname{BR_*}}$ and $\tilde{m}^{\operatorname{Haar}}$ on $\operatorname{T}^1(\mathbb{H}^{d+1})\cong \partial^2(\mathbb{H}^{d+1})\times \mathbb{R}$ similarly as follows:
\begin{align*}
\dd\tilde{m}^{\operatorname{BR}}(x,x_-,s)&=e^{d\beta_x(o,x_*)} e^{\delta \beta_{x_-}(o,x_*)} \dd m_o(x) \dd \mu(x_-)\dd s;\\
\dd\tilde{m}^{\operatorname{BR_*}}(x,x_-,s)&=e^{\delta \beta_x(o,x_*)} e^{\delta \beta_{x_-}(o,x_*)} \dd\mu(x) \dd m_o(x_-)\dd s;\\
\dd\tilde{m}^{\operatorname{Haar}}(x,x_-,s)&=e^{d \beta_x(o,x_*)} e^{d\beta_{x_-}(o,x_*)} \dd m_o(x) \dd m_o(x_-)\dd s,
\end{align*}
where $m_o$ is the unique probability measure on $\partial(\mathbb{H}^{d+1})$ which is invariant under the stabilizer of $o$ in $G$.

 These measures are all left $\Gamma$-invariant and induce measures on $\operatorname{T}^1(\Gamma\backslash \mathbb{H}^{d+1})$, which we will denote by $m^{\operatorname{BR}}$, $m^{\operatorname{BR_*}}$ and $m^{\operatorname{Haar}}$ respectively. Here we do not normalize the BMS measure to a probability measure, which is different from the previous part.
 
 By \cite[Theorem 5.8]{OhWi}, Theorem \ref{main thm} implies exponential decay of matrix coefficients. 
\begin{thm}
\label{thm:matrix}
	There exists $\eta>0$ such that for any compactly supported functions $\phi, \psi\in C^1(\T^1(M))$, we have
	\begin{equation*}
	e^{(d-\delta)t}\int_{\T^1(M)} \phi\cdot\psi\circ\calG_t\ \dd m^{\operatorname{Haar}}=\frac{m^{\operatorname{BR_*}} (\phi) m^{\operatorname{BR}} (\psi)}{\bms(\T^1(M))}+O(\lVert \phi \rVert_{C^1} \lVert \psi\rVert_{C^1}e^{-\eta t})
	\end{equation*}
	for all $t>0$, where $O$ depends on the supports of $\phi,\psi$.
\end{thm}
For $x,y\in\H^{d+1}$ and $T>0$, let 
\[N(T,x,y)=\#\{\gamma\in\Gamma\,|\,d(x,\gamma y)\leq T \}, \]
where $d$ is the hyperbolic distance on $\H^{d+1}$. 
In \cite{MoOh}, it was shown that Theorem \ref{thm:matrix} implies the following:
\begin{cor}\label{cor:counting}
	There exists $\eta>0$ such that for any $x,y\in\H^{d+1}$ and $T>0$, we have
	\[N(T,x,y)=c_{x,y}e^{\delta T}+O(e^{(\delta-\eta)T}), \]
	where $c_{x,y}>0$ is a constant depending on $x,y$.
\end{cor}

\begin{proof}[\textbf{Proof of Corollary~\ref{cor:resonance}}]	
	For $x,y\in\H^{d+1}$ and $s\in\C$ with $\Re s>\delta$, let $P_s(x,y)$ be the Poincar\'e series defined by
	\[P_s(x,y)=\sum_{\gamma\in\Gamma} e^{-sd(x,\gamma y)}. \]
	We first prove that $P_s(x,y)$ is meromorphic on $\Re s>\delta-\eta$ with a unique pole $s=\delta$.
	By Fubini's theorem
	\[P_s(x,y)=\int_{0}^\infty \frac{1}{s}e^{-sT}N(T,x,y)\dd T=\int_{0}^\infty \frac{1}{s}e^{-(s-\delta)T}c_{x,y}\dd T+\int_{0}^\infty \frac{1}{s}e^{-sT}(N(T,x,y)-c_{x,y}e^{\delta T})\dd T. \]
	The first part is a meromorphic function of $s$ having a unique pole at $s=\delta$. The second part, it follows from Corollary~\ref{cor:counting} that it is absolutely convergence if $\Re s>\delta-\eta$, hence it is analytic on $\Re s>\delta-\eta$. Then we use ~\cite[Theorem 7.3]{GM} to deduce that the resolvent $R_M(s)$ is also analytic on $\{s\in \mathbb{C}:\, \delta-\eta<\Re s<\delta\}$. 
\end{proof}

\section{Appendix: proof of Lemma \ref{lem:jpr}}

\begin{proof}[Proof of Lemma \ref{lem:jpr}]
	We divide into the cases when $r$ lies in different intervals. Let $\beta=C\sqrt{\eta}$.
	\begin{itemize}
		\item Case $A$: {{$r\leq\eta h(p)$}}
		
		By Lemma~\ref{lem:explicit}, we have $|(\gamma^{-1})'x|=h(p)/d(x,p)^2$. Using Lemma~\ref{lem:jp}, we have
		\begin{equation}
		\label{location}
		N_r(\partial J_p)\subset B(p, 2\eta h(p))-B(p,c_{\four}\eta h(p)).
		\end{equation}
		Hence for $x\in N_r(\partial J_p)$
		\begin{equation*}
		|(\gamma^{-1})'x|\in [1/(4\eta^2h(p)),1/(c_{\four}^2\eta^2 h(p))].
		\end{equation*}
		Then
		\begin{equation}\label{equ:gammaJ}
		N_{r/(4\eta^2h(p))} (\partial\gamma^{-1}J_p)\subset \gamma^{-1}N_r(\partial J_p) \subset N_{r/(c^2_\four\eta^2h(p))}(\partial\gamma^{-1}J_p).
		\end{equation}
		Notice that $\partial\gamma^{-1} J_p=\partial (B_Y(2/\eta)\times R_{p,\eta})$ and $R_{p,\eta}$ is a parallelotope tiled by the translations of $\overline\Delta_0'$.

		\begin{itemize}
			\item Case $A_1$: {$r\leq \eta^2h(p)$}
			
			Recall that Lemma \ref{lem:boud} is proved using Lemma \ref{lem:part}. Using the same argument, we obtain an analog of Lemma \ref{lem:boud} for $N_r(\partial \gamma^{-1} J_p)$. Using this version of Lemma~\ref{lem:boud} with $\epsilon=4\beta/c_{\four}^2$, the inequality $r/(4\eta^2 h(p))< 1$ and \eqref{equ:gammaJ}, we have
			\begin{align*}
			\mu( \gamma^{-1}N_{\beta r}(\partial J_p))\leq \mu(N_{\beta r/(c_{\four}^2 \eta^2 h(p))}(\partial\gamma^{-1}J_p))\leq \lambda \mu(N_{r/(4\eta^2 h(p))}(\partial\gamma^{-1}J_p))\leq \lambda\mu( \gamma^{-1}N_r(\partial J_p)).
			\end{align*}
			Using Lemma~\ref{lem:annulusquasi}, we obtain
			\begin{equation*}
			\frac{\mu(N_{\beta r}(\partial J_p))}{\mu(N_r(\partial J_p))}\leq C\frac{\mu(\gamma^{-1}N_{\beta r}(\partial J_p))}{\mu(\gamma^{-1} N_{r}(\partial J_p))} \leq C\lambda,
			\end{equation*}
			where $\lambda$ tends to zero as $\eta$ tends to zero.
			
			\item Case $A_2$: {$\eta^{3/2} h(p)<r\leq \eta h(p)$}
			
					We compute the measure by counting the number of translations of $\Delta_0$. By (\ref{location}) and Lemma \ref{lem:explicit}, we obtain $\gamma^{-1}N_r(\partial J_p)\subset B(p',1/(c_{\four}\eta))-B(p',1/(2\eta))$. Let $\gamma' \Delta_0$ be any fundamental domain contained in $\gamma^{-1}N_r(\partial J_p)$ with $\gamma'\in \Gamma_{\infty}$. Using Lemma~\ref{lem:quasi-gammainfinity}, we obtain that $\mu(\gamma'\Delta_0)\approx \eta^{2\delta}\mu(\Delta_0)$.
		By Lemma~\ref{lem:annulusquasi}, \eqref{equ:gammaJ}
		\begin{equation*}
		\mu(N_{\beta r}(\partial J_p))\ll h(p)^{\delta}\mu(N_{\beta r/(c_{\four}^2 \eta^2 h(p))}(\partial\gamma^{-1}J_p)).
		\end{equation*}
		By counting the number of fundamental domains, we obtain that the region $N_{\beta r/(c_{\four}^2 \eta^2 h(p))}(\partial\gamma^{-1}J_p)$ can be covered by $(1/\eta)^{k-1} \cdot (\beta r/(c_{\four}^2 \eta^2 h(p)))$ disjoint rectangles $\gamma'\Delta_{0}$ with $\gamma'\in\Gamma_{\infty}$. So we have
		\begin{equation*}
		\mu(N_{\beta r}(\partial J_p))\ll h(p)^{\delta}\cdot(1/\eta)^{k-1} \cdot (\beta r/(c_{\four}^2 \eta^2 h(p)))\cdot\eta^{2\delta}\mu(\Delta_0).
		\end{equation*}
		By Lemma~\ref{lem:annulusquasi}, \eqref{equ:gammaJ}
		\begin{equation*}
		\mu(N_r(\partial J_p))\gg h(p)^{\delta}\mu(N_{r/(4\eta^2 h(p))}(\partial\gamma^{-1}J_p)).
		\end{equation*}
		Meanwhile, as $r/(4\eta^2h(p))\geq 1/4\eta^{1/2}$, the number of rectangles $\gamma'\Delta_{0}$ inside $N_{r/(4\eta^2 h(p))}(\partial\gamma^{-1}J_p)$ is greater than $(1/\eta)^{k-1}\cdot (r/(4\eta^2 h(p)))$. Hence we have
		\begin{equation*}
		\mu(N_r(\partial J_p))
		\gg h(p)^{\delta}\cdot (1/\eta)^{k-1}\cdot (r/(4\eta^2 h(p)))\cdot \eta^{2\delta}\mu(\Delta_0).
		\end{equation*}
		Therefore,
		\[ \mu(N_{\beta r}(\partial J_p)\ll \beta \mu(N_r(\partial J_p)). \]
		
		\end{itemize}
		\item Case $B$: {$\eta^{1/2} h(p)\leq r\leq h(p)$}
		
		We handle this case using \eqref{equ:bpr}. By Lemma~\ref{lem:jp} and the inequality $h(p)/2>\beta r\geq \eta h(p)$, we have
		\begin{equation*}
		\mu(N_{\beta r}(\partial J_p))\leq \mu(J_p\cup N_{\beta r}(\partial J_p))\leq \mu (B(p,\eta h(p)+\beta r))\leq \mu (B(p,2\beta r))\ll (2\beta r)^{2\delta-k}h(p)^{k-\delta}. 
		\end{equation*}
		Meanwhile, we have 
		\begin{equation*}
		\mu(J_p\cup N_r (\partial J_p))\geq \mu (B(p,r))\gg r ^{2\delta-k}h(p)^{k-\delta}. 
		\end{equation*}
		Hence 
		\begin{equation*}
		\frac{\mu(N_r(\partial J_p)-N_{\beta r}(\partial J_p))}{\mu (N_{\beta r}(\partial J_p))}=\frac{\mu (J_p\cup N_r(\partial J_p))-\mu (J_p\cup N_{\beta r}(\partial J_p))}{\mu (N_{\beta r}(\partial J_p))} \gg \frac{r ^{2\delta-k}h(p)^{k-\delta}-(2\beta r)^{2\delta-k}h(p)^{k-\delta} }{(2\beta r)^{2\delta-k}h(p)^{k-\delta}}. 
		\end{equation*}
		Therefore
		\begin{equation*}
		\mu(N_{\beta r}(\partial J_p))\ll \beta^{2\delta-k} \mu(N_r(\partial J_p)).
		\end{equation*}
	\end{itemize}
	Now to prove (\ref{recdouble})
	we consider $\eta^{1/2} r$ and $r$. Then one of them belongs to $(0,\eta^2 h(p))\cup [ \eta^{3/2} h(p),\eta h(p)]\cup [\eta^{1/2} h(p),h(p)]$. Inequality (\ref{recdouble}) follows from the observation that 
	\begin{equation*}
	\frac{\mu(N_{C\eta r}(\partial J_p))}{\mu(N_r(\partial J_p))}
	\leq \min\left\{\frac{\mu(N_{C\eta r}(\partial J_p))}{\mu(N_{\beta r}(\partial J_p))}, \frac{\mu(N_{\beta r}(\partial J_p))}{\mu(N_r(\partial J_p))}\right \}. \qedhere
	\end{equation*}
\end{proof}

\bigskip
 \noindent 
	\it{Institut f\"ur Mathematik, Universit\"at Z\"urich, 8057 Z\"urich}  \\
	email: {\tt lijialun36@gmail.com} 
		
		\bigskip   
		
		\noindent 
		\it{Department of Mathematics, University of Chicago, Chicago, IL, 60637}  \\
			email: {\tt panwenyu08@gmail.com}

\end{document}